%% file: article.tex
\documentclass[a4paper, 11pt]{amsart}
\usepackage{amsthm}
\usepackage[]{amsmath}
\usepackage{amssymb}
\usepackage{enumerate}
\usepackage{tabularx}
\usepackage[]{color}
\usepackage[left=2.2cm, right=2.2cm, bottom=3cm, top=3cm]{geometry}
\usepackage[colorlinks]{hyperref}
\usepackage{tikz}
\usepackage{multirow}
\usepackage{subcaption}
\usepackage{algorithm}
\usepackage{algorithmic}
\usepackage{multirow}

\usetikzlibrary{decorations.pathreplacing, math}

\newtheorem{assumption}{Assumption}
\newtheorem{theorem}{Theorem}
\newtheorem{lemma}{Lemma}

\newtheorem{remark}{Remark}

\allowdisplaybreaks[4]

\title[Unfitted LSFEM for Elasticity Interface Problem]{The Least
Squares Finite Element Method for Elasticity Interface Problem on
Unfitted Mesh}

\author[F.-Y. Yang]{Fanyi Yang} \address{School of Mathematical
  Sciences, Peking University, Beijing 100871, P.R. China}
\email{yangfanyi@pku.edu.cn}

\begin{document}

\input{def.tex}

\input{abstract.tex}

\maketitle

\input{introduction.tex}

\input{preliminaries.tex}

\input{lsmethod.tex}

\input{numericalresults.tex}

\input{conclusion.tex}

\bibliographystyle{amsplain}
\bibliography{../ref}

\end{document}

%% file: def.tex

\newcommand{\bm}[1]{\boldsymbol{#1}}
\newcommand{\mb}[1]{\mathbb{#1}}
\newcommand{\mr}[1]{\mathrm{#1}}
\newcommand{\bmr}[1]{\bm{\mr{#1}}}
\newcommand{\mc}[1]{\mathcal{#1}}
\newcommand{\wt}[1]{\widetilde{#1}}
\newcommand{\wh}[1]{\widehat{#1}}
\newcommand{\id}[1]{\mr{d} \boldsymbol{#1}}

\newcommand{\jump}[1]{[\hspace{-2pt}[ #1 ]\hspace{-2pt}]}
\newcommand{\jumpn}[1]{[\hspace{-2pt}[ #1 ]\hspace{-2pt}]_N}

\newcommand{\dx}[1]{{\mr{d}}\bm{#1} }

\newcommand{\tr}[1]{\mr{tr}(#1)}

\newcommand{\enorm}[1]{\| #1 \|_{\bm{\mr{e}}}}
\newcommand{\ehnorm}[1]{\| #1 \|_{\bm{\mr{e}}_h}}
\newcommand{\wtehnorm}[1]{\| #1 \|_{\wt{\bm{\mr{e}}}_h}}

\newcommand{\whehnorm}[1]{\| #1 \|_{\wh{\bm{\mr{e}}}_h}}

\newcommand{\BDK}[1]{B_{\Delta(#1)}}

\newcommand{\shnorm}[1]{| #1 |_{s_h}}
\newcommand{\shinorm}[1]{| #1 |_{s_{h, i}}}
\newcommand{\shonorm}[1]{| #1 |_{s_{h, 0}}}
\newcommand{\shlnorm}[1]{| #1 |_{s_{h, 1}}}

\def\Shm{\mc{S}_h^m}
\def\Ghm{\mc{G}_h^m}
\def\Ghol{\mc{G}_{h, 0}^1}

\newcommand{\shsnorm}[1]{| #1 |_{\mc{S}_h^{m}}}
\newcommand{\shunorm}[1]{| #1 |_{\mc{G}_h^{m}}}

\newcommand{\shubonorm}[1]{| #1 |_{s_{h, 0}^1}}

\newcommand{\shirnorm}[1]{| #1 |_{s_{h, i}^r}}

\newcommand{\bhnorm}[1]{| #1 |_{b_h}}

\newcommand{\red}[1]{{\color{red}#1}}

\def\BDM{\bm{\mr{BDM}}}
\def\RT{\bm{\mr{RT}}}

\def\un{\mr{\bm{n}}}

\def\div{\mr{div}}
\def\curl{\mr{curl}}

\def\mA{\mc{A}}
\def\mJ{\mc{J}}
\def\mG{\mc{G}}
\def\bV{\bmr{V}}
\def\bH{\bmr{H}}
\def\bL{\bmr{L}}
\def\bSi{\bmr{\Sigma}}
\def\btau{\bm{\tau}}
\def\bv{\bm{v}}
\def\brho{\bm{\rho}}
\def\bw{\bm{w}}
\def\bsigma{\bm{\sigma}}
\def\bu{\bm{u}}
\def\bx{\bm{x}}
\def\bvphi{\bm{\varphi}}

\def\ba{\bm{a}}
\def\bb{\bm{b}}

\def\beps{\bm{\varepsilon}}

\def\bt{\bm{t}}
\def\bq{\bm{q}}
\def\bz{\bm{z}}

\def\Bs{B^{\bm{\sigma}}}
\def\Bu{B^{\bm{u}}}

\def\be{\bm{\varepsilon}}

\def\MTh{\mc{T}_h}

\def\MThi{\mc{T}_{h, i}}
\def\MTho{\mc{T}_{h, 0}}
\def\MThl{\mc{T}_{h, 1}}

\def\MThic{\mc{T}_{h, i}^{\circ}}
\def\MThoc{\mc{T}_{h, 0}^{\circ}}
\def\MThlc{\mc{T}_{h, 1}^{\circ}}

\def\MThG{\mc{T}_h^{\Gamma}}

\def\Ohi{\Omega_{h, i}}
\def\Oho{\Omega_{h, 0}}
\def\Ohl{\Omega_{h, 1}}

\def\OhG{\Omega_h^{\Gamma}}

\def\Ohic{\Omega_{h, i}^{\circ}}
\def\Ohoc{\Omega_{h, 0}^{\circ}}
\def\Ohlc{\Omega_{h, 1}^{\circ}}

\def\MEh{\mc{E}_h}
\def\MEhI{\mc{E}_h^{I}}
\def\MEhB{\mc{E}_h^B}

%% file: abstract.tex

\begin{abstract}
  In this paper, we propose and analyze the least squares finite
  element methods for the linear elasticity interface problem 
  in the stress-displacement system on unfitted meshes. 
  We consider the cases that the interface is $C^2$ or polygonal, and
  the exact solution $(\bsigma, \bu)$ belongs to $H^s(\div; \Omega_0
  \cup \Omega_1) \times H^{1+s}(\Omega_0 \cup \Omega_1)$ with $s >
  1/2$. Two types of least squares functionals are defined to seek the
  numerical solution. 
  The first is defined by simply applying the
  $L^2$ norm least squares principle, and requires the condition
  $s \geq 1$.
  The second is defined with a discrete minus norm, which is related
  to the inner product in $H^{-1/2}(\Gamma)$. The use of this discrete
  minus norm results in a method of optimal convergence rates and
  allows the exact solution has the regularity of any $s > 1/2$.  The
  stability near the interface for both methods is guaranteed by the
  ghost penalty bilinear forms and we can derive the robust condition
  number estimates.  The convergence rates under $L^2$ norm and the
  energy norm are derived for both methods.  We illustrate the
  accuracy and the robustness of the proposed methods by a series of
  numerical experiments for test problems in two and three dimensions.

  \noindent \textbf{keywords}: linear elasticity interface problem;
  extended finite element space; discrete minus norm; least squares
  finite element method; unfitted mesh
\end{abstract}

%% file: introduction.tex
\section{Introduction}
\label{sec_introduction}
In this paper, we develop the least squares finite element methods
(LSFEMs) for linear elasticity interface problems, which model the
elasticity structure with different or even singular material
properties, and have many applications in fields of materials science
and continuum mechanics \cite{Gibiansky2000multiphase,
Leo2000microstructural, Gao2001continuum, Becker2009nitsche,
Almqvist2011interfacial}. For such problems, the governing equations
usually have discontinuous coefficients and involve the inhomogeneous
jump conditions. 
Because of the discontinuity near the interface and the irregular
geometry of the interface, it is still challenging to design efficient
numerical methods for such equations.

The finite element method is an important numerical method for 
solving interface problems. In the last decades, various numerical
schemes have been developed for the elliptic interface problem, and we
refer to \cite{Li1998immersed, Hansbo2002unfittedFEM,
Chen1998interface, Burman2015cutfem, Bordas2017geometrically,
Huang2017unfitted, Wu2012unfitted, Burman2021unfitted} for some
typical methods.  The finite element methods can be roughly classified
into fitted and unfitted methods based on types of grids.  The
body-fitted method requires the mesh to align with the interface for
representing the geometry of the interface accurately. For complex
geometries, it is a challenging and time-consuming task to generate a
high quality body-fitted mesh especially in high dimensions
\cite{Huang2017unfitted}. In the unfitted method, the interface
description is decoupled from the generation of the mesh, which
provides a good flexibility when handling the problem with complex
geometries.  Examples of such methods are the cut finite element
methods \cite{Hansbo2002unfittedFEM, Burman2015cutfem,
Massing2019stabilized, Bordas2017geometrically, Burman2021unfitted,
Zhang2022high},
the immersed finite element method \cite{Li1998immersed,
Li2006immersed, Guo2020error} and the aggregated finite element
method \cite{Badia2018aggregated, Johansson2013high, Chen2021an}. 

Recently, the unfitted finite element methods are also been applied to
solve the linear elasticity interface problem. 
In \cite{Hansbo2004finite}, Hansbo and Hansbo proposed a linear finite
element method and derived the optimal convergence rates in error
measurements under the assumption that the exact solution $\bu$ is
piecewise $H^2$. In \cite{Becker2009nitsche}, Becker et al. developed
a mixed finite element method with the linear accuracy for the
displacement-pressure formulation. The inf-sup condition and the
optimal error estimates are verified. Both methods are also called
Nitsche extended finite element methods (Nitsche-XFEM), where the jump
conditions are weakly imposed by the Nitsche penalty method
\cite{Hansbo2002unfittedFEM}.
In \cite{Zhang2022high}, Zhang followed the interface-penalty idea and
presented a high-order unfitted method for the elasticity interface
problem. Combining with the penalty method and the hybridizable
discontinuous Galerkin approximation, Han et al. developed a 
X-HDG method for this problem. Another type of unfitted finite element
methods is the immersed finite element method \cite{Li1998immersed},
which mainly modifies the basis functions near the interface to capture
the jump of the solution. In \cite{Lin2013locking,
Lin2019nonconforming}, the authors presented the nonconforming
immersed finite element methods for the linear elasticity interface
problem. Other types of immersed finite element methods can be found
in \cite{Guo2020error, Kwak2017stablized}.
We note that all above mentioned methods derive the error estimates
under the assumption that the exact solution $\bu$ has at least
piecewise $H^2$ regularity.
Many analysis techniques that are
developed for piecewise $H^2$ solutions in the elliptic interface
problem can be used in this case. But
these techniques may be unavailable for the solution that has only
piecewise $H^{1+s}(s < 1)$ regularity. In addition, 
for piecewise $H^1(\div)$ or piecewise $H^1(\curl)$ functions, 
applying these
techniques to estimate the errors will result in a suboptimal
convergence rate, see \cite{Liu2020interface, Li2023curl} for unfitted
methods in solving $H(\curl)$- and $H(\div)$-interface problems.
To our best knowledge, there are few works on the
interface problem of low regularity. In \cite{Guo2020interface}, the
authors present an immersed finite element method for
$H(\curl)$-interface problem with the optimal convergence rates. Only
piecewise $H^1(\curl)$ regularity of the exact solution is required in
the analysis.

In this paper, we develop least squares finite element methods
on unfitted meshes for the linear elasticity interface
problem, based on the stress-displacement formulation. 
For traditional
linear elasticity problems, the LSFEMs have been investigated in
\cite{Cai2003first, Cai2004least, Cai1998first, Bramble2001least,
Starke2011analysis, Li2019least}. The LSFEM can offer the advantage of
circumventing the inf-sup condition arising in mixed methods and
ensure the resulting linear system is always symmetric positive
definite.
As the standard LSFEM, we define the least
squares functionals and seek the numerical solution by minimizing the
functional over finite element spaces. Two types of least squares 
functionals are used in this paper. The first is defined by simply applying
the $L^2$ norm least squares principle to the stress-displacement.
Th defined functional only involves the $L^2$ norms and the jump conditions
are also enforced in the sense of $L^2$ norms. This method requires
the exact solution $(\bsigma, \bu)$ has the regularity $H^s(\div;
\Omega_0 \cup \Omega_1) \times H^{1+s}(\Omega_0 \cup \Omega_1)$ with
$s \geq 1$. The convergence rate in the $H(\div)
\times H^1$ norm is shown to be half order lower than the optimal
rate.
From the embedding theory, we know that any
$\btau \in H(\div; \Omega_0 \cup \Omega_1)$ has the normal trace $\un
\cdot \btau \in H^{-1/2}(\Gamma)$, but the stronger $L^2$ norm is
applied to handle the normal trace in this method. That is the reason
that the condition $s \geq 1$ is required and the convergence rate is
lower than the optimal value. To overcome this difficulty, 
we define another least squares functional 
with a discrete minus inner product.  This method follows from the
ideas in \cite{Bramble2001least, Bramble1997least}, where the
discrete minus norms corresponding to $H^{-1}(\Omega)$ are used. In
this paper, we define a discrete minus inner product that is related
to the inner product in the space $H^{-1/2}(\Gamma)$.
The use of this inner product allows us to relax the regularity
condition as $s > 1/2$, and gives the optimal convergence rates under
the error measurement with respect to the required regularity. 
It is noticeable that the optimal convergence rate under the $H(\div)$
norm is achieved for functions in $H^s(\div; \Omega_0 \cup \Omega_1)$.
We also point out that in the unfitted methods, the $H^1$ trace
estimate is usually the main tool to estimate the numerical error on
the interface. For the low regularity case $s < 1$, this estimate is
unavailable, and we use the embedding theory in the error estimation
instead.

Another important issue for unfitted methods is the presence of small
cuts near the interface. In our method, we employ the ghost penalty
method \cite{Burman2010ghost} to cure the effects bringing by small
cuts. The ghost penalty bilinear forms also correspond to $L^2$ norms,
and they can be added in the least squares functional without any
difficulty. We can prove a uniform upper bound of the condition number
to the resulting linear system with proper penalty forms. We give a
suitable penalty bilinear form with the polynomial local extension,
and the standard penalty forms given in \cite{Burman2010ghost,
Massing2019stabilized} can also be used in our methods. 

The rest of this article is organized as follows. 
In Section \ref{sec_preliminaries}, we introduce the basic notation
and define the ghost penalty bilinear forms. In Section
\ref{sec_problem}, we introduce the stress-displacement formulation
for the linear elasticity interface problem 
and the associated least squares functional. Section \ref{sec_method}
develops the numerical schemes. The least squares finite element
methods with $L^2$ norms and with the discrete minus norm are
established in Subsection \ref{subsec_L2norm} and Subsection
\ref{subsec_mnorm}, respectively, and the error estimates are also
included.  Finally, numerical results of test problems of linear
elasticity interface problems in two and three dimensions are presented
in Section \ref{sec_numericalresults}.

%% file: preliminaries.tex
\section{Preliminaries} 
\label{sec_preliminaries}
Let $\Omega \subset \mb{R}^d(d = 2, 3)$ be a convex polygonal
(polyhedral) domain with the boundary $\partial \Omega$.  Let
$\Omega_0 \Subset \Omega$ be a polygonal (polyhedral) subdomain or a
subdomain with the $C^2$-smooth boundary. We denote by $\Gamma :=
\partial \Omega_0$ the topological boundary, which can be regarded as
an interface dividing $\Omega$ into two disjoint domains $\Omega_0$
and $\Omega_1$, where $\Omega_1: = \Omega \backslash
\overline{\Omega}_0$, $\Omega_0 \cap \Omega_1 = \varnothing$ and
$\overline{\Omega}_0 \cup \overline{\Omega}_1 = \overline{\Omega}$.
We denote by $\MTh$ a quasi-uniform partition of $\Omega$ into
triangle (tetrahedron) elements. 
The mesh $\MTh$ is unfitted that the element faces in $\MTh$ are not
required to be aligned with the interface $\Gamma$.  
For any element $K \in \MTh$,  we denote by $h_K$ its diameter and by
$\rho_K$ the radius of the largest disk (ball) inscribed in $K$. 
Let $h:= \max_{K \in \MTh} h_K$ be the mesh size, and let $\rho :=
\min_{K \in \MTh} \rho_K$. 
The mesh $\MTh$ is quasi-uniform in the sense that there exists a
constant $C_\nu$ independent of $h$ such that $h \leq C_\nu \rho$.

We further introduce the notations related to subdomains $\Omega_0$
and $\Omega_1$.  
For $i = 0, 1$, we define $\MThi := \{ K \in \MTh \ | \
K \cap \Omega_i \neq \varnothing \}$ as the minimal subset of $\MTh$
that entirely covers $\Omega_i$, and define $ \MThic := \{ K \in \MThi
\ | \ K \subset \Omega_i\}$ as the set of interior elements in the
domain $\Omega_i$.
We define $\MThG := \{ K \in \MTh \ | \ K \cap
\Gamma \neq \varnothing \}$ as the collection of all cut elements.
Their corresponding domains are defined as $\Ohi :=
\text{Int}(\bigcup_{K \in \MThi} \overline{K})$, $\Ohic :=
\text{Int}(\bigcup_{K \in \MThic} \overline{K})$, and $\OhG :=
\text{Int}(\bigcup_{K \in \MThG} \overline{K})$. Notice
that there holds
$\OhG = \Ohi \backslash \overline{\Omega}_{h, i}^{\circ}$.
For any element $K \in \MTh$, we let $K^i := K \cap \Omega_i$ and for
any cut element $K \in \MThG$, we let $\Gamma_K := K \cap \Gamma$.

For any element $K \in \MTh$, we define $\Delta(K) := \{ K' \in \MTh \
| \  \overline{K'} \cap \overline{K} \neq \varnothing \}$ as the
set of elements touching $K$. Any element in $\Delta(K)$ at
least shares one vertex with $K$.
We denote by $B(\bm{z}, r)$ the disk (ball) centered at the point
$\bm{z}$ with the radius $r$. Since $\MTh$ is quasi-uniform, there
exists a generic constant $C_\Delta$ such that $\bigcup_{K' \in
\Delta(K)} \overline{K'} \subset B(\bm{x}_K, C_\Delta h)$ for $\forall
K \in \MTh$,  where $\bm{x}_K$ is the barycenter of $K$. 
We set $\BDK{K}:= B(\bm{x}_K, C_\Delta h)$ for $\forall K \in \MTh$. 

Throughout this paper, $C$ and $C$ with subscripts are denoted to be
generic positive constants that may vary in different lines, but are
always independent of the mesh size $h$, how the interface $\Gamma$
cuts the mesh $\MTh$, and the Lam\'e parameter $\lambda$ defined in
\eqref{eq_lame}. 

We make the following geometrical assumptions on the mesh: 
\begin{assumption}
  For any cut element $K \in \MThG$, the interface $\Gamma$ intersects
  each face of $K$ at most once.
  \label{as_mesh1}
\end{assumption}
\begin{assumption}
  For any cut element $K \in \MThG$, the sets $\Delta(K) \cap \MThoc$
  and $\Delta(K) \cap \MThlc$ are not empty.  
  \label{as_mesh2}
\end{assumption}

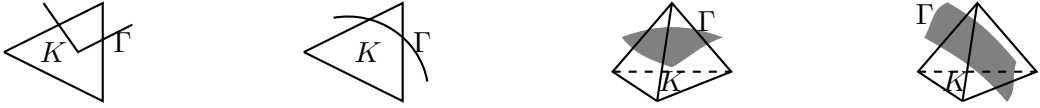
\begin{figure}
  \centering
 \begin{minipage}[t]{0.23\textwidth}
    \begin{center}
      \begin{tikzpicture}[scale=1.3]
        \draw[thick, black] (-1, 0) -- (0, 0.5) -- (0, -0.5) --
        (-1, 0);
        \draw[thick, black] (-0.6, 0.5) -- (-0.25, 0) -- (0.3, 0.28);
        \node at (0.2, 0.1) {$\Gamma$};
        \node at (-0.5, 0) {$K$};
      \end{tikzpicture}
    \end{center}
  \end{minipage}
  \begin{minipage}[t]{0.23\textwidth}
    \begin{center}
      \begin{tikzpicture}[scale=1.3]
        \draw[thick, black] (-1, 0) -- (0, 0.5) -- (0, -0.5) --
        (-1, 0);
        \draw[thick, black] (-0.7, 0.35) to [out=10, in=100] (0.25,
        -0.3);
        \node at (0.2, 0.1) {$\Gamma$};
        \node at (-0.35, 0) {$K$};
      \end{tikzpicture}
    \end{center}
  \end{minipage}
  \begin{minipage}[t]{0.23\textwidth}
    \begin{center}
      \begin{tikzpicture}[scale=1.3]
        \draw[fill, gray] (-0.5, 0.15) to [out=-52, in = 160] (-0,
        -0.15) to [out=30, in = 200] (0.5, 0.15)
        to [out=160, in=20] (-0.5, 0.15);
        \draw[thick, black] (-0.15, -0.5) -- (-0.6, -0.2) -- (0, 0.5) --
        (0.6, -0.2) -- (-0.15, -0.5);
        \draw[thick, black] (0, 0.5) -- (-0.15, -0.5);
        \draw[thick, dashed, black] (-0.6, -0.2) -- (0.6, -0.2);
        \node at (0.35, 0.32) {$\Gamma$};
        \node at (0, -0.3) {$K$};
      \end{tikzpicture}
    \end{center}
  \end{minipage}
  \begin{minipage}[t]{0.23\textwidth}
    \begin{center}
      \begin{tikzpicture}[scale=1.3]
        \draw[fill, gray] (-0.3, 0.5) to [out=200, in=60] (-0.53,
        0.15) to [out=-30, in = 130] (0.3, -0.5) to [out = 50, in =
          250]
        (0.39, -0.1) to [out=130, in = -30] (-0.3, 0.5);
        \draw[thick, black] (-0.15, -0.5) -- (-0.6, -0.2) -- (0, 0.5) --
        (0.6, -0.2) -- (-0.15, -0.5);
        \draw[thick, black] (0, 0.5) -- (-0.15, -0.5);
        \draw[thick, dashed, black] (-0.6, -0.2) -- (0.6, -0.2);
        \node at (-0.53, 0.39) {$\Gamma$};
        \node at (-0.22, -0.3) {$K$};
      \end{tikzpicture}
    \end{center}
  \end{minipage}
  \caption{Examples of cut elements in two dimensions (left) / in
  three dimensions (right).}
  \label{fig_cutelements}
\end{figure}
Assumption \ref{as_mesh1} - Assumption \ref{as_mesh2} ensure the
interface is well-resolved by the mesh $\MTh$, which are widely used
in unfitted finite element methods \cite{Hansbo2002unfittedFEM,
Wu2012unfitted, Burman2015cutfem}. 
Some examples of cut elements are shown in Fig.~\ref{fig_cutelements}. 
From Assumption \ref{as_mesh2}, we can assign two interior elements
$K_0^{\mr{int}}  \in \Delta(K) \cap \MThoc$ and $K_1^{\mr{int}} \in
\Delta(K) \cap \MThlc$ for any cut element $K \in \MThG$. 
In principle, $K_i^{\mr{int}}(i = 0, 1)$ can be anyone in $\Delta(K)
\cap \MThic$.  In practice, one can select $K_i^{\mr{int}}$ to share a
common face with $K$ whenever possible. Consequently, Assumption
\ref{as_mesh2} allows us to define two maps $M^i(\cdot)(i = 0, 1):
\MThG \rightarrow \MThic$ that $M^i(K) = K_i^{\mr{int}}$ for $\forall
K \in \MThG$.

Let us introduce the notation of trace operators on the interface.
Let $\bv$ be the vector- or tensor-valued function, we define the jump
operators $\jump{\cdot}$ and $\jumpn{\cdot}$ as 
\begin{displaymath}
  \jump{\bv}|_\Gamma := \bv^0|_{\Gamma} - \bv^1|_{\Gamma}, \quad
  \jumpn{\bv}|_{\Gamma} := \un_{\Gamma} \cdot ( \bv^0|_{\Gamma} -
  \bv^1|_{\Gamma}), 
\end{displaymath}
where $\bv^0 := \bv|_{\Omega_0}, \bv^1 := \bv|_{\Omega_1}$, and
$\un_{\Gamma}$ denotes the unit outward normal vector pointing to
$\Omega_1$ on $\Gamma$.

For an open bounded domain $D$, we let $H^r(D)$ denote the usual
Sobolev spaces with the exponent $r \geq 0$, and we follow their
corresponding inner products, seminorms and norms.  We define
$\bH^r(D) := (H^r(D))^d$ and $\mb{H}^r(D) := (H^r(D))^{d \times d}$ as
the Sobolev spaces of vector and tensor fields, respectively.
We let $L^2(D)$ coincide with $H^r(D)$ for $r = 0$. Further,
we introduce $H^r(\div; D) := \{ \bv \in
\bH^r(D) \ | \ \nabla \cdot \bv \in H^r(D) \}$  with the norm $\| \bv
\|_{H^r(\div; D)}^2 := \| \bv \|_{H^r(D)}^2 + \| \nabla \cdot \bv
\|_{H^r(D)}^2$, and let $\bH^r(\div; D) := (H^r(\div; D))^d$ be
the spaces of tensor fields. Each column of functions in
$\bH^r(\div; D)$ belongs to $H^r(\div; D)$. 
Let $H^{-1/2}(\partial D)$ be the dual space of
$H^{1/2}(\partial D)$ with the norm 
\begin{equation}
  \| v \|_{H^{-1/2}(\partial D)} := \sup_{0 \neq \varphi \in
  H^{1/2}(\partial D)} \frac{(v, \varphi)_{L^2(\partial D)}}{ \|
  \varphi \|_{H^{1/2}(\partial D)}}, \quad \forall v \in
  H^{-1/2}(\partial D).
  \label{eq_Hm12norm}
\end{equation}
From the trace theory, we know that any function $\bv \in H(\div; D)$
has a normal trace $\un \cdot \bv \in H^{-1/2}(\partial D)$ with that
$\| \un \cdot \bv \|_{H^{-1/2}(\partial D)} \leq C \| \bv \|_{H(\div;
D)}$. For vector fields, we let $\bH^{-1/2}(\partial D)$ be the dual
space of $\bH^{1/2}(\partial D)$, and the corresponding norm is
the same as \eqref{eq_Hm12norm} by replacing $\varphi \in
H^{1/2}(\partial D)$ with $\bm{\varphi} \in \bH^{1/2}(\partial D)$.

To cure the effect bringing by small cuts near the interface, we
follow the idea in the ghost penalty method \cite{Burman2010ghost,
Massing2019stabilized}, which uses the data from the interior domain
to ensure the stability near the interface. For this goal, 
we assume that we can construct two bilinear forms $s_{h, 0}^r(\cdot,
\cdot)$ and $s_{h, 1}^r(\cdot, \cdot)$ which are defined for functions
in $L^2(\Oho)$ and $L^2(\Ohl)$, respectively. 
The induced seminorms $\shirnorm{\cdot}(i = 0, 1)$ are given as $
\shirnorm{v}^2 := s_{h, i}^r(v, v)(i = 0, 1)$ for $\forall v \in
L^2(\Ohi)$. In our method, we assume that the forms  satisfy the
following two properties:
\begin{enumerate}
  \item[\textbf{P1}:] the $L^2$ norm extension property: 
    \begin{equation}
      \begin{aligned}
        \| v \|_{L^2(\Ohi)} \leq C( \| v\|_{L^2(\Omega_i)} +
        \shirnorm{v}) \leq C \| v \|_{L^2(\Ohi)}, \quad \forall v \in
        L^2(\Ohi), \quad i = 0, 1.
      \end{aligned}
      \label{eq_shiL2extension}
    \end{equation}
  \item[\textbf{P2}:] the weak consistency: 
    \begin{equation}
      \shirnorm{v} \leq C h^t \| v \|_{H^{s+1}(\Omega)}, \quad \forall
      v \in H^{s+1}(\Omega), \quad i = 0, 1, \quad t = \min(s + 1, r +
      1).
      \label{eq_weakconsist}
    \end{equation}
\end{enumerate}
The suitable penalty forms $s_{h, i}^r(\cdot,
\cdot)$ can be constructed by the face-based penalties and the
projection-based penalties, see \cite{Massing2019stabilized,
Burman2015cutfem}. The penalty forms constructed 
in \cite[Section 2.7]{Massing2019stabilized} satisfy \textbf{P1} -
\textbf{P2}. We also note that in the standard ghost penalty method
for elliptic interface problems,
extra properties of the penalty form are required, as the $H^1$
seminorm extension property and the inverse estimate, see \cite[EP1 -
EP4]{Massing2019stabilized}. In our method, \textbf{P1} and
\textbf{P2} are enough to ensure the stability near the interface.

Here, we outline a method to construct the penalty bilinear forms by
the local polynomial extension. The implementation is quite simple.
The idea of the local
extension has also widely used in unfitted methods, see
\cite{Badia2018aggregated,Burman2021unfitted, Chen2021an,
Huang2017unfitted, Johansson2013high, Yang2022an}. 
For any element $K \in \MTh$, we define the local extension
operator $E_K^r(r \geq 0)$ that extends the function in $L^2(K)$ to
the ball $\BDK{K}$ by  
\begin{equation}
  \begin{aligned}
    E_K^r:  L^2(K) &\rightarrow \mb{P}_r(\BDK{K}), \\
     v &\rightarrow E_K^r v,
   \end{aligned} \quad \text{$E_K^r v$ has the same expression as
   $\Pi_K^r v$, \ i.e. $(E_K^r v)|_K = \Pi_K^r v$,}
  \label{eq_EK}
\end{equation}
where $\Pi_K^r$ is the $L^2$ projection operator from $L^2(K)$ to
$\mb{P}_r(K)$. Since $K \in \BDK{K}$, for any $v \in L^2(K)$, $E_K^r
v$ is the direct extension of the $L^2$ projection
$\Pi_K^r v$ from $K$ to the ball $\BDK{K}$. 
Particularly, for any polynomial $v \in \mb{P}_r(K)$, $E_K^r v$ is the
direct polynomial extension of $v$ to the ball $\BDK{K}$. 
From the definition \eqref{eq_EK}, we can prove the following basic
property of $E_K^r$, 
\begin{equation}
  \| E_K^r v \|_{L^2(\BDK{K})} \leq C \| \Pi_K^r v \|_{L^2(K)} \leq C
  \| v \|_{L^2(K)}, \quad \forall v \in L^2(K), \quad \forall K \in
  \MTh. 
  \label{eq_EKL2}
\end{equation}
The norm equivalence on the finite dimensional space gives us that $\|
v \|_{L^2(B(\bm{0}, C_\Delta C_{\nu}))} \leq C \| v
\|_{L^2(B(\bm{0},1))}$ for $\forall v \in \mb{P}_r(B(\bm{0}, C_\Delta
C_\nu))$. Considering the affine mapping from $B(\bm{0}, 1)$ to
$B(\bm{x}_K, \rho)$, we derive that
\begin{equation}
  \begin{aligned}
    \| E_K^r v \|_{L^2(\BDK{K})} & \leq C \| E_K^r v
    \|_{L^2(B(\bm{x}_K, \rho))} =  C \| \Pi^r_K v \|_{L^2(B(\bm{x}_K,
    \rho))} \leq C \| \Pi^r_K v \|_{L^2(K)}, \ \forall v \in
    L^2(K),
  \end{aligned}
  \label{eq_EKL22}
\end{equation}
which leads to the stability property \eqref{eq_EKL2}.

For any $v \in L^2(\Ohi)(i = 0, 1)$, there holds $v|_K \in L^2(K)$ 
and $E_K^r(v|_K) \in \mb{P}_r(\BDK{K})$ for $\forall K \in \MThi$.
From \eqref{eq_EK},  $E_K^r (v|_K)$ is the direct polynomial extension
of $\Pi_K^r (v|_K)$ from $K$ to the ball $\BDK{K}$.
Hereafter, we simply write
$E_K^r (v|_K)$ as $E_K^r v$ for $\forall v \in L^2(\Ohi)$. 
This notation is most frequently used for $v$ is a piecewise
polynomial function of degree $r$ on the mesh $\MThi$.
In this case, $E_K^r v \in \mb{P}_r(\BDK{K})$ is just the direct
polynomial extension of $v|_K$ from $K$ to $\BDK{K}$ for $\forall K
\in \MThi$.

Based on $E_K^r$, two bilinear forms $s_{h,0}^r(\cdot,
\cdot)$ and $s_{h, 1}^r(\cdot, \cdot)$ that satisfy the
conditions \textbf{P1} and \textbf{P2} can be constructed as
\begin{equation}
  s_{h, i}^r(v, w) := \sum_{K \in \MThG} \int_K (v - E_{M^i(K)}^r v)(w
  - E_{M^i(K)}^r w) \dx{x}, \quad \forall v, w \in L^2(\Ohi), \quad i
  = 0, 1.
  \label{eq_shi}
\end{equation}
Notice that any cut element $K \in \MThG$ has that $K \subset
\BDK{M^i(K)}$. Therefore, $E_{M^i(K)}^r v \in  \mb{P}_r(\BDK{M^i(K)})$
is well defined on $K$.  Then, we verify the properties \textbf{P1} -
\textbf{P2} for the forms \eqref{eq_shi}. By \eqref{eq_EKL2}, we know
that 
\begin{displaymath}
  \|E_{M^i(K)}^r v \|_{L^2(K)} \leq \|E_{M^i(K)}^r v
  \|_{L^2(\BDK{M^i(K)})} \leq C \| v \|_{L^2(M^i(K))}, \quad \forall v
  \in L^2(\Ohi), \quad \forall K \in \MThi.
\end{displaymath}
Combining with the triangle inequality, one can find that for $\forall
v \in L^2(\Ohi)$, 
\begin{displaymath}
  \begin{aligned}
    s_{h, i}^r(v, v) \leq \sum_{K \in \MThG} ( \|v \|_{L^2(K)}^2 +
    \|E_{M^i(K)}^r v \|_{L^2(K)}^2) \leq C  \sum_{K \in \MThG} ( \|v
    \|_{L^2(K)}^2 + \| v \|_{L^2(M^i(K))}^2) \leq C \|v
    \|_{L^2(\Ohi)}^2,
  \end{aligned}
\end{displaymath}
which gives the upper bound in \eqref{eq_shiL2extension}.
Again by
\eqref{eq_EKL2} and the triangle inequality, there holds
\begin{displaymath}
  \begin{aligned}
    \| v \|_{L^2(\OhG)}^2 &= \sum_{K \in \MThG} \| v \|_{L^2(K)}^2 
    \leq  C \sum_{K \in \MThG} ( \| v - E_{M^i(K)}^r v \|_{L^2(K)}^2
    + \|E_{M^i(K)}^r v \|_{L^2(K)}^2) \\
    & \leq Cs_{h, i}^r(v, v) + C \sum_{K \in \MThG} \|v
    \|_{L^2(M^i(K))}^2 \leq C(\shirnorm{v}^2 + \|v
    \|_{L^2(\Ohic)}^2), \quad \forall v \in L^2(\Ohi).
  \end{aligned}
\end{displaymath}
The lower bound of \eqref{eq_shiL2extension} is reached and the
property \textbf{P1} holds for \eqref{eq_shi}. 
We turn to the weak consistency \textbf{P2}. 
Given any $v \in H^{s+1}(\Omega)$, and for 
any $K \in \MThG$, there exists $p_K \in
\mb{P}_r(\BDK{K})$ such that $ \| v - p_K \|_{L^2(\BDK{K})} \leq C
h^t \|v \|_{H^{s+1}(\BDK{K})}$\cite{Dupont1980polynomial}. From
\eqref{eq_EKL2}, we obtain that
\begin{displaymath}
  \begin{aligned}
    &\| v - E_{M^i(K)}^r v \|_{L^2(K)}  \leq \| v - p_K \|_{L^2(K)} +
    \| E_{M^i(K)}^r(p_K - v) \|_{L^2(K)} \\
    & \leq  \| v - p_K \|_{L^2(K)} + C \|p_K - v \|_{L^2(M^i(K))}
    \leq C \| v - p_K \|_{L^2(\BDK{K})} \leq  C h^t \|v
    \|_{H^{s+1}(\BDK{K})}.
  \end{aligned}
\end{displaymath}
Summation over all cut elements indicates the estimate
\eqref{eq_weakconsist}, i.e.  the property \textbf{P2} is reached.

Furthermore, it is natural to extend the bilinear forms $s_{h,
i}^r(\cdot, \cdot)(i = 0, 1)$ for vector- and
tensor-valued functions in a componentwise manner.
Let $\bv, \bw \in \bL^2(\Ohi)$ with $\bv = (v_j)_d, \bw = (w_j)_d$ be
the vector-valued functions, and let $\btau, \brho \in \mb{L}^2(\Ohi)$
with $\btau = (\tau_{jk})_{d \times d}, \brho = (\rho_{jk})_{d \times
d}$ be the tensor-valued functions, 
we define
\begin{displaymath}
  {s}_{h, i}^r(\bv, \bw) := \sum_{j = 0}^d s_{h, i}^r(v_j, w_j),
  \quad {s}_{h, i}^r(\btau, \brho) := \sum_{1 \leq j, k \leq d}
  s_{h, i}^r(\tau_{jk}, \rho_{jk}),
\end{displaymath}
with the induced seminorms $\shirnorm{\bv}^2 := s_{h, i}^r(\bv, \bv)$
and $\shirnorm{\btau}^2 := s_{h, i}^r(\btau, \btau)$. The properties
\textbf{P1} - \textbf{P2} can be extended for vector- and
tensor-valued functions without any difficulty.
We notice that int the computer implementation, the bilinear forms are
always with piecewise polynomial spaces. The operator $E_K^r$
is just the direct extension operator, and there is no need to
implement the $L^2$ projection in \eqref{eq_EK}.

We close this section by giving the $H^1$ trace estimate
\cite{Wu2012unfitted, Hansbo2002unfittedFEM, Guzman2018infsup} on the
interface. 
\begin{lemma}
  There exists a constant $C$ such that
  \begin{equation}
    \| w \|_{L^2(\Gamma_K)}^2 \leq C ( h_K \| w \|_{H^1(K)}^2 +
    h_K^{-1} \| w \|_{L^2(K)}^2), \quad \forall w \in H^1(K), \quad
    \forall K \in \MThG.
    \label{eq_H1trace}
  \end{equation}
  \label{le_H1trace}
\end{lemma}
We refer to \cite{Guzman2018infsup} for the proof only assuming
$\Gamma$ is Lipschitz. The $H^1$ trace estimate is fundamental in
the penalty-type unfitted finite element methods, as the main tool to
handle the numerical error on the interface, such as
\cite{Hansbo2002unfittedFEM, Wu2012unfitted, Badia2018aggregated,
Burman2015cutfem, Huang2017unfitted, Johansson2013high,
Massjung2012unfitted, Massing2019stabilized}. 
By \eqref{eq_H1trace}, the errors on the interface can be bounded 
by the estimates on elements.  However, this trace estimate
\eqref{eq_H1trace} requires the $H^1$ regularity. For the problem
with the exact solution in $H^s(\div;\Omega_0 \cup \Omega_1)$ or
$H^s(\curl; \Omega_0 \cup \Omega_1)$, applying \eqref{eq_H1trace} to
estimate the numerical errors on the interface will lead to a
suboptimal convergence rate. The optimal convergence needs a higher
regularity assumption that the exact solution is piecewise
$H^{s+1}$-smooth.  see \cite{Liu2020interface, Li2023curl} for
unfitted methods on $H(\div)$- and $H(\curl)$-interface problems.  In
addition, the trace estimate \eqref{eq_H1trace} is not suitable for
the solution of low regularity, such as the solution belongs to the
space $H^s(\Omega_0 \cup \Omega_1)$ with $s < 1$. 

\section{The Problem Setting and Least Squares Functional}
\label{sec_problem}
The model considered in this paper is the linear elasticity interface
problem defined on $\Omega$, which reads: seek the stress $\bm{\sigma}
= (\sigma_{ij})_{d \times d}$ and the displacement $\bu = (u_j)_d$
such that 
\begin{equation}
  \begin{aligned}
    \mA \bm{\sigma} - \be(\bm{u}) &= \bm{0}, && \text{in } \Omega_0 \cup
    \Omega_1, \\
    \nabla \cdot \bm{\sigma} + \bm{f} & = \bm{0}, && \text{in }
    \Omega_0 \cup \Omega_1, \\
    \bm{u} &= \bm{0}, && \text{on } \partial \Omega, \\
    \jumpn{\bm{\sigma}} = \bm{a}, \quad \jump{\bm{u}} &= \bm{b} , &&
    \text{on } \Gamma, \\
  \end{aligned}
  \label{eq_problem}
\end{equation}
where $\bm{f}$ is the source term, and $\bm{a}$, $\bm{b}$ are the jump
conditions on the interface. 
The Lam\'e parameters $\lambda$, $\mu$ are assumed to be piecewise
positive constant functions, 
\begin{equation}
  (\lambda(\bm{x}), \mu(\bm{x})) := \begin{cases}
    (\lambda_0, \mu_0), & \text{in } \Omega_0, \\
    (\lambda_1, \mu_1), & \text{in } \Omega_1, \\
  \end{cases} \quad \lambda_0, \lambda_1, \mu_0, \mu_1 > 0.
  \label{eq_lame}
\end{equation}
The constitutive law is expressed by the linear operator $\mA:
\mb{R}^{d \times d} \rightarrow \mb{R}^{d \times d}$: 
\begin{displaymath}
  \mA \bm{\tau} := \frac{1}{2\mu}\left( \bm{\tau} - \frac{\lambda}{d
  \lambda + 2 \mu} \tr{\bm{\tau}} \bmr{I} \right), \quad \forall
  \bm{\tau} \in \mb{R}^{d \times d},
\end{displaymath}
where $\tr{\cdot}$ denotes the trace operator, and $\bmr{I} :=
(\delta_{ij})_{d \times d}$ is the identity tensor.  The function
$\bm{\varepsilon}(\bm{u})$ denotes the symmetric strain tensor:
\begin{displaymath}
  \bm{\varepsilon}(\bm{u}) = (\varepsilon_{i, j}(\bm{u}))_{d \times
  d}, \quad \varepsilon_{i, j}(\bm{u}) 
  := \frac{1}{2}\left( \frac{\partial u_i}{\partial x_j} +
  \frac{\partial u_j}{\partial x_i} \right), \quad 1 \leq i, j \leq d.
\end{displaymath}
We assume that the interface problem \eqref{eq_problem} admits a
unique solution $(\bsigma, \bu) \in \bSi^s \times \bV^{s+1}$ with $s >
1/2$, where 
\begin{equation}
  \bSi^s := \bH^{s}(\div; \Omega_0 \cup \Omega_1), \quad 
  \bV^{s+1} := \{ \bv \in \bH^{s+1}(\Omega_0 \cup \Omega_1):  \
  \bv|_{\partial \Omega} = \bm{0} \}.
  \label{eq_bSibV}
\end{equation}
For $i = 0, 1$, we let $\bm{\sigma}_i := \bm{\sigma}|_{\Omega_i}$,
$\bm{u}_i := \bm{u}|_{\Omega_i}$.  
We assume that $\bm{\sigma}_i$ and
$\bm{u}_i$ can be extended to the whole domain $\Omega$ in the sense
that there exist $\wt{\bsigma}_i \in \bH^s(\div; \Omega)$ and
$\wt{\bu}_i \in \bH^{s+1}(\Omega)$ such that
$\wt{\bsigma}_i|_{\Omega_i} = \bsigma_i$ with $ \| \wt{\bsigma}_i
\|_{H^s(\div; \Omega)} \leq C \| \bsigma_i \|_{H^s(\div; \Omega_i)}$
and $\wt{\bu}_i|_{\Omega_i} = \bu_i$ with $ \| \wt{\bu}_i
\|_{H^{s+1}(\Omega)} \leq C \| \bu_i \|_{H^{s+1}(\Omega_i)}$.
Consequently, $\bsigma$ and $\bu$ can be decomposed as
\begin{equation}
  \bsigma = \wt{\bsigma}_0 \cdot \chi_0 + \wt{\bsigma}_1 \cdot \chi_1,
  \quad \bu = \wt{\bu}_0 \cdot \chi_0 + \wt{\bu}_1 \cdot \chi_1,
  \label{eq_soldecomp}
\end{equation}
where $\chi_i$ is the characteristic function corresponding to the
domain $\Omega_i$.  We refer to \cite{Hiptmair2010convergence,
Adams2003sobolev, Di2012hitchhiker} for more details about the
extension of Sobolev spaces.
From \eqref{eq_soldecomp}, we formally introduce two projection
operators $\pi_i(i = 0, 1)$ that $\pi_i \bsigma := \wt{\bsigma}_i \in
\bH^s(\div; \Omega)$ ad $\pi_i \bu := \wt{\bu}_i \in
\bH^{s+1}(\Omega)$.

We define an associated least squares fundamental to the interface
problem \eqref{eq_problem}.
Let $\bSi := \bSi^0$ and $\bV := \bV^1$ be the spaces coinciding with
$s = 0$ in \eqref{eq_bSibV}. We define the quadratic functional
$\mJ(\cdot; \cdot)$ by
\begin{equation}
  \mJ(\btau, \bv; \bm{f}, \ba, \bm{b}) := J(\btau, \bv; \bm{f}) +
  B^{\bsigma}(\btau; \ba) + B^{\bu}(\bv; \bb), \quad \forall (\btau,
  \bv) \in \bSi \times \bV,
  \label{eq_mJ}
\end{equation}
where 
\begin{equation}
  \begin{aligned}
    J(\btau, \bv; \bm{f}) &:= \| \mA \btau - \beps(\bv)
    \|_{L^2(\Omega_0 \cup \Omega_1)}^2 + \| \nabla \cdot \btau +
    \bm{f} \|_{L^2(\Omega_0 \cup \Omega_1)}^2, \\
    B^{\bm{\sigma}}(\btau; \ba) &:= \| \jumpn{\btau} -
    \ba \|_{H^{-1/2}(\Gamma)}^2, \quad B^{\bu}(\bv; \bb) := \|
    \jump{\bv} - \bm{b} \|_{H^{1/2}(\Gamma)}^2,
  \end{aligned} 
  \quad \forall (\bm{\sigma}, \bv) \in \bSi \times \bV.
  \label{eq_JBsBu}
\end{equation}
The trace terms $\Bs(\cdot; \cdot)$ and $\Bu(\cdot; \cdot)$ are
well-defined since  $\jumpn{\btau}|_{\Gamma} \in
H^{-1/2}(\Gamma)$ and $\jump{\bv}|_{\Gamma} \in H^{1/2}(\Gamma)$ for
$\forall (\btau, \bv) \in \bSi \times \bV$ from the embedding
theory.
The exact solution $(\bsigma, \bu)$ clearly minimizes the functional
$\mJ(\cdot; \cdot)$ since $\mJ(\bsigma, \bu; \bm{f}, \ba, \bb) = 0$.
In fact, we can prove
that $(\bsigma, \bu)$ is also the unique solution to the minimization
problem $\inf_{(\btau, \bv) \in \bSi \times \bV} \mJ(\btau, \bv;
\bm{f}, \ba, \bb)$. For this purpose, we will give a norm equivalence
property of $\mc{J}(\cdot, \cdot)$, and the norm equivalence is also
crucial for the error analysis in the least squares finite element
method \cite{Bochev1998review}. We notice that the pair of spaces
$\bSi \times \bV$ can be naturally equipped with the norm 
\begin{displaymath}
  \enorm{(\btau, \bv)}^2 :=  \| \btau \|_{H(\div; \Omega_0 \cup
  \Omega_1)}^2 + \| \bv \|_{H^1(\Omega_0 \cup \Omega_1)}^2, \quad
  \forall (\btau, \bv) \in \bSi \times \bV.
\end{displaymath}
The equivalence between $\mc{J}(\cdot; \cdot)$ and
$\enorm{\cdot}$ is given in the following lemma. 
\begin{lemma}
  There exist constants $C$ such that
  \begin{equation}
    \begin{aligned}
      \enorm{(\bm{\tau}, \bm{v})} \leq C \big( \| \mA \bm{\tau} -&
      \bm{\varepsilon}(\bm{v}) \|_{L^2(\Omega_0 \cup \Omega_1)} + \|
      \nabla \cdot \bm{\tau} \|_{L^2(\Omega_0 \cup \Omega_1)} \\
      + &\| \jumpn{\bm{\tau}} \|_{H^{-1/2}(\Gamma)} + \|
      \jump{\bm{v}} \|_{H^{1/2}(\Gamma)} \big) \leq C
      \enorm{(\bm{\tau}, \bm{v})},
    \end{aligned} \quad \forall (\bm{\tau}, \bm{v}) \in \bSi \times
    \bV.
    \label{eq_enormequvi}
  \end{equation}
  \label{le_enormequvi}
\end{lemma}
\begin{proof}
  By the definition of $\mA$, we have that 
  \begin{equation}
    \| \mA \bm{\tau} \|_{L^2(\Omega)}^2 = \frac{1}{(2\mu)^2} \left( \|
    \bm{\tau} \|_{L^2(\Omega)}^2 - \frac{\lambda( d \lambda + 4
    \mu)}{(d \lambda + 2 \mu)^2} \| \tr{\bm{\tau}} \|_{L^2(\Omega)}^2
    \right) \leq  \frac{1}{(2\mu)^2} \| \bm{\tau} \|_{L^2(\Omega)}^2
    \leq C  \| \bm{\tau} \|_{L^2(\Omega)}^2.
    \label{eq_Abound}
  \end{equation}
  The term $\| \mA \btau - \beps(\bv) \|_{L^2(\Omega_0 \cup
  \Omega_1)}$ can be bounded by $\enorm{(\btau, \bv)}$, following from
  the triangle inequality and \eqref{eq_Abound}. From the embedding
  theory, we know that
  \begin{displaymath}
    \| \jumpn{\btau} \|_{H^{-1/2}(\Gamma)} \leq \| \un \cdot
    \btau|_{\Omega_0} \|_{H^{-1/2}(\Gamma)} +  \| \un \cdot
    \btau|_{\Omega_1} \|_{H^{-1/2}(\Gamma)} \leq C \| \btau
    \|_{H(\div; \Omega_0 \cup \Omega_1)}.
  \end{displaymath}
  Similarly, there holds $\|\jump{\bv} \|_{H^{1/2}(\Gamma)} \leq C \|
  \bv \|_{H^1(\Omega_0 \cup \Omega_1)}$. The upper bound in
  \eqref{eq_enormequvi} is reached.

  From \cite[Theorem 3.1]{Cai2004least}, the lower bound in
  \eqref{eq_enormequvi} holds for $\btau \in \bH(\div; \Omega)$ and
  $\bv \in \bH^1(\Omega)$, i.e. the lower bound holds for $\btau$
  and $\bv$ satisfying $\jumpn{\btau}|_{\Gamma} = \jump{\bv}|_{\Gamma}
  = \bm{0}$. For the general case, we construct two auxiliary
  functions $\wt{\btau}$ and $\wt{\bv}$ to prove
  \eqref{eq_enormequvi}.  
  Consider the elliptic problems
  \begin{displaymath}
    -\Delta \bm{w}_0 = \bm{0}, \quad \text{in } \Omega_0, \quad
    \partial_{\un} \bm{w}_0 = -\jumpn{\bm{\tau}}|_{\Gamma}, \quad \text{on }
    \Gamma,
  \end{displaymath}
  and 
  \begin{displaymath}
    -\Delta \bm{w}_1 = \bm{0}, \quad \text{in } \Omega_0,
    \quad \bm{w}_1 = -\jump{\bm{v}}|_{\Gamma}, \quad \text{on } \Gamma.
  \end{displaymath}
  Since $\jumpn{\btau}|_{\Gamma} \in \bH^{-1/2}(\Gamma)$ and
  $\jump{\bv}|_{\Gamma} \in
  H^{1/2}(\Gamma)$, we can know that both problems have a unique
  solution in $\bH^1(\Omega_0)$ with $\| \bw_0 \|_{H^1(\Omega_0)} \leq
  C \| \jumpn{ \btau} \|_{H^{-1/2}(\Gamma)} $ and  $\| \bw_1
  \|_{H^1(\Omega_0)} \leq C \| \jump{ \bv} \|_{H^{1/2}(\Gamma)}$.  We
  then extend $\bm{w}_0$ and $\bm{w}_1$ to the domain $\Omega$ by
  zero. Let $\wt{\btau} := \btau + \nabla \bw_0$ and $\wt{\bv} := \bv
  + \bw_1$, we have that $\jumpn{\wt{\btau}}|_{\Gamma} =
  \jump{\wt{\bv}}|_{\Gamma} = \bm{0}$, which allows us to derive that
  \begin{align*}
    \enorm{(\bm{\tau}, \bm{v})} \leq  \enorm{(\wt{\bm{\tau}},
    \wt{\bm{v}})} + \enorm{(\bm{\tau} - \wt{\bm{\tau}}, \bm{v} -
    \wt{\bm{v}})} \leq \enorm{(\wt{\bm{\tau}}, \wt{\bm{v}})} + C ( \|
    \jumpn{\bm{\tau}} \|_{H^{-1/2}(\Gamma)} + \| \jump{\bm{v}}
    \|_{H^{1/2}(\Gamma)}),
  \end{align*}
  and 
  \begin{align*}
    &\enorm{(\wt{\bm{\tau}}, \wt{\bm{v}})}  \leq C ( \| \mA
    \wt{\bm{\tau}} - \bm{\varepsilon}(\wt{\bm{v}}) \|_{L^2(\Omega_0
    \cup \Omega_1)} + \| \nabla \cdot \wt{\bm{\tau}} \|_{L^2(\Omega_0
    \cup \Omega_1)}) \\
    & \leq C (  \| \mA {\bm{\tau}} - \bm{\varepsilon}({\bm{v}})
    \|_{L^2(\Omega_0 \cup \Omega_1)} + \| \nabla \cdot {\bm{\tau}}
    \|_{L^2(\Omega_0 \cup \Omega_1)}) + C (\| \bm{\tau} -
    \wt{\bm{\tau}} \|_{H(\div; \Omega_0 \cup \Omega_1)} + \| \bm{v} -
    \wt{\bm{v}} \|_{H^1(\Omega_0 \cup \Omega_1)}) \\
    & \leq C (  \| \mA {\bm{\tau}} - \bm{\varepsilon}({\bm{v}})
    \|_{L^2(\Omega_0 \cup \Omega_1)} + \| \nabla \cdot {\bm{\tau}}
    \|_{L^2(\Omega_0 \cup \Omega_1)} + \|  \jumpn{\bm{\tau}}
    \|_{H^{-1/2}(\Gamma)} + \| \jump{\bm{v}} \|_{H^{1/2}(\Gamma)}).
  \end{align*}
  Combining the above two estimates leads to the lower bound in
  \eqref{eq_enormequvi}, which completes the proof.
\end{proof}
It is noticeable that the generic constants in \eqref{eq_enormequvi}
are independent of $\lambda$, which ensure the proposed methods are
robust when $\lambda \rightarrow \infty$.

From the definition to $\mJ(\cdot; \cdot)$, it can be observed that
\begin{displaymath}
  \mJ(\btau, \bv; \bm{0}, \bm{0}, \bm{0}) = \| \mA \btau - \beps(\bv)
  \|_{L^2(\Omega_0 \cup \Omega_1)}^2 + \| \nabla \cdot \btau
  \|_{L^2(\Omega_0 \cup \Omega_1)}^2 + \| \jumpn{\btau}
  \|_{H^{-1/2}(\Gamma)}^2 + \| \jump{\bv} \|_{H^{1/2}(\Gamma)}^2.
\end{displaymath}
Therefore, $\mJ(\btau, \bv; \bm{0}, \bm{0}, \bm{0})$ is equivalent to
$\enorm{(\btau, \bv)}^2$.
Let $(\btau, \bv) \in \bSi \times
\bV$ be the solution to the problem \eqref{eq_problem} with $\bm{f} =
\ba = \bb = \bm{0}$, which implies that $\mJ(\btau, \bv;
\bm{0}, \bm{0}, \bm{0}) = 0$. From the equivalence
\eqref{eq_enormequvi}, we immediately find that $\btau$ and $\bv$ are
zero functions. 
Hence, the minimization problem $\inf_{(\btau, \bv) \in \bSi \times
\bV} \mJ(\btau, \bv; \bm{f}, \ba, \bb)$ admits a unique solution.

%% file: lsmethod.tex
\section{The Least Squares Finite Element Method}
\label{sec_method}
In this section, we present the numerical schemes for solving the
elasticity interface problem \eqref{eq_problem}.  We begin by
introducing the approximation finite element spaces.  
For the mesh $\MThi(i = 0, 1)$, we define $\bSi_{h, i}^m \subset
\bH(\div; \Ohi)$ as the $H(\div)$-conforming tensor-valued piecewise
polynomial space of degree $m$. As the definition of $\bH(\div;
\Ohi)$, each column of functions in $\bSi_{h, i}^m$ belongs to 
$H(\div; \Ohi)$, i.e. the $\BDM_m$ space or the $\RT_m$ space.
In this paper, the schemes are established $\bSi_{h, i}^m$ is the
$\BDM_m$ space, and the methods and the analysis can be extended to
the $\RT_m$ space without any difficulty. 
We define $\bV_{h,0}^m \subset \bH^1(\Oho)$ as the vector-valued $C^0$
finite element space of degree $m$, and define $\bV_{h,1}^m \subset
\bH^1(\Ohl)$ as the $C^0$ finite element space with zero trace on
$\partial \Omega$, i.e. $\bV_{h, 1}^m|_{\partial \Omega} = \bm{0}$.
We define the extended approximation spaces $\bSi_h^m :=
\bSi_{h, 0}^m \cdot \chi_0 + \bSi_{h, 1}^m \cdot \chi_1$ and $\bV_h^m
:= \bV_{h,0}^m \cdot \chi_0 + \bV_{h, 1}^m \cdot \chi_1$ for the
stress and the displacement, respectively, where the characteristic
function $\chi_i$ is defined in \eqref{eq_soldecomp}. Then, any
$\btau_h \in \bSi_h^m$ and any $\bv_h \in \bV_h^m$ admit a unique
decomposition that 
\begin{equation}
  \btau_h = \btau_{h, 0} \cdot \chi_0 + \btau_{h, 1} \cdot \chi_1,
  \quad  \bv_h = \bv_{h, 0} \cdot \chi_0 + \bv_{h, 1} \cdot \chi_1,
  \label{eq_sfemcomp}
\end{equation}
where $\btau_{h, i} \in \bSi_{h, i}^m$, $\bv_{h, i} \in \bV_{h, i}^m$.
As the decomposition \eqref{eq_soldecomp}, we
formally let $\pi_i(i = 0, 1)$ still be the projection operator such
that $\pi_i \btau_h := \btau_{h, i} \in \bSi_{h, i}^m(\forall \btau_h
\in \bSi_h^m)$ and $\pi_i
\bv_h := \bv_{h, i} \in \bV_{h, i}^m(\forall \bv_h \in \bV_h^m)$.

The approximation spaces are conforming in the sense
that $\bSi_h^m \subset \bSi$ and $\bV_h^m \subset \bV$. A natural idea
to seek numerical solutions is minimizing the functional $\mJ(\cdot;
\cdot)$ over the discrete conforming spaces, that is 
$\inf_{(\bm{\tau}_h, \bm{v}_h) \in \bSi_h^m \times \bV_h^m}
\mJ(\bm{\tau}_h, \bm{v}_h; \bm{f}, \bm{a}, \bm{b})$, which, however,
is not a easy task. 
It is impossible to directly compute the trace terms
$B^{\bsigma}(\cdot; \cdot)$ and $B^{\bu}(\cdot; \cdot)$ in
\eqref{eq_JBsBu}, and the minimization problem cannot be
readily rewritten into a variational problem by the means of
Euler-Lagrange equation because of the presence of $\| \cdot
\|_{H^{1/2}(\Gamma)}$ and $\| \cdot \|_{H^{-1/2}(\Gamma)}$.
The main idea of the proposed method is to apply some computational
trace terms to replace $B^{\bsigma}(\cdot; \cdot)$ and $B^{\bu}(\cdot;
\cdot)$ in $\mJ(\cdot; \cdot)$ to define new functionals. 
The numerical approximations are then obtained by minimizing the new
functionals.

For the convergence analysis, we define the spaces $\bSi_h := \bsigma
+ \bSi_h^m$ and $\bV_h := \bu + \bV_h^m$, which consist of the exact
solution and all finite element functions, respectively. 
Here, we introduce the ghost penalty forms for the spaces $\bSi_h$ and
$\bV_h$ from the forms \eqref{eq_shi}. 
By the definition of $\pi_i(i = 0, 1)$, there holds $\pi_i \btau_h
\in \mb{L}^2(\Ohi)(\forall \btau_h \in \bSi_h)$ and $\pi_i \bv_h \in
\bL^2(\Ohi)(\forall \bv_h \in \bV_h)$. 
From \eqref{eq_shi}, we define the forms $\mc{S}_h^m(\cdot,
\cdot)$ and $\mc{G}_h^m(\cdot, \cdot)$ for $\bSi_h$ and $\bV_h$,
respectively, as
\begin{displaymath}
  \begin{aligned}
    \Shm(&\btau_h, \brho_h)  := s_{h, 0}^{m}(\pi_0 \btau_h, \pi_0
    \brho_h) + s_{h, 1}^{m}(\pi_1 \btau_h, \pi_1 \brho_h), 
    \quad
    \forall \btau_h, \brho_h \in \bSi_h,
  \end{aligned}
\end{displaymath}
with the induced seminorm $\shsnorm{\btau_h}^2 := \Shm(\btau_h,
\btau_h)(\forall \btau_h \in \bSi_h)$, and 
\begin{displaymath}
  \begin{aligned}
    \Ghm(&\bv_h, \bw_h) := s_{h, 0}^{m}(\pi_0 \bv_h, \pi_0 \bw_h) +
    s_{h, 1}^{m}(\pi_1 \bv_h, \pi_1 \bw_h), \quad \forall \bv_h, \bw_h
    \in \bV_h,
  \end{aligned}
\end{displaymath}
with the induced seminorm $\shunorm{\bv_h}^2 := \Ghm(\bv_h,
\bv_h)(\forall \bv_h \in \bV_h)$. From \textbf{P1} - \textbf{P2}, 
we have the following estimates, 
\begin{align}
  \| \pi_0 \btau_h \|_{L^2(\Oho)} +  \| \pi_1 \btau_h \|_{L^2(\Ohl)}
  &\leq C ( \| \btau_h \|_{L^2(\Omega)} + \shsnorm{\btau_h}) 
  \label{eq_btL2bound} \\
  & \leq C(  \| \pi_0 \btau_h \|_{L^2(\Oho)} + \| \pi_1 \btau_h
  \|_{L^2(\Ohl)}), \quad \forall \btau_h \in \bSi_h^m, \nonumber \\
  \| \pi_0 \bv_h \|_{L^2(\Oho)} + \| \pi_1 \bv_h \|_{L^2(\Ohl)} 
  & \leq C ( \| \bv_h \|_{L^2(\Omega)} + \shunorm{\bv_h}) 
  \label{eq_buL2bound} \\
  & \leq C(  \| \pi_0 \bv_h \|_{L^2(\Oho)} + \| \pi_1 \bv_h
  \|_{L^2(\Ohl)}), \quad \forall \bv_h \in \bV_h^m, \nonumber 
\end{align}
and
\begin{displaymath}
  \begin{aligned}
    \shsnorm{\btau} & \leq Ch^t \| \btau \|_{H^{s}(\Omega)}, \quad
    \forall \btau \in
    \mb{H}^{s+1}(\Omega), \quad t = \min(m, s), \\
    \shunorm{\bv} & \leq Ch^t \| \bv \|_{H^{s+1}(\Omega)}, \quad
    \forall \bv \in
    \bH^{s+1}(\Omega), \quad t = \min(m + 1, s+1). \\
  \end{aligned}
\end{displaymath}

\subsection{The $L^2$ norm Least Squares Finite Element Method for $s
\geq 1$}
\label{subsec_L2norm}
We present the numerical method that defines the least squares
functional using only $L^2$ norms. 
This method requires the exact solution has the regularity $s
\geq 1$. As commented earlier, we will bound the trace terms
$B^{\bsigma}(\cdot; \cdot)$ and $B^{\bu}(\cdot; \cdot)$ by applying
stronger and easily-computational norms. 
In this case, the space $\mb{H}^s(\Omega_i)$ continuously embeds into
$\mb{L}^2(\Gamma)$, which indicates that
$\jumpn{\btau}|_{\Gamma} = \un \cdot ((\btau|_{\Omega_0} -
\btau|_{\Omega_1}))|_{\Gamma} \in \bL^2(\Gamma)$ for $\forall \btau
\in \bSi^s$.
From the definition \eqref{eq_Hm12norm}, one can find that
\begin{equation}
  \| \jumpn{\btau}\|_{H^{-1/2}(\Gamma)} \leq C \|\jumpn{ \btau}
  \|_{L^2(\Gamma)}, \quad \forall \btau \in \bSi^s.
  \label{eq_Hm12sbound}
\end{equation}
For the displacement variable, we also give a stronger norm for $\|
\cdot \|_{H^{1/2}(\Gamma)}$. 
We first consider the case that $\Gamma$ is $C^2$-smooth. 
By the embedding relationship $\bH^{s+1}(\Omega_i) \hookrightarrow
\bH^1(\Gamma)$, we know that $\jump{\bv}|_{\Gamma} \in \bH^1(\Gamma)$
for $\forall \bv \in \bV^{s+1}$. 
Here the norm $ \| \bv \|_{H^1(\Gamma)}^2 = \| \bv \|_{L^2(\Gamma)}^2
+ \| \nabla_{\Gamma} \bv \|_{L^2(\Gamma)}^2$,
where $\nabla_{\Gamma}$ denotes the tangential gradient on the
interface \cite{Dziuk2013finite}. 
The tangential gradient $\nabla_{\Gamma} \bv (\bx)$ for $\forall \bv \in
\bH^1(\Gamma)$ only depends on values of $\bv$ on $\Gamma \cap U$,
where $U$ is a neighbourhood of $\bx$. 
For $i = 0, 1$, we notice that the any finite element function
$\bv_{h, i} \in \bV_{h, i}^m(i = 0, 1)$ is continuous and piecewise
smooth on $\Gamma$, i.e.  $\bv_{h, i}|_{\Gamma} \in C^0(\Gamma)$ and
$\bv_{h, i}|_{\Gamma_K} \in C^2(\Gamma_K)(\forall K \in \MThG)$,
which gives that $\bv_{h,i} \in \bH^1(\Gamma)$. 
Thus, we have that
$\jump{\bv_h}|_{\Gamma} = (\pi_0\bv_{h})|_{\Gamma} - (\pi_1
\bv_{h})|_{\Gamma} \in \bH^1(\Gamma)$ for $\forall \bv_h \in
\bV_h^m$. Consequently, we conclude that 
$\jump{\bv_h}|_{\Gamma} \in \bH^1(\Gamma)$ for $\forall \bv_h \in
\bV_h$.
By the Sobolev interpolation inequality \cite[Theorem
7.4]{Lions1972nonhomogeneous} and the Cauchy-Schwarz inequality,
we have that
\begin{equation}
  \begin{aligned}
    \| \jump{\bv} \|_{H^{1/2}(\Gamma)}^2 \leq C \| \jump{\bv}
    \|_{L^2(\Gamma)} \| \jump{ \bv} \|_{H^1(\Gamma)} \leq C (h^{-1} \|
    \jump{\bv} \|_{L^2(\Gamma)}^2 + h \|  \jump{\nabla_{\Gamma} \bv}
    \|_{L^2(\Gamma)}^2), \quad \forall \bv \in \bV_h.
  \end{aligned}
  \label{eq_H1uboundC2}
\end{equation}
The right hand side of \eqref{eq_H1uboundC2} only involves the $L^2$
norms, and can be easily computed.

For the case that $\Gamma$ is polygonal (polyhedral), we let
$\Gamma_j(1 \leq j \leq J)$ denote the sides of $\Gamma$, where 
every side $\Gamma_j$ is a line segment (polygon). Since $\Gamma$ has
corners, the
estimate \eqref{eq_H1uboundC2} will be modified in a piecewise manner. 
From the embedding 
theory \cite[Theorem 1.5.2.1]{Grisvard2011elliptic}, we have that
$\bH^{s+1}(\Omega_i) \hookrightarrow \Pi_{j = 1}^J \bH^1(\Gamma_j)(i =
0, 1)$, which brings us that $\jump{\bv}|_{\Gamma_j} \in
\bH^1(\Gamma_j)(1 \leq j \leq J)$ for $\forall \bv \in \bV^{s+1}$.
Moreover, any $\bv_{h, i}^m \in \bV_{h, i}^m(i = 0, 1)$ is continuous
and piecewise smooth on each $\Gamma_j$. 
We can know that $(\pi_i \bv_h)|_{\Gamma_j} \in \bH^1(\Gamma_j)$
and $\jump{\bv_{h}}|_{\Gamma_j} \in \bH^1(\Gamma_j)$ for $\forall
\bv_{h} \in \bV_{h}^m$.  Similar to \eqref{eq_H1uboundC2}, the norm
$\| \cdot \|_{H^{1/2}(\Gamma)}$ can be bounded by
\begin{align}
  &\| \jump{\bv} \|_{H^{1/2}(\Gamma)}^2  \leq C \sum_{j = 1}^J \|
  \jump{\bv} \|_{H^{1/2}(\Gamma_j)}^2  \leq C \sum_{j = 1}^J \|
  \jump{\bv} \|_{L^2(\Gamma_j)} \| \jump{\bv} \|_{H^1(\Gamma_j)} 
  \label{eq_H1uboundpoly} \\
  \leq C \sum_{j = 1}^J (h^{-1} &\| \jump{\bv} \|_{L^2(\Gamma_j)}^2
  + h \| \jump{ \nabla_{\Gamma} \bv} \|_{L^2(\Gamma_j)}^2 ) \leq C
  h^{-1} \| \jump{\bv} \|_{L^2(\Gamma)}^2 + C h \sum_{j = 1}^J \|
  \jump{\nabla_{\Gamma} \bv} \|_{L^2(\Gamma_j)}^2,
  \quad\forall \bv \in \bV_h.\nonumber 
\end{align}
The second inequality follows from the Sobolev interpolation
inequality on each $\Gamma_j$. Compared with \eqref{eq_H1uboundC2}, 
the norm for the tangential gradient on $\Gamma$ in
\eqref{eq_H1uboundpoly} can be understood in a piecewise manner, 
i.e. we can formally let $\| \nabla_\Gamma \cdot \|_{L^{2}(\Gamma)}^2
:= \sum_{j =1}^J \| \nabla_\Gamma \cdot \|_{L^{2}(\Gamma_j)}^2$ for
the case $\Gamma$ is polygonal.  Then, the upper bound of the estimate
\eqref{eq_H1uboundpoly} will look the same as the upper bound of
\eqref{eq_H1uboundC2}.  To simplify the representation, we use the
notation consistent with the case that $\Gamma$ is $C^2$.

Now, we define the discrete least squares functional $\mJ_h(\cdot;
\cdot)$ that only involves the $L^2$ norms by 
\begin{equation}
  \begin{aligned}
    \mJ_h(\btau_h, \bv_h; \bm{f}, &\bm{a}, \bm{b}) :=
    J(\btau_h, \bv_h; \bm{f}) \\
    & + \Bs_h(\btau_h; \ba) + \Bu_h(\bv_h; \bb) +
    \shsnorm{\btau_h}^2 + \shunorm{\bv_h}^2, \quad \forall
    (\bm{\tau}_h, \bm{v}_h) \in \bSi_h \times \bV_h, 
  \end{aligned}
  \label{eq_mJh}
\end{equation}
where $J(\cdot; \cdot)$ is defined in \eqref{eq_JBsBu}, and
\begin{equation}
  \Bs_h(\btau_h; \ba) := \| \jumpn{\btau_h} - \ba
  \|_{L^2(\Gamma)}^2, \quad \Bu_h(\bv_h; \bm{b}) :=  h^{-1} \|
  \jump{\bm{v}_h} - \bm{b} \|_{L^2(\Gamma)}^2 + h \|
  \jump{\nabla_{\Gamma}\bm{v}_h } - \nabla_{\Gamma} \bm{b}
  \|_{L^2(\Gamma)}^2.
  \label{eq_BshBuh}
\end{equation}
The last two seminorms in \eqref{eq_mJh} are applied to guarantee the
uniform upper bound of the condition number of the resulting linear
system.

The numerical solutions are sought by solving the minimization problem
\begin{equation}
  \inf_{(\btau_h, \bv_h) \in \bSi_h^m \times \bV_h^m} \mJ_h(\btau_h,
  \bv_h; \bm{f}, \ba, \bb).
  \label{eq_minmJh}
\end{equation}
Since all terms in $\mJ_h(\cdot; \cdot)$ are defined with $L^2$ norms,
the problem  \eqref{eq_minmJh} can be solved by writing the
corresponding Euler-Lagrange equation, which reads:
seek $(\bsigma_h, \bu_h) \in \bSi_h^m \times \bV_h^m$ such that 
\begin{equation}
  a_{\mJ_h}(\bsigma_h, \bu_h; \btau_h, \bv_h) =
  l_{\mJ_h}(\btau_h, \bv_h), \quad \forall (\btau_h,
  \bm{v}_h) \in \bSi_h^m \times \bV_h^m,
  \label{eq_amJh}
\end{equation}
where the forms $a_{\mJ_h}(\cdot; \cdot)$ and $l_{\mJ_h}(\cdot)$ are
defined as 
\begin{displaymath}
  \begin{aligned}
    a_{\mJ_h}(\bm{\tau}_h, \bm{v}_h; \bm{\rho}_h, \bm{w}_h)& :=  (\mA
    \bm{\tau}_h - \bm{\varepsilon}(\bm{v}_h), \mA \bm{\rho}_h -
    \bm{\varepsilon}(\bm{w}_h))_{L^2(\Omega_0 \cup \Omega_1)} +
    (\nabla \cdot \bm{\tau}_h, \nabla \cdot \bm{\rho}_h)_{L^2(\Omega_0
    \cup \Omega_1)} \\
    + & (\jumpn{\bm{\tau}_h}, \jumpn{\bm{\rho}_h} )_{L^2(\Gamma)} +
    h^{-1} (\jump{\bm{v}_h}, \jump{\bm{w}_h} )_{L^2(\Gamma)}+ h
    (\jump{\nabla_{\Gamma} \bm{v}_h}, \jump{\nabla_{\Gamma}
    \bm{w}_h})_{L^2(\Gamma)} \\
    + & \Shm(\bm{\tau}_h, \bm{\rho}_h) + \Ghm(\bm{v}_h, \bm{w}_h), \\
  \end{aligned}
\end{displaymath}
and 
\begin{displaymath}
  \begin{aligned}
    l_{\mJ_h}(\bm{\tau}_h, \bm{v}_h) := -(\nabla \cdot \bm{\tau}_h,
    \bm{f})_{L^2(\Omega_0 \cup \Omega_1)} + (\jumpn{\bm{\tau}_h},
    \bm{a})_{L^2(\Gamma)} + h(\jump{\bm{v}_h}, \bm{b})_{L^2(\Gamma)} +
    h^{-1} (\jump{\nabla_{\Gamma} \bm{v}_h}, \nabla_{\Gamma}
    \bm{b})_{L^2(\Gamma)}.
  \end{aligned}
\end{displaymath}
The convergence analysis will be derived under the Lax-Milgram
framework. We are aiming to show that $a_{\mJ_h}(\cdot; \cdot)$ is
bounded and coercive and satisfies a weakened Galerkin orthogonality
property.  

We introduce the energy norm $\ehnorm{\cdot}$ by 
\begin{displaymath}
  \begin{aligned}
    \ehnorm{(\bm{\tau}_h, \bm{v}_h)}^2 &:= \| \bm{\tau}_h \|_{H(\div;
    \Omega_0 \cup \Omega_1)}^2 + \| \bm{v}_h \|_{H^1(\Omega_0 \cup
    \Omega_1)}^2 + \| \jumpn{\bm{\tau}_h} \|_{L^2(\Gamma)}^2 \\
    + & h^{-1} \| \jump{\bv_h} \|_{L^2(\Gamma)}^2 + h \|
    \jump{\nabla_{\Gamma} \bv_h }\|_{L^2(\Gamma)}^2 +
    \shsnorm{\btau_h}^2 + \shunorm{\bv_h}^2, \quad \forall
    (\bm{\tau}_h, \bm{v}_h) \in \bSi_h \times \bV_h.
  \end{aligned}
\end{displaymath}
By \eqref{eq_Hm12sbound} - \eqref{eq_H1uboundpoly}, we have that 
\begin{equation}
  \Bs(\btau_h; \bm{0}) = \| \jumpn{\btau_h} \|_{H^{-1/2}(\Gamma)}^2
  \leq C \Bs_h(\btau_h;
  \bm{0}), \quad \forall \btau_h \in
  \bSi_h,
  \label{eq_Bstauh}
\end{equation}
and
\begin{equation}
  \begin{aligned}
    \Bu(\bv_h; \bm{0}) = \| \jump{\bv_h} \|_{H^{1/2}(\Gamma)}^2 
    \leq  C \Bu_h(\bv_h; \bm{0}),
    \quad \forall \bv_h \in \bV_h.
  \end{aligned}
  \label{eq_Buvh}
\end{equation}
Combining the definitions to $\mJ(\cdot; \cdot)$ and $\mJ_h(\cdot;
\cdot)$ and the above two estimates, there holds
\begin{displaymath}
  \mc{J}(\btau_h, \bv_h; \bm{0}, \bm{0}, \bm{0}) \leq C \mJ_h(\btau_h,
  \bv_h; \bm{0}, \bm{0}, \bm{0}) = a_{\mJ_h}(\btau_h, \bv_h; \btau_h,
  \bv_h), \quad \forall (\btau_h, \bv_h) \in
  \bSi_h^m \times \bV_h^m.
\end{displaymath}
From Lemma \ref{le_enormequvi}, we find that 
\begin{equation}
  \begin{aligned}
    \| \btau_h \|_{H(\div; \Omega_0 \cup \Omega_1)}^2 +  \| \bv_h
    \|_{H^1(\Omega_0 \cup \Omega_1)}^2 
    & \leq C a_{\mJ_h}(\bm{\tau}_h, \bm{v}_h; \bm{\tau}_h, \bm{v}_h),
    \quad \forall (\btau_h, \bv_h) \in \bSi_h^m \times \bV_h^m.
  \end{aligned}
  \label{eq_amJhlower}
\end{equation}
The boundedness and the coercivity of the bilinear form $a_{\mJ_h}(\cdot;
\cdot)$ directly follows from the Cauchy-Schwarz inequality, the
estimate \eqref{eq_Abound}, and the estimate \eqref{eq_amJhlower}.
\begin{lemma}
  There exist constants $C$ such that 
  \begin{align}
    a_{\mJ_h}(\btau_h, \bv_h; \brho_h, \bw_h) & \leq C
    \ehnorm{(\btau_h, \bv_h)} \ehnorm{(\brho_h, \bw_h)}, && \forall
    (\btau_h, \bv_h), (\brho_h, \bw_h) \in \bSi_h \times \bV_h,
    \label{eq_amJhbound} \\
    a_{\mJ_h}(\btau_h, \bv_h; \btau_h, \bv_h) & \geq C \ehnorm{ (
    \btau_h, \bv_h)}^2, && \forall (\btau_h, \bv_h) \in \bSi_h^m
    \times \bV_h^m.
    \label{eq_amJhcoer}
  \end{align}
  \label{le_amJhbc}
\end{lemma}
Bringing the exact solution $(\bsigma, \bu)$ into the discrete
variational problem \eqref{eq_amJh} yields that
\begin{displaymath}
  \begin{aligned}
    a_{\mJ_h}(\bsigma_h, \bu_h; \btau_h, \bv_h) = l_{\mJ_h}(\btau_h,
    \bv_h) = a_{\mJ_h}(\bsigma, \bu; \btau_h, \bv_h) - \Shm(\bsigma,
    \btau_h) - \Ghm(\bu, \bv_h),
  \end{aligned}
\end{displaymath}
which leads to a weakened Galerkin orthogonality property of
$a_{\mJ_h}(\cdot; \cdot)$.
\begin{lemma}
  Let $(\bsigma, \bu) \in \bSi^s \times \bV^{s+1}$ be the exact
  solution to \eqref{eq_problem}, and let $(\bsigma_h, \bu_h) \in
  \bSi_h^m \times \bV_h^m$ be the numerical solution to
  \eqref{eq_amJh}, there holds
  \begin{equation}
    a_{\mJ_h}(\bsigma - \bsigma_h, \bu - \bu_h; \btau_h, \bv_h) =
    \Shm(\bsigma, \btau_h) + \Ghm(\bu, \bv_h).
    \label{eq_amJhGalerkinorth}
  \end{equation}
  \label{le_amJhGalerkinorth}
\end{lemma}
To complete the error estimation, we present an approximation estimate
under the norm $\ehnorm{\cdot}$. 
\begin{lemma}
  Let $(\bsigma, \bu) \in \bSi^s \times \bV^{s+1}$ be the exact
  solution to the problem \eqref{eq_problem}, there exists
  $(\bsigma_I, \bu_I) \in \bSi_h^m \times \bV_h^m$ such that
  \begin{equation}
    \ehnorm{ ( \bsigma - \bsigma_I, \bu - \bu_I)} \leq Ch^{t - 1/2} (
    \| \bsigma \|_{H^s(\div; \Omega_0 \cup \Omega_1)} + \| \bu
    \|_{H^{s+1}(\Omega_0 \cup \Omega_1)}), \quad t = \min(m, s).
    \label{eq_apperror1}
  \end{equation}
  Moreover, if $\bsigma \in \mb{H}^{s+1}(\Omega_0 \cup \Omega_1)$,
  there holds 
  \begin{equation}
    \ehnorm{ ( \bsigma - \bsigma_h, \bu - \bu_I)} \leq Ch^{t} ( \|
    \bsigma \|_{H^{s+1}(\Omega_0 \cup \Omega_1)} + \| \bm{u}
    \|_{H^{s+1}(\Omega_0 \cup \Omega_1)}), \quad t = \min(m, s).
    \label{eq_apperror2}
  \end{equation}
  \label{le_apperror}
\end{lemma}
\begin{proof}
  For $i = 0, 1$, we let $\bsigma_{I, i} \in \bSi_{h, i}^m$ be the
  interpolant of $\pi_i \bsigma$ into the space $\bSi_{h, i}^m$, and
  let $\bu_{I, i} \in \bV_{h, i}^m$ be the interpolant of $\pi_i \bu$
  into the space $\bV_{h, i}^m$. From the approximation properties of
  finite element spaces \cite{Demkowicz2005H1, Scott1990finite}, we
  have that
  \begin{equation}
    \begin{aligned}
      \| \pi_i \bsigma - \bsigma_{I, i} \|_{H(\div; \Ohi)} & \leq Ch^t
      \| \pi_i \bsigma \|_{H(\div; \Ohi)} \leq Ch^t \| \bsigma
      \|_{H(\div; \Omega_0 \cup \Omega_1)}, \\
      \| \pi_i \bu - \bu_{I, i} \|_{H^1(\Ohi)} & \leq C h^t \| \pi_i
      \bu \|_{H^{s+1}(\Ohi)} \leq C h^t \| \bu \|_{H^{s+1}(\Omega_0
      \cup \Omega_1)}.
    \end{aligned}
    \label{eq_spaceapp}
  \end{equation}
  Let $\bsigma_I := \bsigma_{I, 0} \cdot \chi_0 + \bsigma_{I, 1}
  \cdot \chi_1$ and $\bu_I := \bu_{I, 0} \cdot \chi_0 + \bu_{I, 1}
  \cdot \chi_1$. The errors $\| \bsigma - \bsigma_I \|_{H(\div;
  \Omega_0 \cup \Omega_1)}$ and $\| \bu - \bu_I \|_{H^1(\Omega_0 \cup
  \Omega_1)}$ can be bounded directly by \eqref{eq_spaceapp}. 
  The trace term  $\| \jumpn{\bsigma - \bsigma_I} \|_{L^2(\Gamma)}$ is
  split into two parts by the triangle inequality that
  \begin{displaymath}
    \| \jumpn{\bsigma - \bsigma_I} \|_{L^2(\Gamma)} \leq \| \pi_0
    \bsigma - \bsigma_{I, 0} \|_{L^2(\Gamma)} +  \| \pi_1
    \bsigma - \bsigma_{I, 1} \|_{L^2(\Gamma)}.
  \end{displaymath}
  Since $s \geq 1$, we are allowed to apply the $H^1$ trace estimate
  \eqref{eq_H1trace} and the approximation property
  \eqref{eq_spaceapp} to find that
  \begin{align}
    \| &\pi_i \bsigma - \bsigma_{I, i} \|_{L^2(\Gamma)}^2 = \sum_{K \in
    \MThG} \| \pi_i \bsigma - \bsigma_{I, i} \|_{L^2(\Gamma_K)}^2
    \label{eq_Bshbound} \\
    \leq C
    & \sum_{K \in \MThG} ( h^{-1} \| \pi_i \bsigma - \bsigma_{I, i}
    \|_{L^2(K)}^2 + h \| \pi_i \bsigma - \bsigma_{I, i} \|_{H^1(K)}^2)
    \leq Ch^{2t - 1} \| \bsigma \|_{H^s(\div; \Omega_0 \cup
    \Omega_1)}^2, \quad i = 0, 1. \nonumber
  \end{align}
  Similarly, we apply the trace estimate \eqref{eq_H1trace} to obtain
  that
  \begin{equation}
     h^{-1} \| \jump{\bu - \bu_I} \|_{L^2(\Gamma)}^2 + h \|
     \jump{\nabla_\Gamma (\bu - \bu_I)} \|_{L^2(\Gamma)}^2 \leq C
     h^{2t} \| \bu \|_{H^{s+1}(\Omega_0 \cup \Omega_1)}^2.
    \label{eq_ugradbound}
  \end{equation}
  Finally, we apply the $L^2$ extension property
  \eqref{eq_btL2bound} - \eqref{eq_buL2bound} to deduce that
  \begin{align*}
    \shsnorm{\bsigma - \bsigma_I} &\leq C( \| \pi_0 \bsigma -
    \bsigma_{I, 0} \|_{L^2(\Oho)} +  \| \pi_1 \bsigma -
    \bsigma_{I, 1} \|_{L^2(\Ohl)} ) \leq
    Ch^t \| \bsigma \|_{H^s(\Omega_0 \cup \Omega_1)} \\
    \shunorm{\bu - \bu_I} & \leq C ( \| \pi_0 \bu - \bu_{I, 0}
    \|_{L^2(\Oho)} + \| \pi_1 \bu - \bu_{I, 1} \|_{L^2(\Ohl)}) \leq C
    h^{t+1} \| \bu \|_{H^{s+1}(\Omega_0 \cup \Omega_1)}.
  \end{align*}
  Combining all above estimates yields the desired estimate
  \eqref{eq_apperror1}.
  In addition, if $\bsigma$ has a higher regularity, i.e. $\bsigma \in
  \mb{H}^{s+1}(\Omega_0 \cup \Omega_1)$, the upper bound in
  \eqref{eq_Bshbound} can be replaced by $Ch^{2t} \| \bsigma
  \|_{H^{s+1}(\Omega_0 \cup \Omega_1)}^2$, which leads to the estimate
  \eqref{eq_apperror2}. This completes the proof.
\end{proof}
The error estimation to the numerical solution of \eqref{eq_mJh}
follows from Lemma \ref{le_amJhbc} - Lemma \ref{le_apperror}.
\begin{theorem}
  Let $(\bsigma, \bu) \in \bSi^s \times
  \bV^{s+1}$ be the exact solution to the problem \eqref{eq_problem},
  and let $(\bsigma_h, \bu_h) \in \bSi_h^m \times \bV_h^m$ be the
  numerical solution to the problem \eqref{eq_mJh}, there holds 
  \begin{equation}
    \ehnorm{(\bsigma - \bsigma_h, \bu - \bu_h)} \leq Ch^{t - 1/2} ( \|
    \bsigma \|_{H(\div;\Omega_0 \cup \Omega_1)} + \| \bu
    \|_{H^{s+1}(\Omega_0 \cup \Omega_1)}), \quad t = \min(m, s).
    \label{eq_error1}
  \end{equation}
  Moreover, if $\bsigma \in
  \mb{H}^{s+1}(\Omega_0 \cup \Omega_1)$, there holds
  \begin{equation}
    \ehnorm{(\bsigma - \bsigma_h, \bu - \bu_h)} \leq Ch^{t} ( \|
    \bsigma \|_{H(\div;\Omega_0 \cup \Omega_1)} + \| \bu
    \|_{H^{s+1}(\Omega_0 \cup \Omega_1)}), \quad t = \min(m, s).
    \label{eq_error2}
  \end{equation}
  \label{th_error}
\end{theorem}
\begin{proof}
  The proof is standard.  For any $(\btau_h, \bv_h) \in \bSi_h^m
  \times \bV_h^m$, we apply \eqref{eq_amJhbound} -
  \eqref{eq_amJhGalerkinorth} to find that
  \begin{align*}
    &\ehnorm{(\bsigma_h - \btau_h, \bu_h - \bv_h)}^2  \leq C
    a_h(\bsigma_h - \btau_h, \bu_h - \bv_h; \bsigma_h - \btau_h, \bu_h
    - \bv_h) \\
    & = C ( a_h(\bsigma - \btau_h, \bu - \bv_h; \bsigma_h - \btau_h, \bu_h
    - \bv_h) - s_h(\bsigma, \bsigma_h - \btau_h) - s_h(\bu, \bu_h -
    \bw_h)) \\
    & \leq C ( \ehnorm{(\bsigma - \btau_h, \bu - \bv_h)} +
    \shsnorm{\bsigma} + \shunorm{\bu})  \ehnorm{(\bsigma_u - \btau_h,
    \bu_h - \bv_h)} \\ 
    & \leq C ( \ehnorm{(\bsigma - \btau_h, \bu - \bv_h)} + h^s \|
    \bsigma \|_{H^s(\Omega_0 \cup \Omega_1)} + h^{s+1} \| \bu
    \|_{H^{s+1}(\Omega_0 \cup \Omega_1)} )  \ehnorm{(\bsigma_u -
    \btau_h, \bu_h - \bv_h)}. 
  \end{align*}
  Applying the triangle inequality, we find that
  \begin{displaymath}
    \ehnorm{(\bsigma - \bsigma_h, \bu - \bu_h)} \leq C
    \ehnorm{(\bsigma - \btau_h, \bu - \bv_h)} + C h^s(\| \bsigma
    \|_{H^s(\div; \Omega_0 \cup \Omega_1)} + \| \bu
    \|_{H^{s+1}(\Omega_0 \cup \Omega_1)}).
  \end{displaymath}
  Since $(\btau_h, \bv_h)$ is arbitrary, the error estimate under the
  norm $\ehnorm{\cdot}$ follows from the approximation properties in
  Lemma \ref{le_apperror}. This completes the proof.
\end{proof}
We estimate the condition number for the
linear system \eqref{eq_amJh}, which is especially desired in the
unfitted method. 
\begin{theorem}
  There exists a constant $C$ such that 
  \begin{equation}
    \kappa(A_{\mJ_h}) \leq C h^{-2},
    \label{eq_amJhcont}
  \end{equation}
  where $A_{\mJ_h}$ is the linear system with respect to
  $a_{\mJ_h}(\cdot; \cdot)$.
  \label{th_amJhcont}
\end{theorem}
\begin{proof}
  Since the bilinear form $a_{\mJ_h}(\cdot; \cdot)$ is bounded and
  coercive with respect to the norm $\ehnorm{\cdot}$.  From
  \cite[Section 3.2]{Ern2006evaluation}, the main step is to show the
  relationship between the energy norm $\ehnorm{\cdot}$ and the $L^2$
  norm.  Note that the finite element spaces $\bSi_h^m$ and $\bV_h^m$
  are the combinations of $\bSi_{h, 0}^m$ and $\bSi_{h,1}^m$ and of
  $\bV_{h, 0}^m$ and $\bV_{h, 1}^m$, respectively, and the spaces
  $\bSi_{h,i}^m$ and $\bV_{h,i}^m$ are defined on the domain $\Ohi$.
  Hence, our goal is to show that 
  \begin{equation}
    \begin{aligned}
      \sum_{i = 0, 1} ( \| \pi_i \btau_h &\|_{L^2(\Ohi)}^2 + \| \pi_i
      \bv_h \|_{L^2(\Ohi)}^2)  \leq C \ehnorm{(\btau_h, \bv_h)}^2 \\
      & \leq C h^{-2}    \sum_{i = 0, 1} ( \| \pi_i \btau_h
      \|_{L^2(\Ohi)}^2 + \| \pi_i \bv_h \|_{L^2(\Ohi)}^2), \quad
      \forall (\btau_h, \bv_h) \in \bSi_h^m \times \bV_h^m.
    \end{aligned}
    \label{eq_ehnormL2bound}
  \end{equation}
  Since $\ehnorm{\cdot}$ is stronger than $\enorm{\cdot}$, we have
  that
  \begin{displaymath}
    \|\btau_h \|_{L^2(\Omega_0 \cup \Omega_1)} + \| \bv_h
    \|_{L^2(\Omega_0 \cup \Omega_1)} \leq C \enorm{(\btau_h, \bv_h)}
    \leq C \ehnorm{(\btau_h, \bv_h)}.
  \end{displaymath}
  The lower bound in \eqref{eq_ehnormL2bound} then follows from
  the estimates \eqref{eq_btL2bound} and \eqref{eq_buL2bound}.
  For the upper bound, we apply the triangle inequality and the
  inverse estimate to find that
  \begin{displaymath}
    \begin{aligned}
      \ehnorm{(\btau_h, \bv_h)}^2 \leq C \sum_{i = 0, 1} ( \|\pi_i
      \btau_h \|_{H^1(\Ohi)}^2 + \| \pi_i \bv_h \|_{H^1(\Ohi)}^2) \leq
      Ch^{-2}  \sum_{i = 0, 1} ( \| \pi_i \btau_{h} \|_{L^2(\Ohi)}^2 +
      \| \pi_i \bv_{h} \|_{L^2(\Ohi)}^2).
    \end{aligned}
  \end{displaymath}
  From \cite[Corollary 3.4]{Ern2006evaluation} and \cite[Section
  2.6]{Massing2019stabilized}, the estimate
  \eqref{eq_amJhcont} comes from \eqref{eq_ehnormL2bound}, which
  completes the proof.
\end{proof}
We have completed the error analysis for the discrete
variational form \eqref{eq_amJh}. The scheme is robust in the sense
that
the constants appearing in the error bounds \eqref{eq_error1} -  
\eqref{eq_error2} are independent of $\lambda$,  and also are
independent of how the interface $\Gamma$ cuts the mesh.
From Theorem \ref{th_error}, the convergence rate of the numerical
error under the energy norm is half order lower than the optimal rate
without the higher regularity assumption. The major reason is that we
use the stronger $L^2$ norm to replace the minus norm $\| \cdot
\|_{H^{-1/2}(\Gamma)}$ to define the new quadratic functional, and the
errors on the interface are essentially established by the $H^1$ trace
estimate as \eqref{eq_Bshbound} and \eqref{eq_ugradbound}.
In addition, both estimates require the exact solution 
$(\bsigma, \bu)$ has at least $H^1(\div) \times H^2$ regularity. 
Consequently, this method and the analysis cannot be extended to the
case of low regularity that $s < 1$. 

\subsection{The Least Squares Finite Element Method with the discrete
minus norm for $s > 1/2$}
\label{subsec_mnorm}
In this subsection, we give the numerical method that has the optimal
convergence speed for $s > 1/2$.
As stated before, the replacement of the $L^2$ norm cannot work and
the $H^1$ trace estimate is also inappropriate for the case of low
regularity. The natural choice is to involve the $\| \cdot
\|_{H^{-1/2}(\Gamma)}$ in the discrete least squares functional, but
the minus norm is not easy to compute. In this subsection, we follow
the idea in \cite{Bramble1997least, Bramble2001least} to employ a
discrete inner product, which is related to the minus norm $\| \cdot
\|_{H^{-1/2}(\Gamma)}$. We first give an inner product in
$\bH^{-1/2}(\Gamma)$.

For any $\bv \in \bH^{-1/2}(\Gamma)$, consider the elliptic problem 
$\bw - \Delta \bw = \bm{0}$ in $\Omega_0$ with $\partial_{\un} \bw =
\bv$ on $\Gamma = \partial \Omega_0$. The corresponding weak form
reads
\begin{equation}
  (\nabla \bm{w}, \nabla \bm{t})_{L^2(\Omega_0)} +
  (\bm{w}, \bm{t})_{L^2(\Omega_0)} = (\bm{v},  \bm{t})_{L^2(\Gamma)},
  \quad \forall \bm{t} \in \bH^1(\Omega_0),
  \label{eq_H12weak}
\end{equation}
The elliptic regularity theory indicates that the problem
\eqref{eq_H12weak} admits a unique solution $\bw \in \bH^1(\Omega_0)$
with $\| \bw \|_{H^1(\Omega_0)} \leq C \| \bv
\|_{H^{-1/2}(\Gamma)}$. This fact allows us to define an operator
$T:\bH^{-1/2}(\Gamma) \rightarrow \bH^1(\Omega_0)$ such that $T \bv$
is the solution to \eqref{eq_H12weak} for $\forall \bv \in
\bH^{-1/2}(\Gamma)$. From $T$ and \eqref{eq_H12weak}, we define an
inner product in $\bH^{-1/2}(\Gamma)$ as 
\begin{equation}
  (\bv, \bq)_{-1/2, \Gamma} := (\bv, T \bq)_{L^2(\Gamma)}, \quad
  \forall \bv, \bq \in H^{-1/2}(\Gamma),
  \label{eq_H12innerproduct}
\end{equation}
with the induced norm $\| \bv \|_{-1/2, \Gamma}^2 := (\bv, \bv)_{-1/2,
\Gamma}$.
We then show that $\| \cdot \|_{-1/2, \Gamma}$  and 
$\| \cdot \|_{H^{-1/2}(\Gamma)}$ are equivalent. For $\forall \bv \in
\bH^{-1/2}(\Gamma)$, we let $\bt = T \bv$ in \eqref{eq_H12weak}, 
which directly gives that $\|T \bv \|_{H^1(\Omega_0)} = \| \bv
\|_{-1/2, \Gamma}$. From the trace estimate $\| \bw
\|_{H^{1/2}(\Gamma)} \leq C \| \bw \|_{H^1(\Omega_0)}(\forall \bw \in
\bH^1(\Omega_0))$, we derive that
\begin{displaymath}
  \begin{aligned}
    \| \bv \|_{-1/2, \Gamma} = \frac{(\bv, T \bv)_{L^2(\Gamma)}}{\| T
    \bv \|_{H^1(\Omega_0)}} \leq C  \frac{(\bv, T \bv)_{L^2(\Gamma)}}{\| T
    \bv \|_{H^{1/2}(\Gamma)}} \leq C \| \bv \|_{H^{-1/2}(\Gamma)},
    \quad \forall \bv \in \bH^{-1/2}(\Gamma).
  \end{aligned}
\end{displaymath}
For any $\bz \in \bH^{1/2}(\Gamma)$, there exists
$\bvphi_{\bz} \in \bH^1(\Omega_0)$ such that $\bvphi_{\bz}|_{\Gamma} =
\bz$ and $ \| \bvphi_{\bz} \|_{H^1(\Omega_0)} \leq C \| \bz
\|_{H^{1/2}(\Gamma)}$. Given $\bv \in \bH^{-1/2}(\Gamma)$, we find
that 
\begin{displaymath}
  \begin{aligned}
    (\bv, \bz)_{L^2(\Gamma)} = (\bv, &\bvphi_{\bz})_{L^2(\Gamma)}  =
    (\nabla T \bv, \nabla \bvphi_{\bz})_{L^2(\Omega_0)} + (T \bv,
    \bvphi_{\bz})_{L^2(\Omega_0)} \\
    & \leq \| T \bv \|_{H^1(\Omega_0)} \| \bvphi_{\bz}
    \|_{H^1(\Omega_0)} = C \| \bv \|_{-1/2, \Gamma} \| \bz
    \|_{H^{1/2}(\Gamma)}, \quad \forall \bz \in \bH^{1/2}(\Gamma),
  \end{aligned}
\end{displaymath}
which indicates that $  \| \bv \|_{H^{-1/2}(\Gamma)} \leq C \| \bv
\|_{-1/2, \Gamma}$. The equivalence between both norms is reached.

Roughly speaking, we have defined an equivalent norm and an inner
product for the space $\bH^{-1/2}(\Gamma)$, although the operator $T$
and the inner product $(\cdot, \cdot)_{-1/2, \Gamma}$ still cannot be
directly computed. 
In the numerical scheme, we introduce a
discrete operator $T_h$ to replace $T$ to make the method
computationally feasible. 
The discrete operator $T_h$ is established with the space $\bV_{h,
0}^1$, which is the $C^0$ linear finite element space defined
on the partition
$\MTho$.
Given any $\bv \in \bH^{-1/2}(\Gamma)$, we define the following
discrete weak form: seek $\bw_h \in \bV_{h, 0}^1$ such that
\begin{equation}
  (\nabla \bw_h, \nabla \bt_h)_{L^2(\Omega_0)} + (\bw_h,
  \bt_h)_{L^2(\Omega_0)} + s_{h, 0}^1(\bw_h, \bt_h) =  (\bv,
  \bt_h)_{L^2(\Gamma)}, \quad \forall \bt_h \in \bV_{h, 0}^1.
  \label{eq_dH12weak}
\end{equation}
\begin{remark}
  For the discrete system \eqref{eq_dH12weak}, the ghost penalty form
  $s_{h, 0}^1(\cdot, \cdot)$ ensures the uniform upper bound of the
  corresponding matrix. From \cite[Theorem
  2.16]{Massing2019stabilized}, the upper bound of the condition
  number is $O(h^{-2})$ for \eqref{eq_dH12weak}.
  \label{re_Bcond}
\end{remark}
It can be easily seen that the problem \eqref{eq_dH12weak} admits a
unique solution in $\bV_{h, 0}^1$. Similarly, this property allows us
to define the operator $T_h: \bH^{-1/2}(\Gamma) \rightarrow \bV_{h,
0}^1$ that
$T_h \bv$ is the solution to the discrete problem \eqref{eq_dH12weak}
for $\forall \bv \in \bH^{-1/2}(\Gamma)$. 
As \eqref{eq_H12innerproduct}, we
define the discrete inner product $(\cdot, \cdot)_{-1/2, h}$ and the
induced norm $\| \cdot \|_{-1/2, h}$ as 
\begin{equation}
  (\bv, \bq)_{-1/2, h} := (\bv, T_h \bq)_{L^2(\Gamma)}, \quad \| \bv
  \|_{-1/2, h}^2 := (\bv, \bv)_{-1/2, h}, \quad \forall
  \bv, \bq \in \bH^{-1/2}(\Gamma).
  \label{eq_dH12inner}
\end{equation}
Let $\bt_h = \bw_h$ in \eqref{eq_dH12weak}, together with the
property \eqref{eq_shiL2extension}, 
it can be seen that
\begin{equation}
  \| \bv \|_{-1/2, h}^2 = (\bv, T_h \bv)_{-1/2, h} = \| T_h \bv
  \|_{H^1(\Omega_0)}^2 + \shubonorm{T_h \bv}^2 \geq \| T_h \bv
  \|_{H^1(\Omega_0)}^2,
  \label{eq_12hnormbound1}
\end{equation}
and
\begin{equation}
  \| \bv \|_{-1/2, h}^2 =  \| T_h \bv \|_{H^1(\Omega_0)}^2 +
  \shubonorm{T_h \bv}^2  \leq C \| T_h \bv \|_{H^1(\Oho)}^2.
  \label{eq_12hnormbound2}
\end{equation}
Then, we give the relationship between $T$ and $T_h$. 
\begin{lemma}
  There exist constants $C$ such that 
  \begin{equation}
    \| \bv \|_{-1/2, h} \leq C \| \bv \|_{{-1/2}, \Gamma}, \quad
    \forall \bv \in \bH^{-1/2}(\Gamma),
    \label{eq_d12normbound1}
  \end{equation}
  and
  \begin{equation}
    \| \bv \|_{{-1/2}, \Gamma} \leq C ( \| \bv \|_{-1/2, h} + h^{1/2}
    \| \bv \|_{L^2(\Gamma)}), \quad \forall \bv \in \bL^2(\Gamma).
    \label{eq_d12normbound2}
  \end{equation}
  \label{le_d12normbound}
\end{lemma}
\begin{proof}
  For $\forall \bv \in \bH^{-1/2}(\Gamma)$, from \eqref{eq_H12weak}
  and \eqref{eq_12hnormbound1}, we have that
  \begin{displaymath}
    \begin{aligned}
      \| \bv \|_{-1/2, h} = \frac{(\bv, T_h \bv)_{L^2(\Gamma)}}{ \|
      \bv \|_{-1/2, h}} \leq \frac{(\bv, T_h \bv)_{L^2(\Gamma)}}{ \|
      T_h \bv \|_{H^1(\Omega_0)}} \leq \sup_{\bt \in \bH^1(\Omega_0)}
      \frac{(\bv, \bt)_{L^2(\Gamma)}}{ \| \bt \|_{H^1(\Omega_0)}} \leq
      \| T \bv \|_{H^1(\Omega_0)} \leq  
      \| \bv \|_{-1/2, \Gamma},
    \end{aligned}
  \end{displaymath}
  which brings us the estimate \eqref{eq_d12normbound1}.

  From the Sobolev extension theory \cite{Adams2003sobolev}, there
  exists a linear extension operator $E_0:  \bH^1(\Omega_0)
  \rightarrow \bH^1(\Omega)$ such that $ (E^0 \bw)|_{\Omega_0} = \bw$,
  $\| E^0 \bw \|_{H^1(\Omega)} \leq \| \bw \|_{H^1(\Omega_0)}$ for
  $\forall \bw \in \bH^1(\Omega_0)$. Let $I_{h, 0}^1$ be the
  Scott-Zhang interpolation operator of the space $\bV_{h, 0}^1$,
  which satisfies the approximation estimates
  \begin{displaymath}
    \| \bw - I_{h, 0}^1 \bw \|_{H^q(\Oho)} \leq C h^{1-q} \| \bw
    \|_{H^1(\Oho)}, \quad q = 0, 1, \quad \forall \bw \in
    \bH^1(\Oho).
  \end{displaymath}
  From the trace estimate \eqref{eq_H1trace}, it is quite standard to
  derive that $\| \bw - I_{h, 0}^1 \bw \|_{L^2(\Gamma)} \leq C h^{1/2}
  \| \bw \|_{H^1(\Oho)}$ for $\forall \bw \in H^1(\Oho)$.  Then, we
  deduce that
  \begin{align*}
    \| \bv \|_{-1/2, \Gamma}^2 = (\bv, T \bv)_{L^2(\Gamma)} &= (\bv,
    E^0 (T \bv))_{L^2(\Gamma)} \\
    & = (\bv, I_{h, 0}^1 E^0 (T
    \bv))_{L^2(\Gamma)} + (\bv, E^0(T \bv) - I_{h,
    0}^1 (E^0 (T \bv)))_{L^2(\Gamma)}.
  \end{align*}
  The second term can be bounded by the approximation property, that is
  \begin{align*}
    (\bv, E^0(T \bv) - I_{h,0} (E^0 (T \bv)))_{L^2(\Gamma)} &\leq
    Ch^{1/2} \| \bv \|_{L^2(\Gamma)} \| E^0(T \bv) \|_{H^1(\Oho)}
    \leq Ch^{1/2}\| \bv \|_{L^2(\Gamma)}  \| T \bv
    \|_{H^1(\Omega_0)} \\
    & \leq   Ch^{1/2} \| \bv \|_{L^2(\Gamma)}  \| \bv \|_{-1/2,
    \Gamma}.
  \end{align*}
  We let $\bt_h =  I_{h, 0}^1(E^0 (T \bv))$ in \eqref{eq_dH12weak} and
  apply \eqref{eq_12hnormbound2} to bound the first term,
  \begin{align*}
    (\bv, I_{h, 0}^1(E^0& (T \bv)))_{L^2(\Gamma)}  \leq C( \| T_h \bv
    \|_{H^1(\Omega_0)} + \shubonorm{T_h \bv})(\|  I_{h, 0}^1 (E^0 (T
    \bv)) \|_{H^1(\Omega_0)} + \shubonorm{ I_{h, 0}^1 (E^0 (T \bv))}) \\
    & \leq C \| \bv \|_{-1/2, h} \|I_{h, 0}^1(E^0 (T
    \bv))\|_{H^1(\Oho)} \leq C \| \bv \|_{-1/2, h} \|E^0 (T
    \bv)\|_{H^1(\Oho)} \\
    &\leq C \| \bv \|_{-1/2, h} \|E^0 (T \bv)\|_{H^1(\Omega)}
    \leq C \| \bv \|_{-1/2, h} \|T
    \bv\|_{H^1(\Omega_0)} \leq C \| \bv \|_{-1/2, h} \| \bv
    \|_{-1/2, \Gamma}. 
  \end{align*}
  Collecting all above estimates yields the second estimate
  \eqref{eq_d12normbound2}, which completes the proof.
\end{proof}

Now, let us define the least squares functional $\wt{\mJ}_h(\cdot;
\cdot)$ by
\begin{equation}
  \begin{aligned}
    \wt{\mJ}_h(\btau_h, \bv_h; \bm{f}, & \bm{a}, \bm{b}) := J(\btau_h,
    \bv_h; \bm{f}) \\
    + & \wt{B}_h^{\bsigma}(\btau_h; \bm{a}) + B_h^{\bu}(\bv_h; \bm{b})
    + \shsnorm{\btau_h}^2 + \shunorm{\bv_h}^2, \quad \forall (\btau_h,
    \bv_h) \in \bSi_h \times \bV_h.
  \end{aligned}
  \label{eq_wtmJh}
\end{equation}
where $\Bu_h(\cdot; \cdot)$ is defined as in \eqref{eq_BshBuh} and
$\wt{B}_h^{\bsigma}(\cdot; \cdot)$ is defined as 
\begin{equation}
  \wt{B}_h^{\bsigma}(\btau_h; \bm{a}) := \| \jumpn{\btau_h} - \bm{a}
  \|_{-1/2, h}^2 + h \| \jumpn{\btau_h} - \bm{a}
  \|_{L^2(\Gamma)}^2, \quad \forall \btau_h \in \bSi_h.
  \label{eq_wtBhs}
\end{equation}
The numerical solutions are sought by minimizing the 
functional $\wt{\mJ}_h(\cdot; \cdot)$ over the finite element spaces
$\bSi_h^m \times \bV_h^m$.
This minimization problem is equivalent to a variational problem by
writing the Euler-Lagrange equation, which reads: 
seek $(\bsigma_h, \bu_h) \in \bSi_h^m \times \bV_h^m$ such that 
\begin{equation}
  a_{\wt{\mJ}_h}(\bsigma_h, \bu_h; \btau_h, \bv_h) =
  l_{\wt{\mJ}_h}(\btau_h, \bv_h), \quad \forall(\btau_h, \bv_h) \in
  \bSi_h^m \times \bV_h^m,
  \label{eq_awtmJh}
\end{equation}
where the forms  $a_{\wt{\mJ}_h}(\cdot; \cdot)$  and
$l_{\wt{\mJ}_h}(\cdot)$ are defined as 
\begin{displaymath}
  \begin{aligned}
    &a_{\wt{\mJ}_h}(\btau_h, \bv_h; \brho_h, \bw_h)  := (\mA \btau_h -
    \beps(\bv_h), \mA \brho_h - \beps(\bw_h))_{L^2(\Omega_0 \cup
    \Omega_1)} + (\nabla \cdot \btau_h, \nabla \cdot,
    \brho_h)_{L^2(\Omega_0 \cup \Omega_1)} \\
    + &  (\jumpn{\btau_h}, \jumpn{\brho_h})_{-1/2, h} + h
    (\jumpn{\btau_h}, \jumpn{\brho_h})_{L^2(\Gamma)} + h^{-1}
    (\jump{\bv_h}, \jump{\bw_h})_{L^2(\Gamma)} + h
    (\jump{\nabla_{\Gamma} \bv_h},
    \jump{\nabla_{\Gamma} \bw_h})_{L^2(\Gamma)} \\
    + & \Shm(\btau_h, \brho_h) + \Ghm(\bv_h, \bw_h),
  \end{aligned}
\end{displaymath}
and 
\begin{displaymath}
  \begin{aligned}
    l_{\wt{\mJ}_h}(\btau_h, \bv_h) := & -(\nabla \cdot \btau_h,
    \bm{f})_{L^2(\Omega_0 \cup \Omega_1)} + (\jumpn{\btau_h},
    \bm{a})_{-1/2, h} + h(\jumpn{\btau_h}, \bm{a})_{L^2(\Gamma)} \\
    + & h^{-1} (\jump{\bv_h}, \bm{b})_{L^2(\Gamma)} + h
    (\jump{\nabla_{\Gamma} \bv_h}, \nabla_{\Gamma}
    \bm{b})_{L^2(\Gamma)}.
  \end{aligned}
\end{displaymath}
From \eqref{eq_awtmJh}, the convergence analysis is still derived
under the Lax-Milgram framework.  We introduce the energy norm
$\wtehnorm{\cdot}$ by
\begin{displaymath}
  \begin{aligned}
    \wtehnorm{(\btau_h, \bv_h)}^2 := & \| \btau_h\|_{H(\div; \Omega_0
    \cup \Omega_1)}^2 + \| \bv_h \|_{H^1(\Omega_0 \cup \Omega_1)}^2 + 
    \| \jump{\btau_h} \|_{-1/2, h}^2 + h \| \jump{\btau_h}
    \|_{L^2(\Gamma)}^2  \\
    + & h^{-1} \| \jump{\bv_h} \|_{L^2(\Gamma)}^2 + h \|
    \jump{\nabla_{\Gamma} \bv_h} \|_{L^2(\Gamma)}^2 +
    \shsnorm{\btau_h}^2 + \shunorm{\bv_h}^2, \quad \forall(\btau_h,
    \bv_h) \in \bSi_h \times \bV_h. 
  \end{aligned}
\end{displaymath}
From Lemma \ref{le_d12normbound}, there holds 
\begin{displaymath}
  \Bs(\btau_h; \bm{0}) = \| \jump{\btau_h} \|_{H^{-1/2}(\Gamma)}^2
  \leq C ( \| \jump{\btau_h} \|_{-1/2, h}^2 + h \| \jump{\btau_h}
  \|_{L^2(\Gamma)}^2) = \wt{B}_h^{\bsigma}(\btau_h; \bm{0}) , \quad
  \forall \btau_h \in \bSi_h.
\end{displaymath}
Together with the estimate \eqref{eq_Buvh}, we can know that 
\begin{displaymath}
  \begin{aligned}
    \mJ(\btau_h, \bv_h; \bm{0},
    \bm{0}, \bm{0}) &\leq C \wt{\mJ}_h(\btau_h, \bv_h; \bm{0}, \bm{0},
    \bm{0}) 
    = C a_{\wt{\mJ}_h}(\btau_h, \bv_h; \btau_h, \bv_h), \quad \forall
    (\btau_h, \bv_h) \in \bSi_h \times \bV_h.
  \end{aligned}
\end{displaymath}
As Lemma \ref{le_amJhbc}, it is similar to get that
$a_{\wt{\mJ}_h}(\cdot;
\cdot)$ is bounded and coercive under the energy norm
$\wtehnorm{\cdot}$, and satisfies the weakened Galerkin orthogonality
property.
\begin{lemma}
  There exist constants $C$ such that 
  \begin{align}
    a_{\wt{\mJ}_h}(\btau_h, \bv_h; \brho_h, \bw_h) & \leq C
    \wtehnorm{(\btau_h, \bv_h)} \wtehnorm{(\brho_h, \bw_h)}, && \forall
    (\btau_h, \bv_h), (\brho_h, \bw_h) \in \bSi_h \times \bV_h,
    \label{eq_awtmJhbound} \\
    a_{\wt{\mJ}_h}(\btau_h, \bv_h; \btau_h, \bv_h) & \geq C \wtehnorm{
    ( \btau_h, \bv_h)}^2, && \forall (\btau_h, \bv_h) \in \bSi_h^m
    \times \bV_h^m.
    \label{eq_awtmJhcoer}
  \end{align}
  \label{le_awtmJhbc}
\end{lemma}
\begin{lemma}
  Let $(\bsigma, \bu) \in \bSi^s \times \bV^{s+1}$ be the exact
  solution to \eqref{eq_problem}, let $(\bsigma_h, \bu_h) \in \bSi_h^m
  \times \bV_h^m$ be the numerical solution to \eqref{eq_amJh}, there
  holds
  \begin{equation}
    a_{\wt{\mJ}_h}(\bsigma - \bsigma_h, \bu - \bu_h; \btau_h, \bv_h) =
    \Shm(\bsigma, \btau_h) + \Ghm(\bu, \bv_h).
    \label{eq_awtmJhGalerkinorth}
  \end{equation}
  \label{le_awtmJhGalerkinorth}
\end{lemma}
\begin{proof}
  The proofs are similar to Lemma \ref{le_amJhbc} and Lemma
  \ref{le_amJhGalerkinorth}. 
\end{proof}
Moreover, we give the approximation estimate under the norm
$\wtehnorm{\cdot}$.

\begin{lemma}
  Let $(\bsigma, \bu) \in \bSi^s \times
  \bV^{s+1} $ be the exact solution to the
  interface problem \eqref{eq_problem}, there exists $(\bsigma_I,
  \bu_I) \in \bSi_h^m \times \bV_h^m$ such that
  \begin{equation}
    \wtehnorm{(\bsigma - \bsigma_I, \bu - \bu_I)} \leq C h^t ( \|
    \bsigma \|_{H^s(\div; \Omega_0 \cup \Omega_1)} + \| \bu
    \|_{H^{s+1}(\Omega_0 \cup \Omega_1)}), 
    \label{eq_wtehapp}
  \end{equation}
  where $t = \min(m, s)$ for $s \geq 1$, and $t = \min(m, s) -
  \varepsilon$ with any $\varepsilon > 0$ for $1/2 < s < 1$.
  \label{le_wtehapp}
\end{lemma}
\begin{proof}
  We let $\bsigma_I := \bsigma_{I, 0} \cdot \chi_0 + \bsigma_{I, 1}
  \cdot \chi_1$, $\bu_I := \bu_{I, 0} \cdot \chi_0 + \bu_{I, 0} \cdot
  \chi_1$ be the interpolants of $\bsigma$ and $\bu$, which are
  defined in the proof in Lemma \ref{le_apperror}.
  The errors $\| \bsigma - \bsigma_I \|_{H(\div; \Omega_0 \cup
  \Omega_1)}$, $\| \bu - \bu_I \|_{H^1(\Omega_0 \cup \Omega_1)}$,
  $\shsnorm{\bsigma_I}$ and $\shunorm{\bu_I}$ have been estimated in 
  Lemma \ref{le_apperror}. We only bound the trace terms in
  $\wtehnorm{\cdot}$. For the error $\| \jumpn{\bsigma - \bsigma_I}
  \|_{-1/2, h}$, we apply the estimate \eqref{eq_d12normbound1}, the
  approximation properties \eqref{eq_spaceapp} and the trace theorem
  of functions in $\bH(\div; \Omega_0 \cup \Omega_1)$ to find that
  \begin{align*}
    &\| \jumpn{\bsigma -\bsigma_I} \|_{-1/2, h}  \leq C  \| \jumpn{
    \bsigma - \bsigma_I} \|_{-1/2, \Gamma} \leq C  \| \jumpn{ \bsigma
    - \bsigma_I} \|_{H^{-1/2}(\Gamma)} \\
    & \leq C \sum_{i = 0,1} \| \un \cdot (\pi_i \bsigma  - \pi_i
    \bsigma_I) \|_{H^{-1/2}(\Gamma)}  \leq C \sum_{i = 0, 1} \| \pi_i
    \bsigma - \pi_i \bsigma_I \|_{H(\div; \Omega_i)} \leq Ch^{t_0}
    \| \bsigma \|_{H(\div; \Omega_0 \cup \Omega_1)},
  \end{align*}
  where $t_0 := \min(m, s)$. 
  Since $\bu \in \bH^{1+s}(\Omega_0 \cup \Omega_1)$, the error $h^{-1}
  \| \jump{\bu - \bu_I} \|_{L^2(\Gamma)}^2$ can be directly estimated
  by the $H^1$ trace estimate \eqref{eq_H1trace}, 
  \begin{displaymath}
    h^{-1} \| \jump{\bu - \bu_I} \|_{L^2(\Gamma)}^2 \leq Ch^{2t_0} \|
    \bu \|_{H^{s+1}(\Omega_0 \cup \Omega_1)}^2,
  \end{displaymath}
  which is the same as in the proof to Lemma \ref{le_apperror}.

  The rest is to bound the terms $ h \| \jumpn{\bsigma - \bsigma_I}
  \|_{L^2(\Gamma)}^2$ and $ h \| \jump{\nabla_{\Gamma} ( \bu - \bu_I)}
  \|_{L^2(\Gamma)}^2$. For the case $s \geq 1$, both errors can be
  bounded by the $H^1$ trace estimate, and we derive that
  \begin{align}
    h \| \jumpn{\bsigma - &\bsigma_I} \|_{L^2(\Gamma)}^2  \leq Ch
    \sum_{i = 0, 1} \| \pi_i \bsigma - \pi_i \bsigma_I
    \|_{L^2(\Gamma)}^2  \leq Ch \sum_{i = 0, 1} \sum_{K \in \MThG} \|
    \pi_i \bsigma - \pi_i \bsigma_I \|_{L^2(\Gamma_K)}^2 \nonumber \\
    & \leq C \sum_{i = 0, 1} \sum_{K \in \MThG} ( \| \pi_i \bsigma -
    \pi_i \bsigma_I \|_{L^2(K)}^2 + h^2 \| \pi_i \bsigma - \pi_i
    \bsigma_I \|_{H^1(K)}^2) \leq C h^{2t_0} \| \bsigma
    \|_{H^s(\Omega_0 \cup \Omega_1)}, \label{eq_bsigmabsigmaI}
  \end{align}
  and similarly, there holds
  \begin{displaymath}
    h \| \jump{\nabla_{\Gamma}(\bu - \bu_I)}
    \|_{L^2(\Gamma)}^2 \leq C h^{2t_0} \| \bu \|_{H^{s+1}(\Omega_0 \cup
    \Omega_1)}.
  \end{displaymath}
  For the case $1/2 < s < 1$, 
   the error $h \| \jumpn{\bsigma - \bsigma_I}
  \|_{L^2(\Gamma)}^2$ cannot be bounded as \eqref{eq_bsigmabsigmaI}
  because the $H^1$ trace estimate is now unavailable.
  In this case, the embedding theory will be the main tool to
  estimate the errors. 
  We let $\mb{V}_h^m \subset \mb{H}^1(\Omega)$ be
  the tensor-valued $C^0$ finite element space of degree $m$ on the
  mesh $\MTh$.  Since $\pi_i \bsigma \in \mb{H}^s(\Omega)$, we let
  $\pi_i^{\bmr{s}} \bsigma$ be its Scott-Zhang interpolant into the
  space $\mb{V}_h^m$ \cite{Ciarlet2013analysis}, which satisfies that 
  $\| \pi_i^{\bmr{s}} \bsigma - \pi_i \bsigma \|_{L^2(\Omega)} \leq C
  h^{t_0} \| \pi_i \bsigma \|_{H^s(\Omega)}$.
  For arbitrarily small $\varepsilon > 0$, 
  the space  $\mb{H}^{1/2 + \varepsilon}(\Omega_i)$
  continuously embeds into $\mb{L}^2(\Gamma)$. 
  Notice that 
  $\pi_i^{\bmr{s}} \bsigma, \pi_i
  \bsigma \in \mb{H}^{1/2 + \varepsilon}(\Omega_i)$ for small enough
  $\varepsilon$, 
  and we apply the
  inverse estimate and the triangle inequality to derive that
  \begin{align*}
    &h \| \jump{\bsigma - \bsigma_I} \|_{L^2(\Gamma)}^2  \leq
    Ch \sum_{i = 0, 1} \| \pi_i \bsigma - \pi_i \bsigma_I
    \|_{L^2(\Gamma)}^2 \\
    \leq  C&h\sum_{i = 0, 1}  \| \pi_i \bsigma - \pi_i^{\bmr{s}}
    \bsigma \|_{L^2(\Gamma)}^2  + Ch\sum_{i = 0, 1}  \|
    \pi_i^{\bmr{s}} \bsigma -  \pi_i \bsigma_I \|_{L^2(\Gamma)}^2   \\
    \leq  C&h\sum_{i = 0, 1}  \| \pi_i \bsigma - \pi_i^{\bmr{s}}
    \bsigma \|_{H^{1/2 + \varepsilon}(\Omega)}^2 + C  \sum_{i = 0, 1}
    \sum_{K \in \MThG} \| \pi_i^{\bmr{s}} \bsigma - \pi_i \bsigma_I
    \|_{L^2(K)}^2  \\
    \leq  C&h\sum_{i = 0, 1}  \| \pi_i \bsigma - \pi_i^{\bmr{s}}
    \bsigma \|_{H^{1/2 + \varepsilon}(\Omega)}^2 + C  \sum_{i = 0, 1}
   \| \pi_i^{\bmr{s}} \bsigma - \pi_i \bsigma_I
    \|_{L^2(\Ohi)}^2  \\
    \leq C&h^{2t_1} \| \bsigma \|_{H^s(\Omega_0 \cup \Omega_1)}^2 + C
    \sum_{i = 0, 1} ( \| \pi_i \bsigma -
    \pi_i^{\bmr{s}} \bsigma \|^2_{L^2(\Ohi)} + \| \pi_i \bsigma  -
    \pi_i \bsigma_I \|^2_{L^2(\Ohi)})  \leq Ch^{2t_1} \| \bsigma
    \|_{H^s(\div; \Omega_0 \cup \Omega_1)}^2,
  \end{align*}
  where $t_1 = \min(m, s) - \varepsilon$. 
  From the  embedding $\bH^{3/2 +
  \varepsilon}(\Omega_i) \hookrightarrow \bH^1(\Gamma)$
  \cite{Ding1996proof}, it is similar to get that
  \begin{displaymath}
    h \| \jump{\nabla_{\Gamma} (\bu - \bu_I)}
    \|_{L^2(\Gamma)}^2 \leq C h^{2t_1} \| \bu \|_{H^{s+1}(\Omega_0 \cup
    \Omega_1)}^2.
  \end{displaymath}
  Collecting all above estimates leads to the
  approximation property \eqref{eq_wtehapp}, which completes the
  proof.
\end{proof}
The error estimation for the numerical solution of \eqref{eq_awtmJh}
is reached. The proof is the same as Theorem \ref{th_error}.
\begin{theorem}
  Let $(\bsigma, \bu) \in \bSi^s \times
  \bV^{s+1}$ be the exact solution to the
  interface problem \eqref{eq_problem}, and let $(\bsigma_h, \bu_h)
  \in \bSi_h^m \times \bV_h^m$ be the numermical solution to the
  problem \eqref{eq_awtmJh}, there holds 
  \begin{equation}
    \wtehnorm{(\bsigma - \bsigma_h)} \leq Ch^t ( \| \bsigma
    \|_{H^s(\div; \Omega_0 \cup \Omega_1)} + \| \bu
    \|_{H^{s+1}(\Omega_0 \cup \Omega_1)}),
    \label{eq_wtJherror}
  \end{equation}
  where $t = \min(m, s)$ for $s \geq 1$, and $t = \min(m, s) -
  \varepsilon$ with any $\varepsilon > 0$ for $1/2 < s < 1$.
  \label{th_wtJherror}
\end{theorem}
The condition number of the linear system of
$a_{\wt{\mJ}_h}(\cdot; \cdot)$ also has a uniform upper bound.  The
proof is the same to Theorem \ref{th_amJhcont}.
\begin{theorem}
  There exists $C$ such that
  \begin{equation}
    \kappa(A_{{\wt{\mJ}_h}}) \leq C h^{-2},
    \label{eq_wtacont}
  \end{equation}
  where $A_{\wt{\mJ}_h}$ is the linear system with respect to
  $a_{\wt{\mJ}_h}(\cdot; \cdot)$.
  \label{th_wtacont}
\end{theorem}
From Theorem \ref{th_wtJherror}, it can be seen that the scheme is
also robust since the constants in the error bounds
\eqref{eq_wtJherror} are independent of $\lambda$ and the relative
location between $\Gamma$ and the mesh.
Compared with the $L^2$ norm least squares finite element method, 
the convergence rate is nearly optimal without a higher regularity
assumption, but the linear system is more complicated. 

Ultimately, let us give some details about solving the linear system
$A_{\wt{\mJ}_h}$. 
Since $a_{\wt{\mJ}_h}(\cdot; \cdot)$ involves the discrete minus inner
product $(\cdot, \cdot)_{-1/2, h}$, we are required to solve the
elliptic system \eqref{eq_dH12weak} in the assembling of
$A_{\wt{\mJ}_h}$. Let $B$ be the sparse matrix of the bilinear form
\eqref{eq_dH12weak}, which is invertible and satisfy that $\kappa(B) \leq
Ch^{-2}$, see Remark \ref{re_Bcond}.
Let $C$ be the sparse matrix corresponding to the inner product
$(\jump{\btau_h}, \bw_h)_{L^2(\Gamma)}$ for $\forall (\btau_h, \bv_h)
\in \bSi_h^m \times \bV_h^m, \forall \bw_h \in \bV_{h, 0}^1$. 
We note that this definition ensures the matrix $C$ has the same
column size as the matrix $A_{\wt{\mJ}_h}$.  Then, $A_{\wt{\mJ}_h}$
has the form that
\begin{equation}
  A_{\wt{\mJ}_h} = A_{L^2} + C^TB^{-1} C,
  \label{eq_Adecop}
\end{equation}
where $A_{L^2}$ corresponds to all $L^2$ inner products in
$a_{\wt{\mJ}_h}(\cdot, \cdot)$, i.e. $A_{L^2}$ is the matrix of the 
bilinear form $a_{L^2}(\cdot; \cdot)$ that
\begin{align}
  &a_{L^2}(\btau_h, \bv_h; \brho_h, \bw_h)  := (\mA \btau_h -
  \beps(\bv_h), \mA \brho_h - \beps(\bw_h))_{L^2(\Omega_0 \cup
  \Omega_1)} + (\nabla \cdot \btau_h, \nabla \cdot,
  \brho_h)_{L^2(\Omega_0 \cup \Omega_1)} 
  \label{eq_aL2} \\
  +  h
  (\jumpn{\btau_h}&, \jumpn{\brho_h})_{L^2(\Gamma)} + h^{-1}
  (\jump{\bv_h}, \jump{\bw_h})_{L^2(\Gamma)} + h
  (\jump{\nabla_{\Gamma} \bv_h},
  \jump{\nabla_{\Gamma} \bw_h})_{L^2(\Gamma)} 
  + \Shm(\btau_h, \brho_h) + \Ghm(\bv_h, \bw_h). \nonumber
\end{align}
Suppose that we have a fast algorithm to solve the linear
system $B \bmr{y} = \bmr{z}$, 
then we can easily compute the
matrix-vector products for the matrix $A_{\wt{\mJ}_h}$ by
\eqref{eq_Adecop}. Consequently, the Krylov iterative method (e.g.
GMRES) can be used as the solver for the linear system. 
Although we have shown that the condition number of
$A_{\wt{\mJ}_h}$ is $O(h^{-2})$, an effective preconditioner is still
expected especially when $h$ tends to zero. 
But $A_{\wt{\mJ}_h}$ involves the inverse matrix $B^{-1}$, it is
costly to form it and it is also not convenient to even extract its
diagonal. Traditional diagonal or block diagonal preconditioners are
hard to use for $A_{\wt{\mJ}_h}$. 
One method is to use the matrix-free preconditioning technique to
$A_{\wt{\mJ}_h}$. For example, one can use the hierarchically
semiseparable (HSS) approximation for the given matrix to accelerate
the convergence of iterative methods \cite{Chandrasekaran2006fast,
Xi2014fast}. The construction of the HSS approximation may 
potentially only use matrix-vector products instead of the original
matrix itself, and we refer to \cite{Lin2011fast, Xi2014fast} for
fully matrix-free techniques to the construction of the HSS
approximation. This idea has been used in the immersed finite element 
method solving the elliptic interface problem. Once the HSS
approximation $H$ for $A_{\wt{\mJ}_h}$ is obtained in a structured
form, it can be quickly factorized and the factors can be used as a
preconditioner. Alternatively, we present a matrix-explicit
preconditioning method for $A_{\wt{\mJ}_h}$. As \eqref{eq_Adecop}, 
the matrix $A_{\wt{\mJ}_h}$  can be split into two parts. 
From \eqref{eq_aL2}, it can be easily seen that 
$a_{L^2}(\btau_h, \bv_h; \btau_h, \bv_h) = 0$ implies $\btau_h =
\bm{0}$, $\bv_h = \bm{0}$ for $(\btau_h, \bv_h) \in \bSi_h^m \times
\bV_h^m$. Hence, $A_{L^2}$ is symmetric positive
definite. Any preconditioning technique can be applied to $A_{L^2}$ 
because $A_{L^2}$ is entirely explicit. As a numerical observation, 
the preconditioner from $A_{L^2}$ can also significantly accelerate
the iterative methods. We show the results in Example 1 of Section
\ref{sec_numericalresults} by constructing the HSS
approximation from $A_{L^2}$ to obtain the preconditioner. 
For solving the linear system $B \bmr{y} = \bmr{z}$, we also use the
HSS approximation as the preconditioner. The codes of the construction
and the factorization of the HSS approximation are freely available in
STRUMPACK \cite{Rouet2016strumpack}.
A comprehensive analysis and a more appropriate method to solve the
linear system are now left in the future study.

%% file: numericalresults.tex
\section{Numerical Results}
\label{sec_numericalresults}
In this section, a series of numerical tests are presented to show
the numerical performance of the proposed method. For all tests, the
source function $\bm{f}$ and the jump condition $\ba, \bb$ are taken
from the exact solution accordingly.
In Example 1 - Example 5, we consider the
elasticity interface problems in two dimensions on the squared domain
$\Omega = (-1, 1)^2$. We adopt a family of triangular meshes with
the mesh size $h = 1/5, 1/10, \ldots, 1/40$ to solve all tests.  In
Example 1 - Example 3, the interface $\Gamma$ is of class $C^2$ and
is described by a level set function, see
Fig.~\ref{fig_2ddomaininterface}. 
In Example 4 - Example 5, the
interface is taken to be the boundary of the L-shaped domain.
In Example 6, we solve a three-dimensional interface problem to
illustrate the numerical performance.

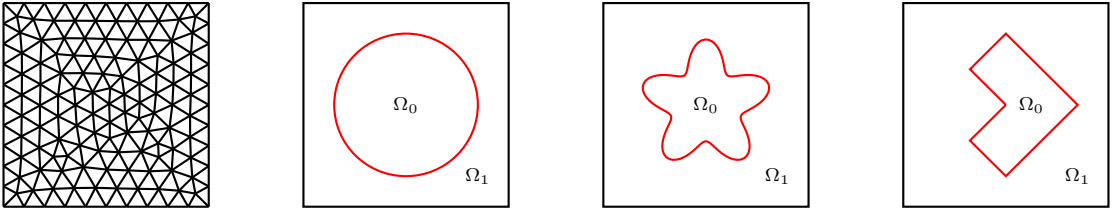
\begin{figure}[htp]
  \centering
  \begin{minipage}[t]{0.23\textwidth}
    \centering
    \begin{tikzpicture}[scale=1.35]
      \centering
      \input{./figure/MThex1.tex}
    \end{tikzpicture}
  \end{minipage}
 \begin{minipage}[t]{0.23\textwidth}
    \centering
    \begin{tikzpicture}[scale=1.35]
      \centering
      \node at (0, 0) {\tiny $\Omega_0$};
      \node at (0.7, -0.7) {\tiny $\Omega_1$};
      \draw[thick, black] (-1.0, -1.0) rectangle (1.0, 1.0);
      \draw[thick, red] (0, 0)  circle [radius=0.7];
    \end{tikzpicture}
  \end{minipage}
 \begin{minipage}[t]{0.23\textwidth}
    \centering
    \begin{tikzpicture}[scale=1.35]
      \centering
      \node at (0, 0) {\tiny $\Omega_0$};
      \node at (0.7, -0.7) {\tiny $\Omega_1$};
      \draw[thick, black] (-1.0, -1.0) rectangle (1.0, 1.0);
      \draw[thick, domain=0:360, red, samples=120] plot (\x:{(0.5 +
      sin(\x*5)/7)*1.00});
    \end{tikzpicture}
  \end{minipage}
  \begin{minipage}[t]{0.23\textwidth}
    \centering
    \begin{tikzpicture}[scale=1.35]
      \centering
      \node at (0.25, 0) {\tiny $\Omega_0$};
      \node at (0.7, -0.7) {\tiny $\Omega_1$};
      \draw[thick, black] (-1.0, -1.0) rectangle (1.0, 1.0);
      \draw[thick, red] (0, 0) -- (-0.35, 0.35) -- (0, 0.7) -- (0.7,
      0.0) -- (0, -0.7) -- (-0.35, -0.35) -- (0, 0);
    \end{tikzpicture}
  \end{minipage}
  \caption{The unfitted mesh and the interfaces in two dimensions.}
  \label{fig_2ddomaininterface}
\end{figure}

\paragraph{\textbf{Example 1.}}
We first consider a linear elasticity interface problem with a circular
interface centered at the origin with radius $r = 0.7$, see
Fig.~\ref{fig_2ddomaininterface}. We solve this problem to study the
convergence rates.  The exact solution is taken as 
\begin{displaymath}
  \bm{u}(x, y) = \begin{bmatrix}
    \sin(2\pi y)(-1 + \cos(2\pi x)) + \frac{1}{1 + \lambda} \sin(\pi
    x)\sin(\pi y) \\
    \sin(2 \pi x)(1 - \cos(2\pi y)) + \frac{1}{1 + \lambda} \sin(\pi
    x) \sin(\pi y) \\
  \end{bmatrix}, \quad \text{in } \Omega_0 \cup \Omega_1,
\end{displaymath}
with the discontinuous parameter $\lambda|_{\Omega_0} = \lambda_0 = 5,
\lambda|_{\Omega_1} = \lambda_1 = 1, \mu|_{\Omega_0} = \mu_0 = 2,
\mu|_{\Omega_1} = \mu_1 = 1$.
The numerical results are gathered in Tab.~\ref{tab_example1} by the
$L^2$ norm least squares finite element method and the least squares
finite element method with the discrete minus norm.
For both methods, the energy norms $\ehnorm{\cdot}$ and
$\wtehnorm{\cdot}$ are stronger than the norm $\enorm{\cdot}$. 
We calculate $\enorm{(\bsigma - \bsigma_h, \bu - \bu_h)}$ to represent
the errors under the energy norm.
From Tab.~\ref{tab_example1}, it can
be seen that the errors under the energy norm
approach zero at the optimal rates.
For the $L^2$ errors, $\| \bm{u} - \bm{u}_h
\|_{L^2(\Omega)}$ and $\| \bsigma - \bsigma_h \|_{L^2(\Omega)}$ have
optimal/suboptimal convergence speeds.  All numerically observed
convergence orders are in agreement with the theoretical analysis in
Theorem \ref{th_error} and Theorem \ref{th_wtJherror}.

\begin{table}
  \centering
  \renewcommand\arraystretch{1.15}
  \scalebox{0.630}{
  \begin{tabular}{p{0.3cm} | p{3.3cm} | p{1.3cm} | p{1.3cm} 
    | p{1.3cm} |  p{1.3cm} | p{0.9cm} }
    \hline\hline
    $m$ & $h$ & 1/5 & 1/10 & 1/20 & 1/40 & order \\
    \hline
    \multirow{3}{*}{$1$} & $\| \bm{u}-\bm{u}_h \|_{L^2(\Omega_0 \cup
    \Omega_1)}$
    & 1.86e-1 & 4.76e-2 &  1.19e-2 & 2.98e-3 & 2.00 \\
    \cline{2-7}
    & $\| \bsigma - \bsigma_h \|_{L^2(\Omega_0 \cup \Omega_1)}$
    & 3.49e-0 & 1.71e-0 &  8.58e-1 & 4.32e-1 & 0.99 \\
    \cline{2-7}
    & $\enorm{ (\bsigma - \bsigma_h, \bu - \bu_h)}$
    & 2.22e1 & 1.12e1 &  5.66e-0 & 2.85e-0 & 0.99 \\
    \hline
    \multirow{3}{*}{$2$} & $\| \bm{u}-\bm{u}_h \|_{L^2(\Omega_0 \cup
    \Omega_1)}$
    & 4.01e-3 & 4.92e-4 &  6.15e-5 & 7.69e-6 & 2.99 \\
    \cline{2-7}
    & $\| \bsigma - \bsigma_h \|_{L^2(\Omega_0 \cup \Omega_1)}$
    & 2.99e-1 & 7.48e-2 &  1.89e-2 & 4.77e-3 & 1.99 \\
    \cline{2-7}
    & $\enorm{ (\bsigma - \bsigma_h, \bu - \bu_h)}$
    & 2.03e-0 & 5.13e-1 &  1.28e-1 & 3.23e-2 & 1.99 \\
    \hline
    \multirow{3}{*}{$3$} & $\| \bm{u}-\bm{u}_h \|_{L^2(\Omega_0 \cup
    \Omega_1)}$
    & 1.65e-4 & 8.95e-6 &  5.42e-7 & 3.36e-8 & 4.01 \\
    \cline{2-7}
    & $\| \bsigma - \bsigma_h \|_{L^2(\Omega_0 \cup \Omega_1)}$
    & 1.93e-2 & 2.31e-3 &  2.88e-4 & 3.59e-5 & 3.00 \\
    \cline{2-7}
    & $\enorm{ (\bsigma - \bsigma_h, \bu - \bu_h)}$
    & 1.26e-1 & 1.60e-2 &  2.02e-3 & 2.53e-3 & 2.99 \\
    \hline\hline
  \end{tabular} \hspace{5pt}
  \begin{tabular}{p{0.3cm} | p{3.3cm} | p{1.5cm} | p{1.5cm} 
    | p{1.5cm} |  p{1.5cm} | p{0.9cm} }
    \hline\hline
    $m$ & $h$ & 1/5 & 1/10 & 1/20 & 1/40 & order \\
    \hline
    \multirow{3}{*}{$1$} & $\| \bm{u}-\bm{u}_h \|_{L^2(\Omega_0 \cup
    \Omega_1)}$
    & 1.88e-1 & 4.83e-2 &  1.21e-2 & 3.03e-3 & 2.00 \\
    \cline{2-7}
    & $\| \bsigma - \bsigma_h \|_{L^2(\Omega_0 \cup \Omega_1)}$
    & 3.50e-0 & 1.71e-0 &  8.58e-1 & 4.32e-1 & 0.99 \\
    \cline{2-7}
    & $\enorm{ (\bsigma - \bsigma_h, \bu - \bu_h)}$
    & 2.22e1 & 1.12e1 &  5.66e-0 & 2.842e-0 & 1.00 \\
    \hline
    \multirow{3}{*}{$2$} & $\| \bm{u}-\bm{u}_h \|_{L^2(\Omega_0 \cup
    \Omega_1)}$
    & 4.02e-3 & 4.93e-4 &  6.13e-5 & 7.68e-6 & 3.00 \\
    \cline{2-7}
    & $\| \bsigma - \bsigma_h \|_{L^2(\Omega_0 \cup \Omega_1)}$
    & 3.00e-1 & 7.47e-2 &  1.90e-2 & 4.78e-3 & 2.00 \\
    \cline{2-7}
    & $\enorm{ (\bsigma - \bsigma_h, \bu - \bu_h)}$
    & 2.03e-0 & 5.12e-1 &  1.29e-1 & 3.22e-2 & 2.00 \\
    \hline
    \multirow{3}{*}{$3$} & $\| \bm{u}-\bm{u}_h \|_{L^2(\Omega_0 \cup
    \Omega_1)}$
    & 1.64e-4 & 8.93e-6 &  5.41e-7 & 3.38e-8 & 4.00 \\
    \cline{2-7}
    & $\| \bsigma - \bsigma_h \|_{L^2(\Omega_0 \cup \Omega_1)}$
    & 1.92e-2 & 2.30e-3 &  2.87e-4 & 3.60e-5 & 3.01 \\
    \cline{2-7}
    & $\enorm{ (\bsigma - \bsigma_h, \bu - \bu_h)}$
    & 1.26e-1 & 1.61e-2 &  2.01e-3 & 2.51e-3 & 3.00 \\
    \hline\hline
  \end{tabular}
  }
  \caption{The numerical results for Example 1 by the $L^2$ norm least
  squares finite element method (left) / the least squares finite
  element method with the discrete minus norm (right).}
  \label{tab_example1}
\end{table}

In this example, we demonstrate the numerical performance of the
iterative method for solving the resulting linear system. 
For all tests, we use GMRES as the iterative solver and the iteration
stops when the relative error $\| A \bmr{x}^k - \bmr{b} \|_{l^2} / \|
\bmr{b} \|_{l^2} < 10^{-8}$ at the stage $k$ is smaller than the
tolerance $10^{-8}$.
For the $L^2$ norm least squares finite element method, the resulting
matrix is entirely explicit and any preconditioning technique can be
used. Here, we try to use the HSS approximation and its diagonal to
construct the preconditioners.  The
convergence steps are collected in Tab.~\ref{tab_convstepm1}.  It can
seen that the convergence is apparently accelerated by the HSS
preconditioning techniques.

\begin{table}
  \centering
  \renewcommand\arraystretch{1.05}
  \begin{tabular}{p{0.3cm}|p{7.2cm}|p{1.3cm}|p{1.3cm}|p{1.3cm}|
    p{1.3cm}}
    \hline\hline
    $m$ & $h$ & $1/5$ & $1/10$ & $1/20$ & $1/40$ \\
    \hline 
    \multirow{3}{*}{$1$} & preconditioner from HSS approximation & 8 &
    12 & 19 & 36 \\
    \cline{2-6}
    &  diagonal preconditioner & 513 & 1003 & $> 2103$ & $>$ 3000 \\
    \hline 
    \multirow{3}{*}{$2$} & preconditioner from HSS approximation & 13
    & 21 & 39 & 70 \\
    \cline{2-6}
    &  diagonal preconditioner & 2273 & $>$ 3000 & $>$ 3000 & $>$
    3000 \\
    \hline 
    \multirow{2}{*}{$3$} & preconditioner from HSS approximation & 21
    & 32 & 56 & 91 \\
    \cline{2-6}
    &  diagonal preconditioner & $>$ 3000 & $>$ 3000 & $>$ 3000 & $>$
    3000 \\
    \hline\hline
  \end{tabular}
  \caption{Convergence steps for the linear system of the $L^2$ norm
  least squares finite element method.}
  \label{tab_convstepm1}
\end{table}

For the method with the discrete minus norm, we shall solve the linear
system $A_{{\wt{\mJ}_h}} \bmr{x} = \bmr{b}$. As stated in the end of
Subsection \ref{subsec_mnorm}, 
we construct the HSS approximation $H_{L^2}$ from $A_{L^2}$ to obtain
the preconditioner, and we also directly construct the HSS
approximation $H$ from $A$ to obtain the preconditioner as a
comparison.
The convergence steps are collected in
Tab.~\ref{tab_convstepm2}.
The convergence speeds of both preconditioned GMRES methods are
significantly faster than the standard GMRES method. The convergence
histories for the iterative methods on the mesh $h=1/20$ are depicted
in Fig.~\ref{fig_convhis}. The numerical performances of both
preconditioned methods are very close, and it is more convenient to
construct the HSS approximation from $A_{L^2}$ in a standard procedure. 
In every step of Krylov iteration, we are required to solve the linear
system $B\bmr{y} = \bmr{z}$. We still use GMRES with the
preconditioner from the HSS approximation as the solver for this
system.
The average iterative steps for different meshes are collected in
Tab.~\ref{tab_convstepm2}. The iterative method also has a fast
convergence speed.  In the rest examples, we use the factors of
$H_{L^2}$ as the preconditioner to solve the linear system. To develop
a more appropriate method to solve the linear system are now left in
the future study.

\begin{table}
  \centering
  \renewcommand\arraystretch{1.15}
  \scalebox{0.95}{
  \begin{tabular}{p{0.3cm}|p{7.5cm}|p{1.3cm}|p{1.3cm}|p{1.3cm}|
    p{1.3cm}}
    \hline\hline
    $m$ & $h$ & $1/5$ & $1/10$ & $1/20$ & $1/40$ \\
    \hline 
    \multirow{4}{*}{$1$} & preconditioner from HSS approximation $H$ &
    9 & 13 & 21 & 38
    \\
    \cline{2-6}
    & preconditioner from HSS approximation
    $H_{L^2}$ & 10 & 16 & 23 & 42 \\
    \cline{2-6}
    &  identical preconditioner $I$ & 1738 & $>3000$ & $>3000$ & $>3000$ \\
    \hline
    \multirow{4}{*}{$2$} & preconditioner from HSS approximation $H$ 
    & 15 & 22 & 38 & 70
    \\
    \cline{2-6}
    & preconditioner from HSS approximation
    $H_{L^2}$ & 16 & 23 & 41 & 78 \\
    \cline{2-6}
    &  identical preconditioner $I$ & $>3000$ & $>3000$ & $>3000$ &
    $>3000$ \\
    \hline
    \multirow{4}{*}{$3$} & preconditioner from HSS approximation $H$ 
    & 23 & 33 & 58 & 97 \\
    \cline{2-6}
    & preconditioner from HSS approximation
    $H_{L^2}$ & 25 & 37 & 63 & 108 \\
    \cline{2-6}
    &  identical preconditioner $I$ & $>3000$ & $>3000$ & $>3000$ &
    $>3000$ \\
    \hline 
    $B$ & preconditioner from HSS approximation & 4 & 7 & 13 & 21  \\
    \hline\hline
  \end{tabular}}
  \caption{Convergence steps for the linear system of the least
  squares finite element method with the discrete minus norm.}
  \label{tab_convstepm2}
\end{table}

\begin{figure}[htp]
  \centering
  \includegraphics[width=0.3\textwidth]{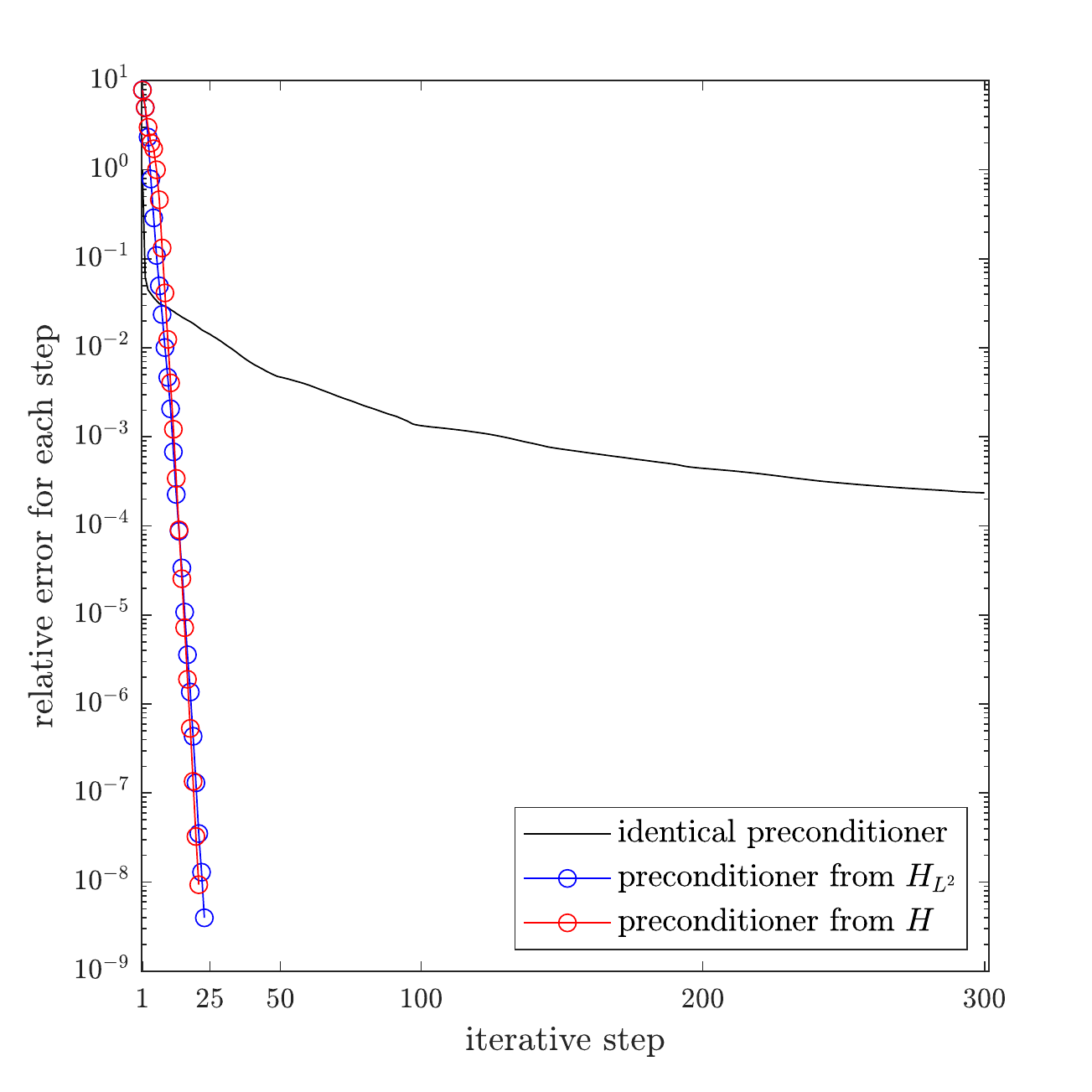}
  \hspace{10pt}
  \includegraphics[width=0.3\textwidth]{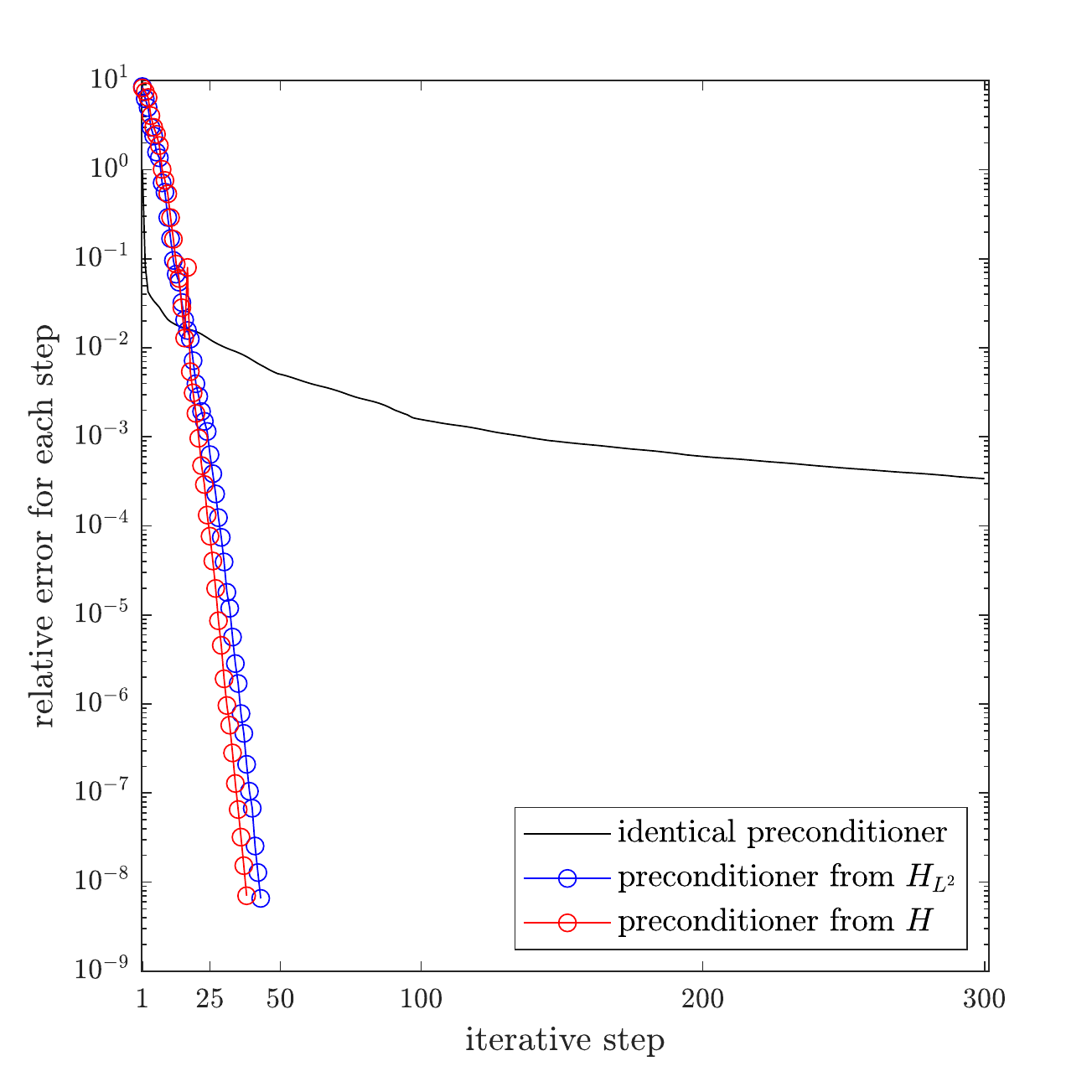}
  \hspace{10pt}
  \includegraphics[width=0.3\textwidth]{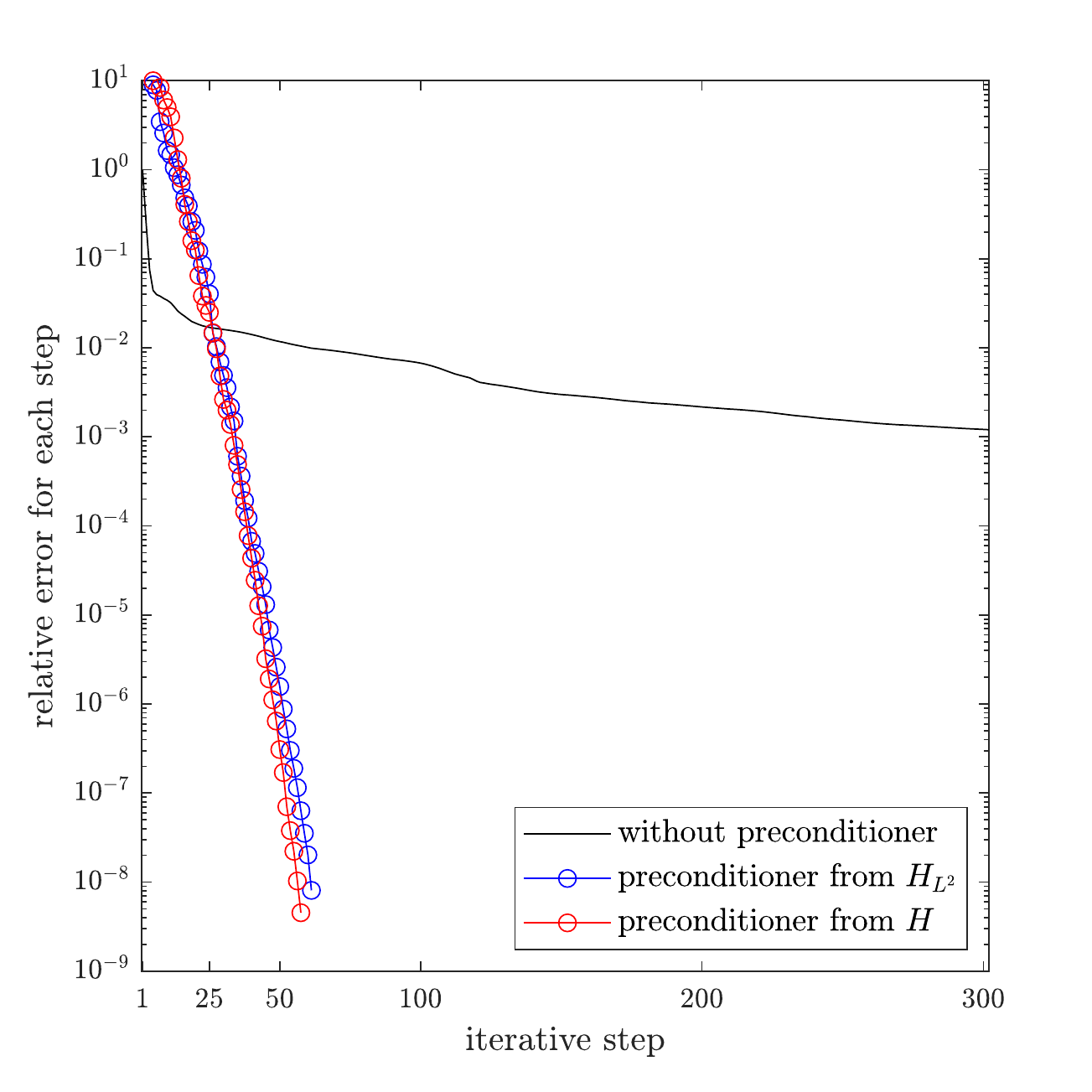}
  \caption{The convergence histories for iterative methods with the
  accuracy $m =1$ (left) / $m=2$ (mid) / $m=3$ (right).}
  \label{fig_convhis}
\end{figure}

\paragraph{\textbf{Example 2.}}
In this example, we consider the same interface problem as Example 1
but the Lam\'e parameters have a large jump across the interface. 
We solve this problem to
demonstrate the robustness of the proposed method when $\lambda
\rightarrow \infty$. 
The parameters are selected as $\lambda_0 = 10000/100, \lambda_1 = 1,
\mu_0 = \mu_1 = 1$. 
The numerical results for both methods are displayed in
Tab.~\ref{tab_example2m1} and Tab.~\ref{tab_example2m2}, respectively. 
We observe that the discrete solutions from both methods
converge uniformly as $\lambda \rightarrow \infty$. The results
illustrates the robustness when the parameter $\lambda \rightarrow
\infty$.

\begin{table}
  \centering
  \renewcommand\arraystretch{1.15}
  \scalebox{0.825}{
  \begin{tabular}{p{0.3cm} | p{3.3cm} | p{1.39cm} | p{1.39cm}  | p{1.39cm} 
    | p{1.39cm}  | p{1.39cm} |  p{1.39cm}  | p{1.39cm} | p{1.39cm} | p{0.9cm} }
    \hline\hline
    $m$ & $h$ & \multicolumn{2}{c|}{1/10} & \multicolumn{2}{c|}{1/10} &
    \multicolumn{2}{c|}{1/10} & \multicolumn{2}{c|}{1/10} & order \\
    \hline
    & $\lambda$  & $100$ & $10000$ & $100$ & $10000$& $100$ & $10000$&
    $100$ & $10000$ & \\
    \hline
    \multirow{3}{*}{1}  &  $\| \bm{u}-\bm{u}_h \|_{L^2(\Omega_0 \cup
    \Omega_1)}$  & 
    6.686e-1 & 6.690e-1 & 1.840e-1 & 1.842e-1 & 4.709e-2 & 4.713e-2 &
    1.180e-2 & 1.181e-2 & 2.00 \\
    \cline{2-11}
    &$\| \bsigma - \bsigma_h \|_{L^2(\Omega_0
    \cup \Omega_1)}$ & 
    7.877e-0 & 7.883e-0 & 3.565e-0 & 3.569e-0 & 1.751e-0 & 1.753e-0 &
    8.800e-1 & 8.814e-1 & 1.00 \\
    \cline{2-11}
    & $\ehnorm{ (\bsigma - \bsigma_h, \bu - \bu_h)}$ & 
    4.281e1 & 4.282e1 & 2.220e1 & 2.220e1 & 1.127e1 & 1.126e1 &
    5.665e-0 & 5.665e-0 & 0.99 \\
    \hline
    \multirow{3}{*}{2}  &  $\| \bm{u}-\bm{u}_h \|_{L^2(\Omega_0 \cup
    \Omega_1)}$  & 
    3.538e-2 & 3.538e-1 & 4.016e-3 & 4.016e-3 & 4.924e-4 & 4.925e-4 &
    6.149e-5 & 6.150e-5 & 3.00 \\
    \cline{2-11}
    &$\| \bsigma - \bsigma_h \|_{L^2(\Omega_0
    \cup \Omega_1)}$ & 
    1.451e-0 & 1.453e-0 & 3.055e-1 & 3.058e-1 & 7.626e-2 & 7.636e-2 &
    1.932e-2 & 1.933e-2 & 2.00 \\
    \cline{2-11}
    & $\ehnorm{ (\bsigma - \bsigma_h, \bu - \bu_h)}$ & 
    7.913e-0 & 7.913e-0 & 2.031e-0 & 2.030e-0 & 5.135e-1 & 5.135e-1 &
    1.289e-1 & 1.289e-1 & 2.00 \\
    \hline
    \multirow{3}{*}{3}  &  $\| \bm{u}-\bm{u}_h \|_{L^2(\Omega_0 \cup
    \Omega_1)}$  & 
    3.749e-3 & 3.750e-3 & 1.651e-4 & 1.651e-4 & 8.925e-6 & 8.924e-6 &
    5.410e-7 & 5.411e-7 & 4.03 \\
    \cline{2-11}
    &$\| \bsigma - \bsigma_h \|_{L^2(\Omega_0
    \cup \Omega_1)}$ & 
    2.169e-1 & 2.172e-1 & 2.222e-2 & 2.223e-2 & 2.419e-3 & 2.422e-3 &
    2.940e-4 & 2.943e-4 & 3.02 \\
    \cline{2-11}
    & $\ehnorm{ (\bsigma - \bsigma_h, \bu - \bu_h)}$ & 
    9.961e-1 & 9.961e-1 & 1.263e-1 & 1.263e-1 & 1.602e-2 & 1.601e-2 &
    2.017e-3 & 2.017e-3 & 2.98 \\
    \hline\hline
  \end{tabular}}
  \caption{Numerical results for Example 2 by the $L^2$ norm least
  squares finite element method with $\lambda = 100, 10000$.}
  \label{tab_example2m1}
\end{table}

\begin{table}
  \centering
  \renewcommand\arraystretch{1.15}
  \scalebox{0.825}{
  \begin{tabular}{p{0.3cm} | p{3.3cm} | p{1.39cm} | p{1.39cm}  | p{1.39cm} 
    | p{1.39cm}  | p{1.39cm} |  p{1.39cm}  | p{1.39cm} | p{1.39cm} | p{0.9cm} }
    \hline\hline
    $m$ & $h$ & \multicolumn{2}{c|}{1/10} & \multicolumn{2}{c|}{1/10} &
    \multicolumn{2}{c|}{1/10} & \multicolumn{2}{c|}{1/10} & order \\
    \hline
    & $\lambda$  & $100$ & $10000$ & $100$ & $10000$& $100$ & $10000$&
    $100$ & $10000$ & \\
    \hline
    \multirow{3}{*}{1}  &  $\| \bm{u}-\bm{u}_h \|_{L^2(\Omega_0 \cup
    \Omega_1)}$  & 
    6.711e-1 & 6.708e-1 & 1.859e-1 & 1.858e-1 & 4.773e-2 & 4.771e-2 &
    1.199e-2 & 1.198e-2 & 1.99 \\
    \cline{2-11}
    &$\| \bsigma - \bsigma_h \|_{L^2(\Omega_0
    \cup \Omega_1)}$ & 
    7.892e-0 & 7.898e-0 & 3.572e-0 & 3.576e-0 & 1.753e-0 & 1.756e-0 &
    8.806e-1 & 8.882e-1 & 0.99 \\
    \cline{2-11}
    & $\ehnorm{ (\bsigma - \bsigma_h, \bu - \bu_h)}$ & 
    4.282e1 & 4.282e1 & 2.220e1 & 2.220e1 & 1.127e1 & 1.127e1 &
    5.666e-0 & 5.665e-0 & 0.99 \\
    \hline
    \multirow{3}{*}{2}  &  $\| \bm{u}-\bm{u}_h \|_{L^2(\Omega_0 \cup
    \Omega_1)}$  & 
    3.547e-2 & 3.548e-1 & 4.017e-3 & 4.017e-3 & 4.925e-4 & 4.926e-4 &
    6.150e-5 & 6.151e-5 & 3.00 \\
    \cline{2-11}
    &$\| \bsigma - \bsigma_h \|_{L^2(\Omega_0
    \cup \Omega_1)}$ & 
    1.449e-0 & 1.451e-0 & 3.053e-1 & 3.056e-1 & 7.627e-2 & 7.637e-2 &
    1.932e-2 & 1.934e-2 & 1.98 \\
    \cline{2-11}
    & $\ehnorm{ (\bsigma - \bsigma_h, \bu - \bu_h)}$ & 
    7.913e-0 & 7.913e-0 & 2.031e-0 & 2.030e-0 & 5.133e-1 & 5.133e-1 &
    1.289e-1 & 1.289e-1 & 1.99 \\
    \hline
    \multirow{3}{*}{3}  &  $\| \bm{u}-\bm{u}_h \|_{L^2(\Omega_0 \cup
    \Omega_1)}$  & 
    3.743e-3 & 3.743e-3 & 1.649e-4 & 1.648e-4 & 8.920e-6 & 8.920e-6 &
    5.409e-7 & 5.408e-7 & 4.05 \\
    \cline{2-11}
    &$\| \bsigma - \bsigma_h \|_{L^2(\Omega_0
    \cup \Omega_1)}$ & 
    2.153e-1 & 2.155e-1 & 2.217e-2 & 2.220e-2 & 2.418e-3 & 2.420e-3 &
    2.939e-4 & 2.942e-4 & 3.03 \\
    \cline{2-11}
    & $\ehnorm{ (\bsigma - \bsigma_h, \bu - \bu_h)}$ & 
    9.967e-1 & 9.968e-1 & 1.263e-1 & 1.264e-1 & 1.602e-2 & 1.602e-2 &
    2.017e-3 & 2.018e-3 & 2.99 \\
    \hline\hline
  \end{tabular}}
  \caption{Numerical results for Example 2 by the least squares finite
  element method of the discrete minus norm with $\lambda = 100,
  10000$.}
  \label{tab_example2m2}
\end{table}

\paragraph{\textbf{Example 3.}}
In this example, we consider a elasticity interface problem with a
star-shaped interface consisting of both concave and convex curve
segments, see Fig.~\ref{fig_2ddomaininterface}. The interface $\Gamma$
is governed by the polar angle $\theta$ that 
\begin{displaymath}
  r = \frac{1}{2} + \frac{\sin 5\theta}{7}.
\end{displaymath}
The analytic solution $\bm{u}$ is taken in a piecewise manner as
\begin{displaymath}
  \bm{u}(x, y) = \begin{cases}
    \bm{u}^0(x, y) & \text{in } \Omega_0, \\
    \bm{u}^1(x, y), & \text{in } \Omega_1, \\
  \end{cases}
\end{displaymath}
where
\begin{displaymath}
  \bu^0(x, y) = [
    \cos(\pi x) \cos(\pi y),\cos(\pi y)]^T, 
    \quad \bu^1(x, y) =   [   \sin(\pi x) \sin(\pi y),  x(1 - x)\sin(\pi
    y)]^T.
\end{displaymath}
The parameters are chosen as $\lambda_0 = \lambda_1 = \mu_0 = \mu_1 =
1$.  We present the numerical results in Tab.~\ref{tab_example3}. The
predicted convergence rates under the energy norms and the $L^2$ norms
for both unknowns are verified from the convergence histories.

\begin{table}
  \centering
  \renewcommand\arraystretch{1.05}
  \scalebox{0.630}{
  \begin{tabular}{p{0.3cm} | p{3.3cm} | p{1.5cm} | p{1.5cm} 
    | p{1.5cm} |  p{1.5cm} | p{0.9cm} }
    \hline\hline
    $m$ & $h$ & 1/5 & 1/10 & 1/20 & 1/40 & order \\
    \hline
    \multirow{3}{*}{$1$} & $\| \bm{u}-\bm{u}_h \|_{L^2(\Omega_0 \cup
    \Omega_1)}$
    & 9.131e-1 & 2.431e-2 &  6.244e-3 & 1.575e-3 & 1.98 \\
    \cline{2-7}
    & $\| \bsigma - \bsigma_h \|_{L^2(\Omega_0 \cup \Omega_1)}$
    & 1.127e-0 & 5.245e-1 &  2.621e-1 & 1.326e-1 & 0.98 \\
    \cline{2-7}
    & $\enorm{ (\bsigma - \bsigma_h, \bu - \bu_h)}$
    & 9.270e-0 & 4.722e-0 &  2.377e-0 & 1.193e-0 & 1.00 \\
    \hline
    \multirow{3}{*}{$2$} & $\| \bm{u}-\bm{u}_h \|_{L^2(\Omega_0 \cup
    \Omega_1)}$
    & 3.047e-3 & 3.513e-4 &  4.412e-5 & 5.551e-6 & 2.99 \\
    \cline{2-7}
    & $\| \bsigma - \bsigma_h \|_{L^2(\Omega_0 \cup \Omega_1)}$
    & 1.193e-1 & 2.432e-2 &  5.946e-3 & 1.503e-6 & 1.98 \\
    \cline{2-7}
    & $\enorm{ (\bsigma - \bsigma_h, \bu - \bu_h)}$
    & 8.171e-1 & 2.079e-1 &  5.239e-2 & 1.315e-2 & 2.00 \\
    \hline
    \multirow{3}{*}{$3$} & $\| \bm{u}-\bm{u}_h \|_{L^2(\Omega_0 \cup
    \Omega_1)}$
    & 1.595e-4 & 6.891e-6 &  3.799e-7 & 2.315e-8 & 4.03 \\
    \cline{2-7}
    & $\| \bsigma - \bsigma_h \|_{L^2(\Omega_0 \cup \Omega_1)}$
    & 8.609e-3 & 8.586e-4 &  9.233e-5 & 1.135e-5 & 3.02 \\
    \cline{2-7}
    & $\enorm{ (\bsigma - \bsigma_h, \bu - \bu_h)}$
    & 5.527e-2 & 6.517e-3 &  8.173e-4 & 1.027e-4 & 2.99 \\
    \hline\hline
  \end{tabular} \hspace{5pt} \begin{tabular}{p{0.3cm} | p{3.3cm} | p{1.5cm} | p{1.5cm} 
    | p{1.5cm} |  p{1.5cm} | p{0.9cm} }
    \hline\hline
    $m$ & $h$ & 1/5 & 1/10 & 1/20 & 1/40 & order \\
    \hline
    \multirow{3}{*}{$1$} & $\| \bm{u}-\bm{u}_h \|_{L^2(\Omega_0 \cup
    \Omega_1)}$
    & 9.127e-2 & 2.430e-2 &  6.237e-3 & 1.573e-3 & 1.99 \\
    \cline{2-7}
    & $\| \bsigma - \bsigma_h \|_{L^2(\Omega_0 \cup \Omega_1)}$
    & 1.127e-0 & 5.267e-1 &  2.627e-1 & 1.327e-1 & 0.99 \\
    \cline{2-7}
    & $\enorm{ (\bsigma - \bsigma_h, \bu - \bu_h)}$
    & 9.270e-0 & 4.722e-0 &  2.377e-0 & 1.193e-0 & 1.00 \\
    \hline
    \multirow{3}{*}{$2$} & $\| \bm{u}-\bm{u}_h \|_{L^2(\Omega_0 \cup
    \Omega_1)}$
    & 3.441e-3 & 3.513e-3 &  4.412e-5 & 5.551e-6 & 2.99 \\
    \cline{2-7}
    & $\| \bsigma - \bsigma_h \|_{L^2(\Omega_0 \cup \Omega_1)}$
    & 1.182e-1 & 2.430e-2 &  5.953e-3 & 1.503e-3 & 1.98 \\
    \cline{2-7}
    & $\enorm{ (\bsigma - \bsigma_h, \bu - \bu_h)}$
    & 8.170e-1 & 2.079e-1 &  5.239e-2 & 1.315e-2 & 2.00 \\
    \hline
    \multirow{3}{*}{$3$} & $\| \bm{u}-\bm{u}_h \|_{L^2(\Omega_0 \cup
    \Omega_1)}$
    & 1.574e-4 & 6.863e-6 &  3.797e-7 & 2.315e-8 & 4.03 \\
    \cline{2-7}
    & $\| \bsigma - \bsigma_h \|_{L^2(\Omega_0 \cup \Omega_1)}$
    & 7.772e-3 & 8.269e-4 &  1.003e-4 & 1.250e-5 & 3.00 \\
    \cline{2-7}
    & $\enorm{ (\bsigma - \bsigma_h, \bu - \bu_h)}$
    & 5.553e-2 & 6.517e-3 &  8.173e-4 & 1.026e-4 & 2.99 \\
    \hline\hline
  \end{tabular}
  }
  \caption{Numerical results for Example 3 by the $L^2$ norm least
  squares finite element method (left) / the least squares finite
  element method with the discrete minus norm (right).}
  \label{tab_example3}
\end{table}

\paragraph{\textbf{Example 4.}}
In this test, we consider the interface problem with an L-shaped
polygonal interface $\Gamma$, which is described by the following
vertices, see Fig.~\ref{fig_2ddomaininterface}, 
\begin{displaymath}
  (0, 0), \quad (-0.35, 0.35), \quad (0, 0.7), \quad (0.7, 0), \quad
  (0, -0.7), \quad (-0.35, 0.35). 
\end{displaymath}
The exact solution and the parameters are taken the same as Example 1.
The numerical errors are displayed in Tab.~\ref{tab_example4}. For the
case of the polygonal interface, all numerically detected convergence
speeds still agree with the theoretical results. 

\begin{table}
  \centering
  \renewcommand\arraystretch{1.05}
  \scalebox{0.630}{
  \begin{tabular}{p{0.3cm} | p{3.3cm} | p{1.5cm} | p{1.5cm} 
    | p{1.5cm} |  p{1.5cm} | p{0.9cm} }
    \hline\hline
    $m$ & $h$ & 1/5 & 1/10 & 1/20 & 1/40 & order \\
    \hline
    \multirow{3}{*}{$1$} & $\| \bm{u}-\bm{u}_h \|_{L^2(\Omega_0 \cup
    \Omega_1)}$
    & 8.082e-1 & 2.383e-1 &  6.356e-2 & 1.609e-2 & 1.98 \\
    \cline{2-7}
    & $\| \bsigma - \bsigma_h \|_{L^2(\Omega_0 \cup \Omega_1)}$
    & 8.601e-0 & 3.583e-0 &  1.686e-0 & 8.412e-1 & 1.00 \\
    \cline{2-7}
    & $\enorm{ (\bsigma - \bsigma_h, \bu - \bu_h)}$
    & 3.222e1 & 2.065e1 &  1.125e1 & 5.810e-0 & 0.96 \\
    \hline
    \multirow{3}{*}{$2$} & $\| \bm{u}-\bm{u}_h \|_{L^2(\Omega_0 \cup
    \Omega_1)}$
    & 8.849e-2 & 1.015e-2 &  1.205e-3 & 1.476e-4 & 3.03 \\
    \cline{2-7}
    & $\| \bsigma - \bsigma_h \|_{L^2(\Omega_0 \cup \Omega_1)}$
    & 2.157e-0 & 4.519e-1 &  1.147e-1 & 2.914e-2 & 1.98 \\
    \cline{2-7}
    & $\enorm{ (\bsigma - \bsigma_h, \bu - \bu_h)}$
    & 8.596e-0 & 2.652e-0 &  7.256e-1 & 1.875e-1 & 1.96 \\
    \hline
    \multirow{3}{*}{$3$} & $\| \bm{u}-\bm{u}_h \|_{L^2(\Omega_0 \cup
    \Omega_1)}$
    & 1.453e-2 & 5.489e-4 &  3.156e-5 & 1.907e-6 & 4.05 \\
    \cline{2-7}
    & $\| \bsigma - \bsigma_h \|_{L^2(\Omega_0 \cup \Omega_1)}$
    & 5.333e-1 & 4.399e-2 &  5.080e-3 & 6.361e-4 & 3.00 \\
    \cline{2-7}
    & $\enorm{ (\bsigma - \bsigma_h, \bu - \bu_h)}$
    & 1.499e-0 & 2.609e-1 &  3.626e-2 & 4.699e-3 & 2.95 \\
    \hline\hline
  \end{tabular} \hspace{5pt} \begin{tabular}{p{0.3cm} | p{3.3cm} | p{1.5cm} | p{1.5cm} 
    | p{1.5cm} |  p{1.5cm} | p{0.9cm} }
    \hline\hline
    $m$ & $h$ & 1/5 & 1/10 & 1/20 & 1/40 & order \\
    \hline
    \multirow{3}{*}{$1$} & $\| \bm{u}-\bm{u}_h \|_{L^2(\Omega_0 \cup
    \Omega_1)}$
    & 8.093e-1 & 2.416e-1 &  6.583e-2 & 1.651e-2 & 1.99 \\
    \cline{2-7}
    & $\| \bsigma - \bsigma_h \|_{L^2(\Omega_0 \cup \Omega_1)}$
    & 8.471e-0 & 3.593e-0 &  1.699e-0 & 8.456e-1 & 1.00 \\
    \cline{2-7}
    & $\enorm{ (\bsigma - \bsigma_h, \bu - \bu_h)}$
    & 3.221e1 & 2.065e1 &  1.125e1 & 5.811e-0 & 0.95 \\
    \hline
    \multirow{3}{*}{$2$} & $\| \bm{u}-\bm{u}_h \|_{L^2(\Omega_0 \cup
    \Omega_1)}$
    & 8.899e-2 & 1.013e-2 &  1.208e-3 & 1.477e-4 & 3.03 \\
    \cline{2-7}
    & $\| \bsigma - \bsigma_h \|_{L^2(\Omega_0 \cup \Omega_1)}$
    & 2.127e-0 & 4.525e-1 &  1.153e-1 & 2.922e-2 & 1.98 \\
    \cline{2-7}
    & $\enorm{ (\bsigma - \bsigma_h, \bu - \bu_h)}$
    & 8.593e-0 & 2.658e-0 &  7.255e-1 & 1.875e-1 & 1.95 \\
    \hline
    \multirow{3}{*}{$3$} & $\| \bm{u}-\bm{u}_h \|_{L^2(\Omega_0 \cup
    \Omega_1)}$
    & 1.453e-2 & 5.486e-4 &  3.157e-5 & 1.908e-6 & 4.03 \\
    \cline{2-7}
    & $\| \bsigma - \bsigma_h \|_{L^2(\Omega_0 \cup \Omega_1)}$
    & 5.272e-1 & 4.375e-2 &  5.087e-3 & 6.416e-4 & 2.99 \\
    \cline{2-7}
    & $\enorm{ (\bsigma - \bsigma_h, \bu - \bu_h)}$
    & 1.499e-0 & 2.609e-1 &  3.625e-2 & 4.699e-3 & 2.95 \\
    \hline\hline
  \end{tabular}
  }
  \caption{Numerical results for Example 3 by the $L^2$ norm least
  squares finite element method (left) / the least squares finite
  element method with the discrete minus norm (right).}
  \label{tab_example4}
\end{table}

\paragraph{\textbf{Example 5.}} 
In this test, we still consider the L-shaped interface and we 
investigate the performance of the method dealing with the problem of
a singular solution. 
The exact solution is selected to be 
\begin{displaymath}
  \bm{u}(x, y) = \begin{cases}
    \bm{u}^0(x, y), & \text{in } \Omega_0, \\
    [1, 1]^T , & \text{in } \Omega_1, \\
  \end{cases}
\end{displaymath}
with the parameters $\lambda_0 = \lambda_1 = \mu_0 = \mu_1 = 1$,
where $\bm{u}^0$ is given as 
\begin{displaymath}
  \begin{aligned}
    u_r^0(r, \theta) &= \frac{r^\alpha}{2\mu} \left( -(\alpha + 1) \cos(
    (\alpha + 1) \theta) + (C_2 - (\alpha + 1)) C_1 \cos( (\alpha - 1)
    \theta) \right),  \\
    u_\theta^0(r, \theta) &= \frac{r^\alpha}{2 \mu} \left( (\alpha + 1)
    \sin( (\alpha + 1) \theta) + (C_2 + \alpha - 1) C_1 \sin( (\alpha
    - 1) \theta)
    \right),
  \end{aligned}
\end{displaymath}
in the polar coordinates $(r, \theta)$. We let 
$\alpha \approx
0.5444837$ be the solution of the following function 
\begin{displaymath}
  \alpha \sin(2 w) + \sin(2 w \alpha) =0,
\end{displaymath}
with $w = 3\pi /4$, and the constants $C_1$ and $C_2$ are
\begin{displaymath}
  C_1 = - \frac{\cos( (\alpha + 1) w)}{\cos( (\alpha - 1) w)}, \quad
  C_2 = \frac{2(\lambda + 2\mu)}{\lambda + \mu}.
\end{displaymath}
The source term $\bm{f} = 0$ on the domain $\Omega_0 \cup \Omega_1$.
The exact solution has the regularity that $(\bsigma,
\bu) \in \bH^{\alpha - \varepsilon}(\div; \Omega_0 \cup \Omega_1)
\times \bH^{1+\alpha - \varepsilon}(\Omega_0 \cup \Omega_1)$ for
$\forall \varepsilon > 0$.Hence, we solve this problem by the least
squares finite element with the discrete minus norm. 
We consider the linear accuracy $m = 1$ for this test, and
the results are shown in Tab.~\ref{tab_example5}. 
It can be observed that the convergence rate for the energy norm is 
$O(h^{0.52})$, which is consistent with the regularity of the exact
solution and the theoretical analysis.
The $L^2$ errors for both variables are less than the optimal
convergence rates, i.e. $O(h^{1.3})$ and $O(h^{0.51})$ for $\bu$ and
$\bsigma$, respectively. The reason may be traced back to the
singularity of the exact solution. 

\begin{table}
  \centering
  \renewcommand\arraystretch{1.10}
  \scalebox{0.825}{
  \begin{tabular}{p{0.3cm} | p{3.3cm} | p{1.5cm} | p{1.5cm} 
    | p{1.5cm} |  p{1.5cm} | p{1.5cm} |  p{1.5cm} |p{0.9cm} }
    \hline\hline
    $m$ & $h$ & 1/5 & 1/10 & 1/20 & 1/40 & 1/80 & 1/160 & order \\
    \hline
    \multirow{3}{*}{$1$} & $\| \bm{u}-\bm{u}_h \|_{L^2(\Omega_0 \cup
    \Omega_1)}$
    & 4.712e-2 & 2.100e-2 &  9.502e-3 & 3.369e-3 & 1.308e-3 & 5.158e-3
    & 1.33 \\
    \cline{2-9}
    & $\| \bsigma - \bsigma_h \|_{L^2(\Omega_0 \cup \Omega_1)}$
    & 7.779e-1 & 6.119e-1 &  4.950e-1 & 3.612e-1 & 2.506e-1 & 1.761e-1
    & 0.50 \\
    \cline{2-9}
    & $\enorm{ (\bsigma - \bsigma_h, \bu - \bu_h)}$
    & 1.089e-0 & 7.980e-1 &  6.278e-1 & 4.428e-1 & 3.133e-1 & 2.170e-1
    & 0.52 \\
    \hline\hline
  \end{tabular}}
  \caption{Numerical errors for Example 5 by the least squares finite
  element method with discrete minus norm.}
  \label{tab_example5}
\end{table}

\paragraph{\textbf{Example 6.}}
In this test, we consider the interface problem in the cubic domain
$\Omega = (0, 1)^3$. 
We take $\Gamma$ as a spherical interface with the radius $r = 0.35$
centered at the point $(0.5, 0.5, 0.5)$. Let the exact displacement  
$\bm{u}$ be
\begin{displaymath}
  \bm{u}(x, y, z) = \begin{bmatrix}
    2^4 \\ 2^5 \\ 2^6 \\
  \end{bmatrix} x(1 - x) y (1 - y) z (1 - z), \quad \text{in }
  \Omega_0 \cup \Omega_1.
\end{displaymath}
Th parameters $\lambda$, $\mu$ are discontinuous across the interface,
$\lambda_0 = 10, \lambda_1 = 1, \mu_0 = 10, \mu_1 = 1$.  We adopt a
series of tetrahedral meshes with the mesh $h = 1/4, 1/8, 1/16, 1/32$
to solve this problem, see Fig.~\ref{fig_3dinterface}.  The numerical
errors under all error measurements are reported in
Tab.~\ref{tab_example6} for both methods. The numerical results
illustrate the accuracy of the methods in three dimensions. 

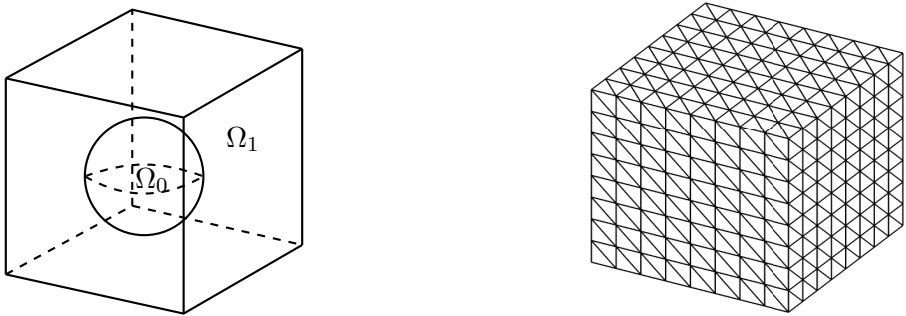
\begin{figure}[htb]
  \centering
  \begin{tikzpicture}[scale=2.6]
    \draw[thick] (-0.9, 0.2) -- (0, 0) -- (0.6, 0.35);
    \draw[thick] (0, 0) -- (0, 1);
    \draw[thick] (0.6, 1.35) -- (0, 1) -- (-0.9, 1.2);
    \draw[thick] (0.6, 1.35) -- (0.6, 0.35);
    \draw[thick] (-0.9, 0.2) -- (-0.9, 1.2);
    \draw[thick] (-0.9, 1.2) -- (-0.26, 1.55) -- (0.6, 1.35);
    \draw[thick, dashed] (-0.9, 0.2) -- (-0.26, 0.55) -- (0.6, 0.35);
    \draw[thick, dashed] (-0.26, 0.55) -- (-0.26, 1.55);
    \draw[thick] (-0.2, 0.7) circle [radius=0.3];
    \draw[thick, dashed] (-0.5, 0.7) [out = -30, in = 210] to (0.1,
    0.7);
    \draw[thick, dashed] (-0.5, 0.7) [out = 20, in = 160] to (0.1,
    0.7);
    \node at (-0.16, 0.68) {$\Omega_0$};
    \node at (0.3, 0.89) {$\Omega_1$};
  \end{tikzpicture}
  \hspace{100pt}
  \includegraphics[width=0.25\textwidth, height=0.25\textwidth]{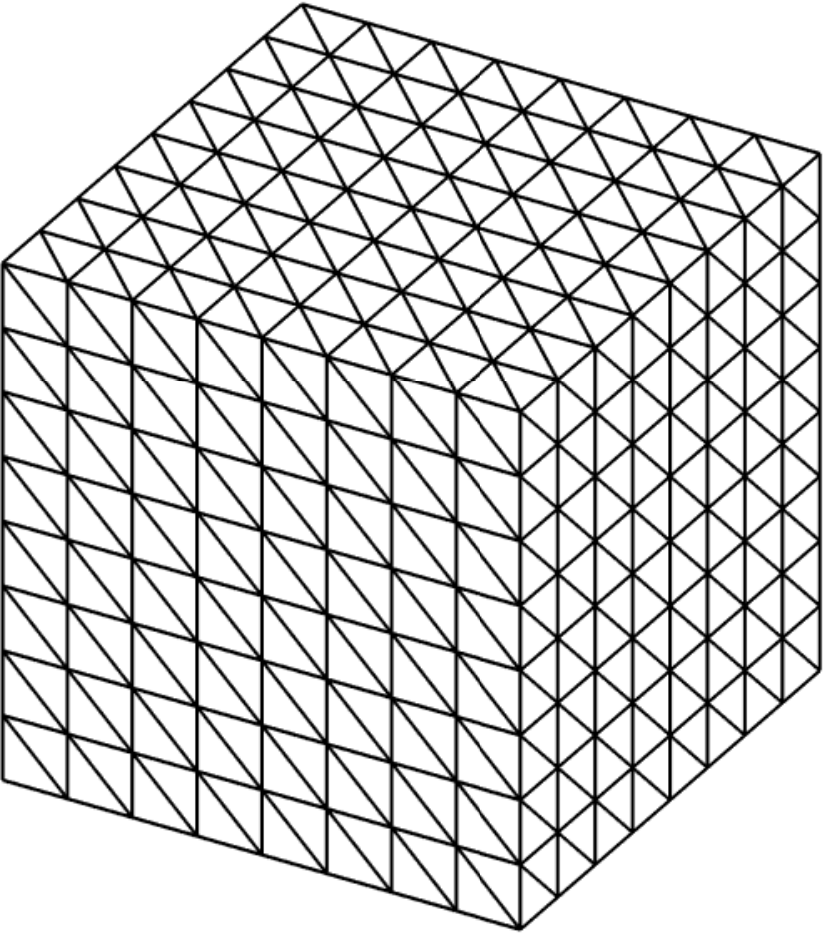}
  \caption{The spherical domain and the tetrahedral mesh of Example 3.}
  \label{fig_3dinterface}
\end{figure}

\begin{table}
  \centering
  \renewcommand\arraystretch{1.10}
  \scalebox{0.625}{
  \begin{tabular}{p{0.3cm} | p{3.3cm} | p{1.5cm} | p{1.5cm} 
    | p{1.5cm} |  p{1.5cm} | p{0.9cm} }
    \hline\hline
    $m$ & $h$ & 1/5 & 1/10 & 1/20 & 1/40 & order \\
    \hline
    \multirow{3}{*}{$1$} & $\| \bm{u}-\bm{u}_h \|_{L^2(\Omega_0 \cup
    \Omega_1)}$
    & 1.483e-1 & 4.786e-3 &  1.307e-2 & 3.335e-4 & 1.98 \\
    \cline{2-7}
    & $\| \bsigma - \bsigma_h \|_{L^2(\Omega_0 \cup \Omega_1)}$
    & 1.556e-1 & 7.528e-1 &  3.921e-1 & 1.948e-1 & 1.01 \\
    \cline{2-7}
    & $\enorm{ (\bsigma - \bsigma_h, \bu - \bu_h)}$
    & 5.089e-0 & 2.636e-0 &  1.353e-0 & 6.859e-1 & 0.98 \\
    \hline
    \multirow{3}{*}{$2$} & $\| \bm{u}-\bm{u}_h \|_{L^2(\Omega_0 \cup
    \Omega_1)}$
    & 1.147e-2 & 9.065e-4 &  1.023e-4 & 1.233e-5 & 3.03 \\
    \cline{2-7}
    & $\| \bsigma - \bsigma_h \|_{L^2(\Omega_0 \cup \Omega_1)}$
    & 2.199e-1 & 5.603e-2 &  1.443e-2 & 3.685e-3 & 1.97 \\
    \cline{2-7}
    & $\enorm{ (\bsigma - \bsigma_h, \bu - \bu_h)}$
    & 7.433e-1 & 1.902e-1 &  4.888e-2 & 1.239e-2 & 1.98 \\
    \hline\hline
  \end{tabular} \hspace{5pt}
  \begin{tabular}{p{0.3cm} | p{3.3cm} | p{1.5cm} | p{1.5cm} 
    | p{1.5cm} |  p{1.5cm} | p{0.9cm} }
    \hline\hline
    $m$ & $h$ & 1/5 & 1/10 & 1/20 & 1/40 & order \\
    \hline
    \multirow{3}{*}{$1$} & $\| \bm{u}-\bm{u}_h \|_{L^2(\Omega_0 \cup
    \Omega_1)}$
    & 1.149e-1 & 4.639e-2 &  1.291e-2 & 3.334e-3 & 1.95 \\
    \cline{2-7}
    & $\| \bsigma - \bsigma_h \|_{L^2(\Omega_0 \cup \Omega_1)}$
    & 1.563e-0 & 7.551e-1 &  3.925e-1 & 1.960e-1 & 0.99 \\
    \cline{2-7}
    & $\enorm{ (\bsigma - \bsigma_h, \bu - \bu_h)}$
    & 5.089e-0 & 2.637e-0 &  1.353e-0 & 6.857e-1 & 0.98 \\
    \hline
    \multirow{3}{*}{$2$} & $\| \bm{u}-\bm{u}_h \|_{L^2(\Omega_0 \cup
    \Omega_1)}$
    & 1.153e-2 & 9.058e-4 &  1.022e-4 & 1.239e-5 & 3.04 \\
    \cline{2-7}
    & $\| \bsigma - \bsigma_h \|_{L^2(\Omega_0 \cup \Omega_1)}$
    & 2.197e-1 & 5.693e-2 &  1.449e-2 & 3.691e-3 & 1.98 \\
    \cline{2-7}
    & $\enorm{ (\bsigma - \bsigma_h, \bu - \bu_h)}$
    & 7.433e-1 & 1.901e-1 &  4.887e-2 & 1.238e-2 & 1.98 \\
    \hline\hline
  \end{tabular}
  }
  \caption{Numerical results for Example 6 by the $L^2$ norm least
  squares finite element method (left) / the least squares finite
  element method with the discrete minus norm (right).}
  \label{tab_example6}
\end{table}

%% file: figure/MThex1.tex
\draw[thick, black] (-1, -0.8) -- (-0.87108404055969, -0.870153453860364);
\draw[thick, black] (-1, -1) -- (-1, -0.8);
\draw[thick, black] (-1, -1) -- (-0.87108404055969, -0.870153453860364);
\draw[thick, black] (-0.8, -1) -- (-0.87108404055969, -0.870153453860364);
\draw[thick, black] (-1, -1) -- (-0.8, -1);
\draw[thick, black] (-1, -0.8) -- (-0.84254628658715, -0.710288625156193);
\draw[thick, black] (-0.84254628658715, -0.710288625156193) -- (-0.87108404055969, -0.870153453860364);
\draw[thick, black] (-0.8, -1) -- (-0.712454356100077, -0.840537277332634);
\draw[thick, black] (-0.712454356100077, -0.840537277332634) -- (-0.87108404055969, -0.870153453860364);
\draw[thick, black] (-0.712454356100077, -0.840537277332634) -- (-0.84254628658715, -0.710288625156193);
\draw[thick, black] (-1, -0.6) -- (-1, -0.8);
\draw[thick, black] (-1, -0.6) -- (-0.84254628658715, -0.710288625156193);
\draw[thick, black] (-0.6, -1) -- (-0.712454356100077, -0.840537277332634);
\draw[thick, black] (-0.8, -1) -- (-0.6, -1);
\draw[thick, black] (-0.712454356100077, -0.840537277332634) -- (-0.642945059174593, -0.638321188961982);
\draw[thick, black] (-0.84254628658715, -0.710288625156193) -- (-0.642945059174593, -0.638321188961982);
\draw[thick, black] (-1, -0.6) -- (-0.827601620208667, -0.512838341257268);
\draw[thick, black] (-0.827601620208667, -0.512838341257268) -- (-0.84254628658715, -0.710288625156193);
\draw[thick, black] (-0.827601620208667, -0.512838341257268) -- (-0.642945059174593, -0.638321188961982);
\draw[thick, black] (-0.6, -1) -- (-0.516939422809003, -0.824439512744817);
\draw[thick, black] (-0.712454356100077, -0.840537277332634) -- (-0.516939422809003, -0.824439512744817);
\draw[thick, black] (-0.516939422809003, -0.824439512744817) -- (-0.642945059174593, -0.638321188961982);
\draw[thick, black] (-1, -0.4) -- (-1, -0.6);
\draw[thick, black] (-1, -0.4) -- (-0.827601620208667, -0.512838341257268);
\draw[thick, black] (-0.6, -1) -- (-0.4, -1);
\draw[thick, black] (-0.4, -1) -- (-0.516939422809003, -0.824439512744817);
\draw[thick, black] (-0.642945059174593, -0.638321188961982) -- (-0.652867420143679, -0.420031567141846);
\draw[thick, black] (-0.827601620208667, -0.512838341257268) -- (-0.652867420143679, -0.420031567141846);
\draw[thick, black] (-0.642945059174593, -0.638321188961982) -- (-0.431440856669231, -0.650350717570841);
\draw[thick, black] (-0.516939422809003, -0.824439512744817) -- (-0.431440856669231, -0.650350717570841);
\draw[thick, black] (-0.827601620208667, -0.512838341257268) -- (-0.825353709685261, -0.308616918027312);
\draw[thick, black] (-1, -0.4) -- (-0.825353709685261, -0.308616918027312);
\draw[thick, black] (-0.825353709685261, -0.308616918027312) -- (-0.652867420143679, -0.420031567141846);
\draw[thick, black] (-0.652867420143679, -0.420031567141846) -- (-0.513249316512507, -0.509440188563727);
\draw[thick, black] (-0.642945059174593, -0.638321188961982) -- (-0.513249316512507, -0.509440188563727);
\draw[thick, black] (-0.431440856669231, -0.650350717570841) -- (-0.513249316512507, -0.509440188563727);
\draw[thick, black] (-0.4, -1) -- (-0.31279950873976, -0.816829858732144);
\draw[thick, black] (-0.31279950873976, -0.816829858732144) -- (-0.516939422809003, -0.824439512744817);
\draw[thick, black] (-0.31279950873976, -0.816829858732144) -- (-0.431440856669231, -0.650350717570841);
\draw[thick, black] (-1, -0.2) -- (-1, -0.4);
\draw[thick, black] (-1, -0.2) -- (-0.825353709685261, -0.308616918027312);
\draw[thick, black] (-0.646090283874711, -0.214428865310949) -- (-0.652867420143679, -0.420031567141846);
\draw[thick, black] (-0.825353709685261, -0.308616918027312) -- (-0.646090283874711, -0.214428865310949);
\draw[thick, black] (-0.4, -1) -- (-0.2, -1);
\draw[thick, black] (-0.2, -1) -- (-0.31279950873976, -0.816829858732144);
\draw[thick, black] (-0.652867420143679, -0.420031567141846) -- (-0.45805785792233, -0.336537892915967);
\draw[thick, black] (-0.45805785792233, -0.336537892915967) -- (-0.513249316512507, -0.509440188563727);
\draw[thick, black] (-0.431440856669231, -0.650350717570841) -- (-0.376901999376875, -0.501074596498745);
\draw[thick, black] (-0.513249316512507, -0.509440188563727) -- (-0.376901999376875, -0.501074596498745);
\draw[thick, black] (-0.646090283874711, -0.214428865310949) -- (-0.45805785792233, -0.336537892915967);
\draw[thick, black] (-0.431440856669231, -0.650350717570841) -- (-0.221600235617308, -0.610258757248247);
\draw[thick, black] (-0.31279950873976, -0.816829858732144) -- (-0.221600235617308, -0.610258757248247);
\draw[thick, black] (-1, -0.2) -- (-0.823361221932333, -0.104818848783997);
\draw[thick, black] (-0.823361221932333, -0.104818848783997) -- (-0.825353709685261, -0.308616918027312);
\draw[thick, black] (-0.823361221932333, -0.104818848783997) -- (-0.646090283874711, -0.214428865310949);
\draw[thick, black] (-0.45805785792233, -0.336537892915967) -- (-0.376901999376875, -0.501074596498745);
\draw[thick, black] (-0.221600235617308, -0.610258757248247) -- (-0.376901999376875, -0.501074596498745);
\draw[thick, black] (-0.31279950873976, -0.816829858732144) -- (-0.10427539686274, -0.814234841786736);
\draw[thick, black] (-0.2, -1) -- (-0.10427539686274, -0.814234841786736);
\draw[thick, black] (-0.10427539686274, -0.814234841786736) -- (-0.221600235617308, -0.610258757248247);
\draw[thick, black] (-0.468082825679428, -0.111134376779427) -- (-0.45805785792233, -0.336537892915967);
\draw[thick, black] (-0.646090283874711, -0.214428865310949) -- (-0.468082825679428, -0.111134376779427);
\draw[thick, black] (-1, 0) -- (-0.823361221932333, -0.104818848783997);
\draw[thick, black] (-1, 0) -- (-1, -0.2);
\draw[thick, black] (-0.646090283874711, -0.214428865310949) -- (-0.644984558393036, -0.00574600046772065);
\draw[thick, black] (-0.823361221932333, -0.104818848783997) -- (-0.644984558393036, -0.00574600046772065);
\draw[thick, black] (-0.45805785792233, -0.336537892915967) -- (-0.255623167343293, -0.396607118033676);
\draw[thick, black] (-0.255623167343293, -0.396607118033676) -- (-0.376901999376875, -0.501074596498745);
\draw[thick, black] (-0.468082825679428, -0.111134376779427) -- (-0.644984558393036, -0.00574600046772065);
\draw[thick, black] (-0.221600235617308, -0.610258757248247) -- (-0.255623167343293, -0.396607118033676);
\draw[thick, black] (-0.2, -1) -- (0, -1);
\draw[thick, black] (0, -1) -- (-0.10427539686274, -0.814234841786736);
\draw[thick, black] (-0.287674455523961, -0.201445973325857) -- (-0.45805785792233, -0.336537892915967);
\draw[thick, black] (-0.287674455523961, -0.201445973325857) -- (-0.468082825679428, -0.111134376779427);
\draw[thick, black] (-0.10427539686274, -0.814234841786736) -- (0.00699557157857294, -0.635311360682441);
\draw[thick, black] (0.00699557157857294, -0.635311360682441) -- (-0.221600235617308, -0.610258757248247);
\draw[thick, black] (-0.287674455523961, -0.201445973325857) -- (-0.255623167343293, -0.396607118033676);
\draw[thick, black] (-0.823361221932333, -0.104818848783997) -- (-0.821691244738411, 0.0993086906034427);
\draw[thick, black] (-1, 0) -- (-0.821691244738411, 0.0993086906034427);
\draw[thick, black] (-0.821691244738411, 0.0993086906034427) -- (-0.644984558393036, -0.00574600046772065);
\draw[thick, black] (-0.221600235617308, -0.610258757248247) -- (-0.071635985033069, -0.452616375960779);
\draw[thick, black] (-0.255623167343293, -0.396607118033676) -- (-0.071635985033069, -0.452616375960779);
\draw[thick, black] (-0.468053963833877, 0.0937082012530377) -- (-0.468082825679428, -0.111134376779427);
\draw[thick, black] (-0.468053963833877, 0.0937082012530377) -- (-0.644984558393036, -0.00574600046772065);
\draw[thick, black] (0.00699557157857294, -0.635311360682441) -- (-0.071635985033069, -0.452616375960779);
\draw[thick, black] (0, -1) -- (0.104411954307321, -0.820122630283604);
\draw[thick, black] (0.104411954307321, -0.820122630283604) -- (-0.10427539686274, -0.814234841786736);
\draw[thick, black] (0.104411954307321, -0.820122630283604) -- (0.00699557157857294, -0.635311360682441);
\draw[thick, black] (-0.299773591781154, -0.00214452673516305) -- (-0.287674455523961, -0.201445973325857);
\draw[thick, black] (-0.299773591781154, -0.00214452673516305) -- (-0.468082825679428, -0.111134376779427);
\draw[thick, black] (-1, 0.2) -- (-1, 0);
\draw[thick, black] (-1, 0.2) -- (-0.821691244738411, 0.0993086906034427);
\draw[thick, black] (-0.468053963833877, 0.0937082012530377) -- (-0.299773591781154, -0.00214452673516305);
\draw[thick, black] (-0.11255268190757, -0.273548042372154) -- (-0.255623167343293, -0.396607118033676);
\draw[thick, black] (-0.287674455523961, -0.201445973325857) -- (-0.11255268190757, -0.273548042372154);
\draw[thick, black] (-0.639636413785809, 0.202125270532684) -- (-0.644984558393036, -0.00574600046772065);
\draw[thick, black] (-0.821691244738411, 0.0993086906034427) -- (-0.639636413785809, 0.202125270532684);
\draw[thick, black] (-0.11255268190757, -0.273548042372154) -- (-0.071635985033069, -0.452616375960779);
\draw[thick, black] (-0.639636413785809, 0.202125270532684) -- (-0.468053963833877, 0.0937082012530377);
\draw[thick, black] (0, -1) -- (0.2, -1);
\draw[thick, black] (0.2, -1) -- (0.104411954307321, -0.820122630283604);
\draw[thick, black] (0.00699557157857294, -0.635311360682441) -- (0.122183126920358, -0.465444249302391);
\draw[thick, black] (0.122183126920358, -0.465444249302391) -- (-0.071635985033069, -0.452616375960779);
\draw[thick, black] (-0.127962888012354, -0.0860986290077132) -- (-0.287674455523961, -0.201445973325857);
\draw[thick, black] (-0.299773591781154, -0.00214452673516305) -- (-0.127962888012354, -0.0860986290077132);
\draw[thick, black] (0.00699557157857294, -0.635311360682441) -- (0.217129575967013, -0.643225700885707);
\draw[thick, black] (0.104411954307321, -0.820122630283604) -- (0.217129575967013, -0.643225700885707);
\draw[thick, black] (-0.127962888012354, -0.0860986290077132) -- (-0.11255268190757, -0.273548042372154);
\draw[thick, black] (-1, 0.2) -- (-0.820627885557708, 0.303706205434743);
\draw[thick, black] (-0.820627885557708, 0.303706205434743) -- (-0.821691244738411, 0.0993086906034427);
\draw[thick, black] (-0.820627885557708, 0.303706205434743) -- (-0.639636413785809, 0.202125270532684);
\draw[thick, black] (-0.468053963833877, 0.0937082012530377) -- (-0.312959963000838, 0.167958080325561);
\draw[thick, black] (-0.312959963000838, 0.167958080325561) -- (-0.299773591781154, -0.00214452673516305);
\draw[thick, black] (0.217129575967013, -0.643225700885707) -- (0.122183126920358, -0.465444249302391);
\draw[thick, black] (-0.071635985033069, -0.452616375960779) -- (0.0355544578465519, -0.32785676790684);
\draw[thick, black] (-0.11255268190757, -0.273548042372154) -- (0.0355544578465519, -0.32785676790684);
\draw[thick, black] (-0.440735977882326, 0.310411571384623) -- (-0.468053963833877, 0.0937082012530377);
\draw[thick, black] (-0.639636413785809, 0.202125270532684) -- (-0.440735977882326, 0.310411571384623);
\draw[thick, black] (0.122183126920358, -0.465444249302391) -- (0.0355544578465519, -0.32785676790684);
\draw[thick, black] (0.104411954307321, -0.820122630283604) -- (0.309037648290863, -0.824407788203264);
\draw[thick, black] (0.2, -1) -- (0.309037648290863, -0.824407788203264);
\draw[thick, black] (0.309037648290863, -0.824407788203264) -- (0.217129575967013, -0.643225700885707);
\draw[thick, black] (-0.440735977882326, 0.310411571384623) -- (-0.312959963000838, 0.167958080325561);
\draw[thick, black] (-0.299773591781154, -0.00214452673516305) -- (-0.132073647613519, 0.124591148911417);
\draw[thick, black] (-0.132073647613519, 0.124591148911417) -- (-0.127962888012354, -0.0860986290077132);
\draw[thick, black] (0.0356680805786379, -0.171174044916343) -- (-0.11255268190757, -0.273548042372154);
\draw[thick, black] (-0.127962888012354, -0.0860986290077132) -- (0.0356680805786379, -0.171174044916343);
\draw[thick, black] (-0.312959963000838, 0.167958080325561) -- (-0.132073647613519, 0.124591148911417);
\draw[thick, black] (-1, 0.4) -- (-1, 0.2);
\draw[thick, black] (-1, 0.4) -- (-0.820627885557708, 0.303706205434743);
\draw[thick, black] (0.0356680805786379, -0.171174044916343) -- (0.0355544578465519, -0.32785676790684);
\draw[thick, black] (-0.639636413785809, 0.202125270532684) -- (-0.639977357040965, 0.410398893867143);
\draw[thick, black] (-0.820627885557708, 0.303706205434743) -- (-0.639977357040965, 0.410398893867143);
\draw[thick, black] (0.217129575967013, -0.643225700885707) -- (0.34197907384985, -0.455501376844728);
\draw[thick, black] (0.34197907384985, -0.455501376844728) -- (0.122183126920358, -0.465444249302391);
\draw[thick, black] (-0.639977357040965, 0.410398893867143) -- (-0.440735977882326, 0.310411571384623);
\draw[thick, black] (0.122183126920358, -0.465444249302391) -- (0.206753967667158, -0.270320576966461);
\draw[thick, black] (0.206753967667158, -0.270320576966461) -- (0.0355544578465519, -0.32785676790684);
\draw[thick, black] (0.4, -1) -- (0.309037648290863, -0.824407788203264);
\draw[thick, black] (0.2, -1) -- (0.4, -1);
\draw[thick, black] (-0.127962888012354, -0.0860986290077132) -- (0.0297191795586318, 0.00988209299187094);
\draw[thick, black] (-0.132073647613519, 0.124591148911417) -- (0.0297191795586318, 0.00988209299187094);
\draw[thick, black] (0.0297191795586318, 0.00988209299187094) -- (0.0356680805786379, -0.171174044916343);
\draw[thick, black] (0.309037648290863, -0.824407788203264) -- (0.421624917184731, -0.652130159039502);
\draw[thick, black] (0.421624917184731, -0.652130159039502) -- (0.217129575967013, -0.643225700885707);
\draw[thick, black] (0.34197907384985, -0.455501376844728) -- (0.206753967667158, -0.270320576966461);
\draw[thick, black] (0.0356680805786379, -0.171174044916343) -- (0.206753967667158, -0.270320576966461);
\draw[thick, black] (-0.312959963000838, 0.167958080325561) -- (-0.224062474564572, 0.312308567758938);
\draw[thick, black] (-0.440735977882326, 0.310411571384623) -- (-0.224062474564572, 0.312308567758938);
\draw[thick, black] (-0.132073647613519, 0.124591148911417) -- (-0.224062474564572, 0.312308567758938);
\draw[thick, black] (0.421624917184731, -0.652130159039502) -- (0.34197907384985, -0.455501376844728);
\draw[thick, black] (-0.820627885557708, 0.303706205434743) -- (-0.821533500916821, 0.509936066797593);
\draw[thick, black] (-1, 0.4) -- (-0.821533500916821, 0.509936066797593);
\draw[thick, black] (-0.821533500916821, 0.509936066797593) -- (-0.639977357040965, 0.410398893867143);
\draw[thick, black] (-0.440735977882326, 0.310411571384623) -- (-0.486527071863865, 0.499533884215876);
\draw[thick, black] (-0.639977357040965, 0.410398893867143) -- (-0.486527071863865, 0.499533884215876);
\draw[thick, black] (0.4, -1) -- (0.513099912226945, -0.827528016466304);
\draw[thick, black] (0.513099912226945, -0.827528016466304) -- (0.309037648290863, -0.824407788203264);
\draw[thick, black] (0.513099912226945, -0.827528016466304) -- (0.421624917184731, -0.652130159039502);
\draw[thick, black] (0.206753967667158, -0.270320576966461) -- (0.183127291619893, -0.0737147096339331);
\draw[thick, black] (0.0356680805786379, -0.171174044916343) -- (0.183127291619893, -0.0737147096339331);
\draw[thick, black] (-0.313638730743655, 0.485208343829403) -- (-0.224062474564572, 0.312308567758938);
\draw[thick, black] (-0.440735977882326, 0.310411571384623) -- (-0.313638730743655, 0.485208343829403);
\draw[thick, black] (0.0297191795586318, 0.00988209299187094) -- (0.183127291619893, -0.0737147096339331);
\draw[thick, black] (0.0297191795586318, 0.00988209299187094) -- (0.039896479091982, 0.164799148583102);
\draw[thick, black] (-0.132073647613519, 0.124591148911417) -- (0.039896479091982, 0.164799148583102);
\draw[thick, black] (-0.313638730743655, 0.485208343829403) -- (-0.486527071863865, 0.499533884215876);
\draw[thick, black] (0.34197907384985, -0.455501376844728) -- (0.410136698021313, -0.25779972333264);
\draw[thick, black] (0.206753967667158, -0.270320576966461) -- (0.410136698021313, -0.25779972333264);
\draw[thick, black] (-0.132073647613519, 0.124591148911417) -- (-0.0314829622655164, 0.304842617011469);
\draw[thick, black] (-0.0314829622655164, 0.304842617011469) -- (-0.224062474564572, 0.312308567758938);
\draw[thick, black] (-1, 0.6) -- (-1, 0.4);
\draw[thick, black] (-1, 0.6) -- (-0.821533500916821, 0.509936066797593);
\draw[thick, black] (0.421624917184731, -0.652130159039502) -- (0.512115883915378, -0.5132352837013);
\draw[thick, black] (0.34197907384985, -0.455501376844728) -- (0.512115883915378, -0.5132352837013);
\draw[thick, black] (-0.0314829622655164, 0.304842617011469) -- (0.039896479091982, 0.164799148583102);
\draw[thick, black] (-0.821533500916821, 0.509936066797593) -- (-0.630088462875087, 0.635886929471564);
\draw[thick, black] (-0.630088462875087, 0.635886929471564) -- (-0.639977357040965, 0.410398893867143);
\draw[thick, black] (-0.630088462875087, 0.635886929471564) -- (-0.486527071863865, 0.499533884215876);
\draw[thick, black] (0.4, -1) -- (0.6, -1);
\draw[thick, black] (0.6, -1) -- (0.513099912226945, -0.827528016466304);
\draw[thick, black] (0.206753967667158, -0.270320576966461) -- (0.319504675067028, -0.132594656731767);
\draw[thick, black] (0.183127291619893, -0.0737147096339331) -- (0.319504675067028, -0.132594656731767);
\draw[thick, black] (0.421624917184731, -0.652130159039502) -- (0.639126582242843, -0.643220825991396);
\draw[thick, black] (0.513099912226945, -0.827528016466304) -- (0.639126582242843, -0.643220825991396);
\draw[thick, black] (0.410136698021313, -0.25779972333264) -- (0.319504675067028, -0.132594656731767);
\draw[thick, black] (0.183127291619893, -0.0737147096339331) -- (0.177111050280417, 0.10264690612111);
\draw[thick, black] (0.0297191795586318, 0.00988209299187094) -- (0.177111050280417, 0.10264690612111);
\draw[thick, black] (0.34197907384985, -0.455501376844728) -- (0.50671992794786, -0.377719891911633);
\draw[thick, black] (0.410136698021313, -0.25779972333264) -- (0.50671992794786, -0.377719891911633);
\draw[thick, black] (0.177111050280417, 0.10264690612111) -- (0.039896479091982, 0.164799148583102);
\draw[thick, black] (0.512115883915378, -0.5132352837013) -- (0.50671992794786, -0.377719891911633);
\draw[thick, black] (-0.115578715448933, 0.478653652848604) -- (-0.224062474564572, 0.312308567758938);
\draw[thick, black] (-0.115578715448933, 0.478653652848604) -- (-0.313638730743655, 0.485208343829403);
\draw[thick, black] (0.639126582242843, -0.643220825991396) -- (0.512115883915378, -0.5132352837013);
\draw[thick, black] (-0.115578715448933, 0.478653652848604) -- (-0.0314829622655164, 0.304842617011469);
\draw[thick, black] (-0.408391090974466, 0.654758676255816) -- (-0.313638730743655, 0.485208343829403);
\draw[thick, black] (-0.408391090974466, 0.654758676255816) -- (-0.486527071863865, 0.499533884215876);
\draw[thick, black] (-1, 0.6) -- (-0.838091025250051, 0.709351985738355);
\draw[thick, black] (-0.838091025250051, 0.709351985738355) -- (-0.821533500916821, 0.509936066797593);
\draw[thick, black] (-0.838091025250051, 0.709351985738355) -- (-0.630088462875087, 0.635886929471564);
\draw[thick, black] (-0.630088462875087, 0.635886929471564) -- (-0.408391090974466, 0.654758676255816);
\draw[thick, black] (0.152283483216741, 0.28083119627518) -- (0.039896479091982, 0.164799148583102);
\draw[thick, black] (0.152283483216741, 0.28083119627518) -- (-0.0314829622655164, 0.304842617011469);
\draw[thick, black] (0.6, -1) -- (0.710341307796489, -0.842701077799337);
\draw[thick, black] (0.513099912226945, -0.827528016466304) -- (0.710341307796489, -0.842701077799337);
\draw[thick, black] (0.710341307796489, -0.842701077799337) -- (0.639126582242843, -0.643220825991396);
\draw[thick, black] (0.183127291619893, -0.0737147096339331) -- (0.327345060323161, 0.0229729184114218);
\draw[thick, black] (0.327345060323161, 0.0229729184114218) -- (0.319504675067028, -0.132594656731767);
\draw[thick, black] (0.152283483216741, 0.28083119627518) -- (0.177111050280417, 0.10264690612111);
\draw[thick, black] (0.177111050280417, 0.10264690612111) -- (0.327345060323161, 0.0229729184114218);
\draw[thick, black] (0.468215089022371, -0.0801750188726556) -- (0.410136698021313, -0.25779972333264);
\draw[thick, black] (0.468215089022371, -0.0801750188726556) -- (0.319504675067028, -0.132594656731767);
\draw[thick, black] (-0.208106488807151, 0.653515005779124) -- (-0.115578715448933, 0.478653652848604);
\draw[thick, black] (-0.208106488807151, 0.653515005779124) -- (-0.313638730743655, 0.485208343829403);
\draw[thick, black] (-0.408391090974466, 0.654758676255816) -- (-0.208106488807151, 0.653515005779124);
\draw[thick, black] (0.617344674661071, -0.224402756953881) -- (0.410136698021313, -0.25779972333264);
\draw[thick, black] (0.617344674661071, -0.224402756953881) -- (0.50671992794786, -0.377719891911633);
\draw[thick, black] (0.639126582242843, -0.643220825991396) -- (0.653030321805961, -0.432377250990455);
\draw[thick, black] (0.653030321805961, -0.432377250990455) -- (0.512115883915378, -0.5132352837013);
\draw[thick, black] (0.0861988800995658, 0.464873546343088) -- (-0.0314829622655164, 0.304842617011469);
\draw[thick, black] (0.0861988800995658, 0.464873546343088) -- (-0.115578715448933, 0.478653652848604);
\draw[thick, black] (0.653030321805961, -0.432377250990455) -- (0.50671992794786, -0.377719891911633);
\draw[thick, black] (0.0861988800995658, 0.464873546343088) -- (0.152283483216741, 0.28083119627518);
\draw[thick, black] (0.617344674661071, -0.224402756953881) -- (0.468215089022371, -0.0801750188726556);
\draw[thick, black] (-1, 0.8) -- (-1, 0.6);
\draw[thick, black] (-1, 0.8) -- (-0.838091025250051, 0.709351985738355);
\draw[thick, black] (-0.508384091501224, 0.826361975017828) -- (-0.630088462875087, 0.635886929471564);
\draw[thick, black] (-0.508384091501224, 0.826361975017828) -- (-0.408391090974466, 0.654758676255816);
\draw[thick, black] (0.468215089022371, -0.0801750188726556) -- (0.327345060323161, 0.0229729184114218);
\draw[thick, black] (-0.707625175444442, 0.840272297838361) -- (-0.630088462875087, 0.635886929471564);
\draw[thick, black] (-0.838091025250051, 0.709351985738355) -- (-0.707625175444442, 0.840272297838361);
\draw[thick, black] (0.6, -1) -- (0.8, -1);
\draw[thick, black] (0.8, -1) -- (0.710341307796489, -0.842701077799337);
\draw[thick, black] (0.840650980171442, -0.712668400231478) -- (0.639126582242843, -0.643220825991396);
\draw[thick, black] (0.840650980171442, -0.712668400231478) -- (0.710341307796489, -0.842701077799337);
\draw[thick, black] (0.653030321805961, -0.432377250990455) -- (0.617344674661071, -0.224402756953881);
\draw[thick, black] (0.152283483216741, 0.28083119627518) -- (0.325100784469194, 0.210831465563934);
\draw[thick, black] (0.177111050280417, 0.10264690612111) -- (0.325100784469194, 0.210831465563934);
\draw[thick, black] (-0.508384091501224, 0.826361975017828) -- (-0.707625175444442, 0.840272297838361);
\draw[thick, black] (0.325100784469194, 0.210831465563934) -- (0.327345060323161, 0.0229729184114218);
\draw[thick, black] (0.825148907000718, -0.517370547068503) -- (0.639126582242843, -0.643220825991396);
\draw[thick, black] (0.825148907000718, -0.517370547068503) -- (0.653030321805961, -0.432377250990455);
\draw[thick, black] (-0.00586899580942461, 0.647907781849544) -- (-0.115578715448933, 0.478653652848604);
\draw[thick, black] (-0.00586899580942461, 0.647907781849544) -- (-0.208106488807151, 0.653515005779124);
\draw[thick, black] (-0.304475837202875, 0.826834551885855) -- (-0.408391090974466, 0.654758676255816);
\draw[thick, black] (-0.304475837202875, 0.826834551885855) -- (-0.208106488807151, 0.653515005779124);
\draw[thick, black] (-0.00586899580942461, 0.647907781849544) -- (0.0861988800995658, 0.464873546343088);
\draw[thick, black] (0.840650980171442, -0.712668400231478) -- (0.825148907000718, -0.517370547068503);
\draw[thick, black] (0.468215089022371, -0.0801750188726556) -- (0.486308719377917, 0.112138696367571);
\draw[thick, black] (0.486308719377917, 0.112138696367571) -- (0.327345060323161, 0.0229729184114218);
\draw[thick, black] (-0.508384091501224, 0.826361975017828) -- (-0.304475837202875, 0.826834551885855);
\draw[thick, black] (-0.838091025250051, 0.709351985738355) -- (-0.869153161757602, 0.869926174687143);
\draw[thick, black] (-1, 0.8) -- (-0.869153161757602, 0.869926174687143);
\draw[thick, black] (0.310971850342883, 0.433679372113334) -- (0.0861988800995658, 0.464873546343088);
\draw[thick, black] (0.310971850342883, 0.433679372113334) -- (0.152283483216741, 0.28083119627518);
\draw[thick, black] (-0.707625175444442, 0.840272297838361) -- (-0.869153161757602, 0.869926174687143);
\draw[thick, black] (0.645616838901902, 0.0014450582538924) -- (0.617344674661071, -0.224402756953881);
\draw[thick, black] (0.645616838901902, 0.0014450582538924) -- (0.468215089022371, -0.0801750188726556);
\draw[thick, black] (0.818897175366853, -0.313730874004021) -- (0.653030321805961, -0.432377250990455);
\draw[thick, black] (0.818897175366853, -0.313730874004021) -- (0.617344674661071, -0.224402756953881);
\draw[thick, black] (0.486308719377917, 0.112138696367571) -- (0.325100784469194, 0.210831465563934);
\draw[thick, black] (0.310971850342883, 0.433679372113334) -- (0.325100784469194, 0.210831465563934);
\draw[thick, black] (0.8, -1) -- (0.870124692763664, -0.871186008397934);
\draw[thick, black] (0.710341307796489, -0.842701077799337) -- (0.870124692763664, -0.871186008397934);
\draw[thick, black] (0.840650980171442, -0.712668400231478) -- (0.870124692763664, -0.871186008397934);
\draw[thick, black] (0.645616838901902, 0.0014450582538924) -- (0.486308719377917, 0.112138696367571);
\draw[thick, black] (0.818897175366853, -0.313730874004021) -- (0.825148907000718, -0.517370547068503);
\draw[thick, black] (-0.6, 1) -- (-0.508384091501224, 0.826361975017828);
\draw[thick, black] (-0.6, 1) -- (-0.707625175444442, 0.840272297838361);
\draw[thick, black] (-0.102839096186653, 0.825416350235223) -- (-0.00586899580942461, 0.647907781849544);
\draw[thick, black] (-0.102839096186653, 0.825416350235223) -- (-0.208106488807151, 0.653515005779124);
\draw[thick, black] (-0.102839096186653, 0.825416350235223) -- (-0.304475837202875, 0.826834551885855);
\draw[thick, black] (0.201902881964529, 0.638930855541988) -- (0.0861988800995658, 0.464873546343088);
\draw[thick, black] (-0.00586899580942461, 0.647907781849544) -- (0.201902881964529, 0.638930855541988);
\draw[thick, black] (-0.4, 1) -- (-0.508384091501224, 0.826361975017828);
\draw[thick, black] (-0.4, 1) -- (-0.304475837202875, 0.826834551885855);
\draw[thick, black] (0.310971850342883, 0.433679372113334) -- (0.201902881964529, 0.638930855541988);
\draw[thick, black] (0.817955386108298, -0.106094907791163) -- (0.645616838901902, 0.0014450582538924);
\draw[thick, black] (0.817955386108298, -0.106094907791163) -- (0.617344674661071, -0.224402756953881);
\draw[thick, black] (0.818897175366853, -0.313730874004021) -- (0.817955386108298, -0.106094907791163);
\draw[thick, black] (1, -0.6) -- (0.840650980171442, -0.712668400231478);
\draw[thick, black] (1, -0.6) -- (0.825148907000718, -0.517370547068503);
\draw[thick, black] (-1, 1) -- (-1, 0.8);
\draw[thick, black] (-1, 1) -- (-0.869153161757602, 0.869926174687143);
\draw[thick, black] (-0.4, 1) -- (-0.6, 1);
\draw[thick, black] (-0.8, 1) -- (-0.869153161757602, 0.869926174687143);
\draw[thick, black] (-0.8, 1) -- (-0.707625175444442, 0.840272297838361);
\draw[thick, black] (0.486308719377917, 0.112138696367571) -- (0.490631892247699, 0.310123858543609);
\draw[thick, black] (0.490631892247699, 0.310123858543609) -- (0.325100784469194, 0.210831465563934);
\draw[thick, black] (-0.6, 1) -- (-0.8, 1);
\draw[thick, black] (0.310971850342883, 0.433679372113334) -- (0.490631892247699, 0.310123858543609);
\draw[thick, black] (0.8, -1) -- (1, -1);
\draw[thick, black] (1, -1) -- (0.870124692763664, -0.871186008397934);
\draw[thick, black] (1, -0.8) -- (0.840650980171442, -0.712668400231478);
\draw[thick, black] (1, -0.8) -- (0.870124692763664, -0.871186008397934);
\draw[thick, black] (1, -0.4) -- (0.818897175366853, -0.313730874004021);
\draw[thick, black] (1, -0.4) -- (0.825148907000718, -0.517370547068503);
\draw[thick, black] (0.656138560192133, 0.206641092349089) -- (0.486308719377917, 0.112138696367571);
\draw[thick, black] (0.645616838901902, 0.0014450582538924) -- (0.656138560192133, 0.206641092349089);
\draw[thick, black] (1, -0.8) -- (1, -0.6);
\draw[thick, black] (-0.102839096186653, 0.825416350235223) -- (0.0999190831840636, 0.822583044946834);
\draw[thick, black] (0.0999190831840636, 0.822583044946834) -- (-0.00586899580942461, 0.647907781849544);
\draw[thick, black] (1, -0.6) -- (1, -0.4);
\draw[thick, black] (0.0999190831840636, 0.822583044946834) -- (0.201902881964529, 0.638930855541988);
\draw[thick, black] (-0.8, 1) -- (-1, 1);
\draw[thick, black] (-0.2, 1) -- (-0.102839096186653, 0.825416350235223);
\draw[thick, black] (-0.2, 1) -- (-0.304475837202875, 0.826834551885855);
\draw[thick, black] (-0.2, 1) -- (-0.4, 1);
\draw[thick, black] (0.411993009432665, 0.639321932629435) -- (0.310971850342883, 0.433679372113334);
\draw[thick, black] (0.411993009432665, 0.639321932629435) -- (0.201902881964529, 0.638930855541988);
\draw[thick, black] (0.656138560192133, 0.206641092349089) -- (0.490631892247699, 0.310123858543609);
\draw[thick, black] (0.824911538891892, 0.101168882706749) -- (0.645616838901902, 0.0014450582538924);
\draw[thick, black] (0.824911538891892, 0.101168882706749) -- (0.817955386108298, -0.106094907791163);
\draw[thick, black] (1, -1) -- (1, -0.8);
\draw[thick, black] (1, -0.2) -- (0.818897175366853, -0.313730874004021);
\draw[thick, black] (1, -0.2) -- (0.817955386108298, -0.106094907791163);
\draw[thick, black] (0.502516308636153, 0.48524533609437) -- (0.310971850342883, 0.433679372113334);
\draw[thick, black] (0.502516308636153, 0.48524533609437) -- (0.490631892247699, 0.310123858543609);
\draw[thick, black] (0.824911538891892, 0.101168882706749) -- (0.656138560192133, 0.206641092349089);
\draw[thick, black] (1, -0.4) -- (1, -0.2);
\draw[thick, black] (0.411993009432665, 0.639321932629435) -- (0.502516308636153, 0.48524533609437);
\draw[thick, black] (0, 1) -- (-0.102839096186653, 0.825416350235223);
\draw[thick, black] (0, 1) -- (0.0999190831840636, 0.822583044946834);
\draw[thick, black] (0, 1) -- (-0.2, 1);
\draw[thick, black] (0.304436104385622, 0.820911650262993) -- (0.0999190831840636, 0.822583044946834);
\draw[thick, black] (0.304436104385622, 0.820911650262993) -- (0.201902881964529, 0.638930855541988);
\draw[thick, black] (0.304436104385622, 0.820911650262993) -- (0.411993009432665, 0.639321932629435);
\draw[thick, black] (1, 0) -- (0.824911538891892, 0.101168882706749);
\draw[thick, black] (1, 0) -- (0.817955386108298, -0.106094907791163);
\draw[thick, black] (1, -0.2) -- (1, 0);
\draw[thick, black] (0.656138560192133, 0.206641092349089) -- (0.657687753533015, 0.40830970561707);
\draw[thick, black] (0.657687753533015, 0.40830970561707) -- (0.490631892247699, 0.310123858543609);
\draw[thick, black] (0.502516308636153, 0.48524533609437) -- (0.657687753533015, 0.40830970561707);
\draw[thick, black] (0.828040255972815, 0.304469584606627) -- (0.656138560192133, 0.206641092349089);
\draw[thick, black] (0.824911538891892, 0.101168882706749) -- (0.828040255972815, 0.304469584606627);
\draw[thick, black] (0.2, 1) -- (0, 1);
\draw[thick, black] (0.2, 1) -- (0.0999190831840636, 0.822583044946834);
\draw[thick, black] (0.2, 1) -- (0.304436104385622, 0.820911650262993);
\draw[thick, black] (0.828040255972815, 0.304469584606627) -- (0.657687753533015, 0.40830970561707);
\draw[thick, black] (0.411993009432665, 0.639321932629435) -- (0.63782354467115, 0.630691018457456);
\draw[thick, black] (0.63782354467115, 0.630691018457456) -- (0.502516308636153, 0.48524533609437);
\draw[thick, black] (0.510968891693937, 0.821999613373123) -- (0.411993009432665, 0.639321932629435);
\draw[thick, black] (0.304436104385622, 0.820911650262993) -- (0.510968891693937, 0.821999613373123);
\draw[thick, black] (1, 0) -- (1, 0.2);
\draw[thick, black] (1, 0.2) -- (0.824911538891892, 0.101168882706749);
\draw[thick, black] (0.63782354467115, 0.630691018457456) -- (0.657687753533015, 0.40830970561707);
\draw[thick, black] (1, 0.2) -- (0.828040255972815, 0.304469584606627);
\draw[thick, black] (0.510968891693937, 0.821999613373123) -- (0.63782354467115, 0.630691018457456);
\draw[thick, black] (0.4, 1) -- (0.2, 1);
\draw[thick, black] (0.4, 1) -- (0.304436104385622, 0.820911650262993);
\draw[thick, black] (0.4, 1) -- (0.510968891693937, 0.821999613373123);
\draw[thick, black] (0.827681694587204, 0.508967292257806) -- (0.828040255972815, 0.304469584606627);
\draw[thick, black] (0.827681694587204, 0.508967292257806) -- (0.657687753533015, 0.40830970561707);
\draw[thick, black] (0.827681694587204, 0.508967292257806) -- (0.63782354467115, 0.630691018457456);
\draw[thick, black] (1, 0.4) -- (0.828040255972815, 0.304469584606627);
\draw[thick, black] (1, 0.2) -- (1, 0.4);
\draw[thick, black] (0.710165890302475, 0.838672151014105) -- (0.510968891693937, 0.821999613373123);
\draw[thick, black] (0.710165890302475, 0.838672151014105) -- (0.63782354467115, 0.630691018457456);
\draw[thick, black] (1, 0.4) -- (0.827681694587204, 0.508967292257806);
\draw[thick, black] (0.6, 1) -- (0.4, 1);
\draw[thick, black] (0.6, 1) -- (0.510968891693937, 0.821999613373123);
\draw[thick, black] (0.827681694587204, 0.508967292257806) -- (0.8411255389628, 0.708232614101842);
\draw[thick, black] (0.8411255389628, 0.708232614101842) -- (0.63782354467115, 0.630691018457456);
\draw[thick, black] (0.710165890302475, 0.838672151014105) -- (0.8411255389628, 0.708232614101842);
\draw[thick, black] (0.6, 1) -- (0.710165890302475, 0.838672151014105);
\draw[thick, black] (1, 0.4) -- (1, 0.6);
\draw[thick, black] (1, 0.6) -- (0.827681694587204, 0.508967292257806);
\draw[thick, black] (1, 0.6) -- (0.8411255389628, 0.708232614101842);
\draw[thick, black] (0.8, 1) -- (0.6, 1);
\draw[thick, black] (0.8, 1) -- (0.710165890302475, 0.838672151014105);
\draw[thick, black] (0.8411255389628, 0.708232614101842) -- (0.870306758875961, 0.869502655544755);
\draw[thick, black] (0.710165890302475, 0.838672151014105) -- (0.870306758875961, 0.869502655544755);
\draw[thick, black] (0.8, 1) -- (0.870306758875961, 0.869502655544755);
\draw[thick, black] (1, 0.6) -- (1, 0.8);
\draw[thick, black] (1, 0.8) -- (0.8411255389628, 0.708232614101842);
\draw[thick, black] (1, 0.8) -- (0.870306758875961, 0.869502655544755);
\draw[thick, black] (1, 1) -- (0.8, 1);
\draw[thick, black] (1, 1) -- (0.870306758875961, 0.869502655544755);
\draw[thick, black] (1, 0.8) -- (1, 1);

%% file: conclusion.tex
\section*{Acknowledgements}
This work was supported by National Natural Science Foundation of
China (12201442, 11971041).